\documentclass{amsart}
\usepackage[T1]{fontenc}

\usepackage{amsmath}
\usepackage{amssymb}
\usepackage{mathrsfs}

\usepackage{float}
\usepackage{graphicx,color}
\usepackage{ascmac}

\usepackage{tabularx}

\usepackage{tikz}
\usetikzlibrary{calc,decorations.markings}
\usepackage{amsthm}
\theoremstyle{definition}
\newtheorem{THM}{Theorem}
\newtheorem{LEM}[THM]{Lemma}

\newtheorem{DEF}[THM]{Definition}
\newtheorem{RMK}[THM]{Remark}

\newtheorem*{THM*}{Theorem}
\newtheorem*{LEM*}{Lemma}
\newtheorem*{PROP*}{Proposition}
\newtheorem*{COR*}{Corollary}
\newtheorem*{DEF*}{Definition}
\newtheorem*{RMK*}{Remark}
\newtheorem*{EX*}{Example}

\numberwithin{figure}{section}
\numberwithin{equation}{section}
\numberwithin{THM}{section}

\title{$A_2$ colored polynomials of rigid vertex graphs}
\author{Wataru Yuasa}
\address{Department of Mathematics\\
  Tokyo Institute of Technology\\
  2-12-1 Ookayama, Meguro-ku, Tokyo 152-8551, Japan}
\email[]{yuasa.w.aa@m.titech.ac.jp}

\begin{document}
\begin{abstract}
The Kauffman-Vogel polynomials are three variable polynomial invariants of $4$-valent rigid vertex graphs. 
A one-variable specialization of the Kauffman-Vogel polynomials for unoriented $4$-valent rigid vertex graphs was given by using the Kauffman bracket and the Jones-Wenzl idempotent colored with $2$. 
Bataineh, Elhamdadi and Hajij generalized it to any color with even positive integers. 
We give another generalization of the one-variable Kauffman-Vogel polynomial for oriented and unoriented $4$-valent rigid vertex graphs by using the $A_2$ bracket and the $A_2$ clasps. 
These polynomial invariants are considered as the $\mathfrak{sl}_3$ colored Jones polynomials for singular knots and links. 
\end{abstract}
\maketitle
\tikzset{->-/.style={decoration={
  markings,
  mark=at position #1 with {\arrow[black,thin]{>}}},postaction={decorate}}}
\tikzset{-<-/.style={decoration={
  markings,
  mark=at position #1 with {\arrow[black,thin]{<}}},postaction={decorate}}}
\tikzset{-|-/.style={decoration={
  markings,
  mark=at position #1 with {\arrow[black,thin]{|}}},postaction={decorate}}}
\tikzset{
    triple/.style args={[#1] in [#2] in [#3]}{
        #1,preaction={preaction={draw,#3},draw,#2}
    }
}

\section{Introduction}
Kauffman considered an isotopy for embeddings of graphs in the $3$-space in \cite{Kauffman89}. 
It is called the vertex isotopy and he proved the vertex isotopy for a $4$-valent graph is generated by a generalization of the Reidemeister moves. 
Kauffman and Vogel defined polynomial invariants for regular vertex isotopy classes of $4$-valent graphs in \cite{KauffmanVogel92}.
These polynomial invariants are three-variable generalizations of polynomial invariants of knots: the Homfly polynomials for oriented knots and the Kauffman polynomials~\cite{Kauffman90} for unoriented knots. 
A one-variable specialization of the Kauffman-Vogel polynomial of a regular vertex isotopy class of an unoriented $4$-valent graph was given by using the Kauffman bracket and the Jones-Wenzl idempotents (see, for example, Sect.~4.4 in \cite{KauffmanLins94}). 
For an unoriented $4$-valent graph $G$, 
this polynomial is defined by the value of the Kauffman bracket of a skein element obtained by coloring each edge of $G$ with $2$ and replacing $4$-valent vertices by a certain type of skein element. 
In~\cite{BatainehElhamdadiHajij16}, 
Bataineh, Elhamdadi and Hajij generalize the one-variable Kauffman-Vogel polynomials by changing the coloring from $2$ to any even positive integer $2n$. 
There are many invariants of vertex isotopy classes of graphs, for example, Kauffman and Mishra~\cite{KauffmanMishra13}, 
Juyumaya and Lambropoulou~\cite{JuyumayaLambropoulou09} as an invariant of singular knots,  Yamada~\cite{Yamada89} as an invariant of spatial graphs, and so on. 
In \cite{Wu12}, 
Wu showed a relationship between the Kauffman-Vogel polynomial and the MOY graph polynomial~\cite{MurakamiOhtsukiYamada98}. 
By using linear skein theory, 
Some invariants of topological graphs were constructed, for example, in Yokota~\cite{Yokota96} and Kawagoe~\cite{Kawagoe16}.

In this paper, 
we will define one-variable polynomial invariants of the regular vertex isotopy classes of oriented and unoriented $4$-valent graphs. 
These polynomial invariants are a skein theoretical generalization of the one-variable Kauffman-Vogel polynomials. 
We will construct the invariants by using the $A_2$ bracket and the $A_2$ clasps instead of the Kauffman bracket and the Jones-Wenzl idempotents.

This paper is organized as follows. 
We introduce the definition of a $4$-valent rigid vertex graph by diagrams on $S^2$ and the generalized Reidemeister moves in Sect.~2. 
Next, 
we define the $A_2$ bracket, the $A_2$ clasps and show some useful formulas in Sect.~3. 
In Sect.~4, 
we define the polynomial invariants of oriented and unoriented $4$-valent rigid vertex graphs. 
In Sect.~5, 
we compute these invariants for some $4$-valent rigid vertex graphs.

\section{rigid vertex graphs}
We will treat diagrammatically regular vertex isotopy classes of embeddings of oriented and unoriented $4$-valent graphs in $S^3$ through an equivalence class of $4$-valent graph diagrams on $S^2$. 
We briefly explain the geometric definition of the rigid vertex graphs (see Kauffman~\cite{Kauffman89} for details.) 
The rigid vertex means that the half-edges attaching to the vertex have a cyclic ordering.
An embedding of a $4$-valent rigid vertex graphs into $S^3$ is an embedding of the underlying $4$-valent graph into $S^3$ with the following condition. 
Each embedded vertex $v$ can be replaced by a small disk $D_v$ in $S^3$ and half-edges at $v$ are attached to $\partial D_v$ such that the cyclic ordering coincides with the orientation of $\partial D_v$.

We deal with the regular isotopy classes of the above graphs in $S^3$ as diagrams on $S^2$ with an equivalence relation generated by Reidemeister moves (RI) -- (RV).
\begin{DEF}\ 
\begin{itemize}
\item A $4$-valent graph diagram on $S^2$ is an immersion of $4$-valent graph into $S^2$ whose intersection points are only transverse double points of edges. 
At each intersection point, two edges are equipped with crossing data 
\,\tikz[baseline=-.6ex, scale=.8]{
\draw [thin, dashed, fill=white] (0,0) circle [radius=.5];
\draw[->-=.8] (-45:.5) -- (135:.5);
\draw[->-=.8, white, double=black, double distance=0.4pt, ultra thick] (-135:.5) -- (45:.5);
}\ .
\item Two $4$-valent graph diagrams $G$ and $G'$ are equivalent if $G$ is related to $G'$ by a finite sequence of Reidemeister moves (RI) -- (RV) in Fig.~\ref{Reidemeister}. 
This equivalence relation is called {\em regular vertex isotopy} in \cite{Kauffman89}.
\item An oriented $4$-valent graph diagram is a $4$-valent graph diagram whose edges are oriented as one of the following:
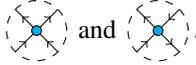
\[
\,\tikz[baseline=-.6ex, scale=.8]{
\draw [thin, dashed, fill=white] (0,0) circle [radius=.5];
\draw[-<-=.5] (45:.5) -- (0:0);
\draw[-<-=.5] (135:.5) -- (0:0);
\draw[->-=.5] (-45:.5) -- (0:0);
\draw[->-=.5] (-135:.5) -- (0:0);
\draw[fill=cyan] (0,0) circle [radius=.08];
}\,\ \text{and}\ 
\,\tikz[baseline=-.6ex, scale=.8]{
\draw [thin, dashed, fill=white] (0,0) circle [radius=.5];
\draw[->-=.5] (45:.5) -- (0:0);
\draw[-<-=.5] (135:.5) -- (0:0);
\draw[-<-=.5] (-45:.5) -- (0:0);
\draw[->-=.5] (-135:.5) -- (0:0);
\draw[fill=cyan] (0,0) circle [radius=.08];
}\, .
\]
\end{itemize}
\end{DEF}

We call equivalence classes of $4$-valent graph diagrams {\em $4$-valent rigid vertex graphs}. 

\begin{figure}
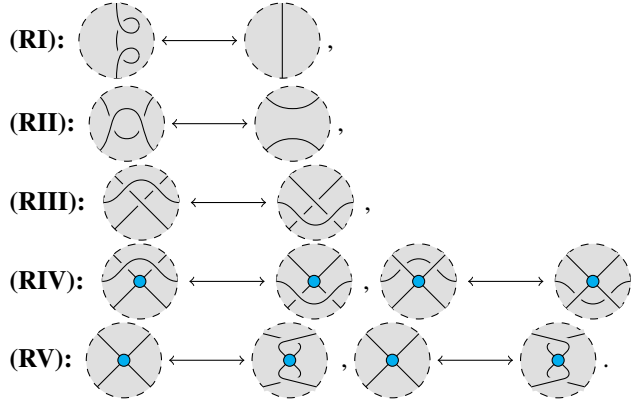

\centering
\begin{description}
\item[(RI)]
\tikz[baseline=-.6ex]{
\draw [thin, dashed, fill=lightgray!50] (0,0) circle [radius=.5];
\draw (.3,-.2)
to[out=south, in=east] (.2,-.3)
to[out=west, in=south] (.0,.0)
to[out=north, in=west] (.2,.3)
to[out=east, in=north] (.3,.2);
\draw[lightgray!50, double=black, double distance=0.4pt, ultra thick] (0,-.5) 
to[out=north, in=west] (.2,-.1)
to[out=east, in=north] (.3,-.2);
\draw[lightgray!50, double=black, double distance=0.4pt, ultra thick] (.3,.2)
to[out=south, in=east] (.2,.1)
to[out=west, in=south] (0,.5);
}
\tikz[baseline=-.6ex]{
\draw [<->, xshift=1.5cm] (1,0)--(2,0);
}
\tikz[baseline=-.6ex]{
\draw[xshift=3cm, thin, dashed, fill=lightgray!50] (0,0) circle [radius=.5];
\draw[xshift=3cm] (90:.5) to (-90:.5);
}\ ,
\item[(RII)]
\tikz[baseline=-.6ex]{
\draw [thin, dashed, fill=lightgray!50] (0,0) circle [radius=.5];
\draw (135:.5) to [out=south east, in=west](0,-.2) to [out=east, in=south west](45:.5);
\draw [lightgray!50, double=black, double distance=0.4pt, ultra thick](-135:.5) to [out=north east, in=left](0,.2) to [out=right, in=north west] (-45:.5);
}
\tikz[baseline=-.6ex]{
\draw [<->, xshift=1.5cm] (1,0)--(2,0);
}
\tikz[baseline=-.6ex]{
\draw[xshift=3cm, thin, dashed, fill=lightgray!50] (0,0) circle [radius=.5];
\draw[xshift=3cm] (135:.5) to [out=south east, in=west](0,.2) to [out=east, in=south west](45:.5);
\draw[xshift=3cm] (-135:.5) to [out=north east, in=left](0,-.2) to [out=right, in=north west] (-45:.5);
}\ ,
\item[(RIII)]
\tikz[baseline=-.6ex]{
\draw [thin, dashed, fill=lightgray!50] (0,0) circle [radius=.5];
\draw (-135:.5) -- (45:.5);
\draw [lightgray!50, double=black, double distance=0.4pt, ultra thick] (135:.5) -- (-45:.5);
\draw[lightgray!50, double=black, double distance=0.4pt, ultra thick](180:.5) to [out=right, in=left](0,.3) to [out=right, in=left] (-0:.5);
}
\tikz[baseline=-.6ex]{
\draw[<->, xshift=1.5cm] (1,0)--(2,0);
}
\tikz[baseline=-.6ex]{
\draw [xshift=3cm, thin, dashed, fill=lightgray!50] (0,0) circle [radius=.5];
\draw [xshift=3cm] (-135:.5) -- (45:.5);
\draw [xshift=3cm, lightgray!50, double=black, double distance=0.4pt, ultra thick] (135:.5) -- (-45:.5);
\draw[xshift=3cm, lightgray!50, double=black, double distance=0.4pt, ultra thick](180:.5) to [out=right, in=left](0,-.3) to [out=right, in=left] (-0:.5);
}\ ,
\item[(RIV)]
\tikz[baseline=-.6ex]{
\draw [thin, dashed, fill=lightgray!50] (0,0) circle [radius=.5];
\draw (-135:.5) -- (45:.5); 
\draw (-45:.5) -- (135:.5); 
\draw[lightgray!50, double=black, double distance=0.4pt, ultra thick](180:.5) to [out=right, in=left](0,.3) to [out=right, in=left] (-0:.5);
\draw[fill=cyan] (0,0) circle [radius=.08];
}
\tikz[baseline=-.6ex]{
\draw[<->, xshift=1.5cm] (1,0)--(2,0);
}
\tikz[baseline=-.6ex]{
\draw [thin, dashed, fill=lightgray!50] (0,0) circle [radius=.5];
\draw (-135:.5) -- (45:.5); 
\draw (-45:.5) -- (135:.5); 
\draw[lightgray!50, double=black, double distance=0.4pt, ultra thick](180:.5) to [out=right, in=left](0,-.3) to [out=right, in=left] (-0:.5);
\draw[fill=cyan] (0,0) circle [radius=.08];
}\ , 
\tikz[baseline=-.6ex]{
\draw [thin, dashed, fill=lightgray!50] (0,0) circle [radius=.5];
\draw[lightgray!50, double=black, double distance=0.4pt, ultra thick](180:.5) to [out=right, in=left](0,.3) to [out=right, in=left] (-0:.5);
\draw[lightgray!50, double=black, double distance=0.4pt, ultra thick] (-135:.5) -- (45:.5); 
\draw[lightgray!50, double=black, double distance=0.4pt, ultra thick] (-45:.5) -- (135:.5); 
\draw[fill=cyan] (0,0) circle [radius=.08];
}
\tikz[baseline=-.6ex]{
\draw[<->, xshift=1.5cm] (1,0)--(2,0);
}
\tikz[baseline=-.6ex]{
\draw[thin, dashed, fill=lightgray!50] (0,0) circle [radius=.5];
\draw[lightgray!50, double=black, double distance=0.4pt, ultra thick] (180:.5) to [out=right, in=left](0,-.3) to [out=right, in=left] (-0:.5);
\draw[lightgray!50, double=black, double distance=0.4pt, ultra thick] (-135:.5) -- (45:.5); 
\draw[lightgray!50, double=black, double distance=0.4pt, ultra thick] (-45:.5) -- (135:.5); 
\draw[fill=cyan] (0,0) circle [radius=.08];
},
\item[(RV)]
\tikz[baseline=-.6ex]{
\draw [thin, dashed, fill=lightgray!50] (0,0) circle [radius=.5];
\draw (0,0) -- (45:.5);
\draw (0,0) -- (-135:.5); 
\draw (0,0) -- (-45:.5);
\draw (0,0) -- (135:.5); 
\draw[fill=cyan] (0,0) circle [radius=.08];
}
\tikz[baseline=-.6ex]{
\draw[<->, xshift=1.5cm] (1,0)--(2,0);
}
\tikz[baseline=-.6ex]{
\draw [thin, dashed, fill=lightgray!50] (0,0) circle [radius=.5];
\draw[lightgray!50, double=black, double distance=0.4pt, ultra thick, rounded corners] 
(0,0) -- (45:.3) -- (135:.5);
\draw[lightgray!50, double=black, double distance=0.4pt, ultra thick, rounded corners] 
(0,0) -- (135:.3) -- (45:.5);
\draw[lightgray!50, double=black, double distance=0.4pt, ultra thick, rounded corners] 
(0,0) -- (-45:.3) -- (-135:.5);
\draw[lightgray!50, double=black, double distance=0.4pt, ultra thick, rounded corners] 
(0,0) -- (-135:.3) -- (-45:.5);
\draw[fill=cyan] (0,0) circle [radius=.08];
}
\ , 
\tikz[baseline=-.6ex]{
\draw [thin, dashed, fill=lightgray!50] (0,0) circle [radius=.5];
\draw (-135:.5) -- (45:.5); 
\draw (-45:.5) -- (135:.5); 
\draw[fill=cyan] (0,0) circle [radius=.08];
}
\tikz[baseline=-.6ex]{
\draw[<->, xshift=1.5cm] (1,0)--(2,0);
}
\tikz[baseline=-.6ex]{
\draw [thin, dashed, fill=lightgray!50] (0,0) circle [radius=.5];
\draw[lightgray!50, double=black, double distance=0.4pt, ultra thick, rounded corners] 
(0,0) -- (135:.3) -- (45:.5);
\draw[lightgray!50, double=black, double distance=0.4pt, ultra thick, rounded corners] 
(0,0) -- (45:.3) -- (135:.5);
\draw[lightgray!50, double=black, double distance=0.4pt, ultra thick, rounded corners] 
(0,0) -- (-135:.3) -- (-45:.5);
\draw[lightgray!50, double=black, double distance=0.4pt, ultra thick, rounded corners] 
(0,0) -- (-45:.3) -- (-135:.5);
\draw[fill=cyan] (0,0) circle [radius=.08];
}
.
\end{description}
\caption{The Reidemeister moves for $4$-valent graph diagrams}
\label{Reidemeister}
\end{figure}

\section{The $A_2$ bracket and some formulas}
We will construct invariants of oriented and unoriented $4$-valent rigid vertex graphs by using the linear skein theory corresponding to the quantum $A_2$ representation.
In this section, 
we introduce the $A_2$ web spaces, the $A_2$ bracket, and the $A_2$ clasps defined by Kuperberg~\cite{Kuperberg94, Kuperberg96}. 
Special skein elements called the $A_2$ clasps play an important role in construction of the colored $A_2$ polynomials for $4$-valent rigid vertex graphs.

Let $\varepsilon=(\varepsilon_1,\varepsilon_2,\dots,\varepsilon_m)$ be an $m$-tuple of signs ${+}$ or ${-}$. 
Let $D_\varepsilon$ denote the unit disk with signed marked points $\{\exp(2\pi\sqrt{-1}/m)^{j-1}\mid j=1,2,\dots,m \}$ on its boundary. 
The sign of $\exp(2\pi\sqrt{-1}/m)^{j-1}$ is given by $\varepsilon_j$ for $j=1,2,\dots,m$.
A {\em bipartite uni-trivalent graph} $G$ is a directed graph such that each vertex is either trivalent or univalent and the vertices are divided into the sinks and the sources. 
A sink (resp. source) is a vertex such that all edges adjoining to the vertex point into (resp. away from) it.
A {\em bipartite trivalent graph} $G$ in $D_{\varepsilon}$ is an embedding of a uni-trivalent graph into $D_\varepsilon$ such that any vertex $v$ has the following neighborhoods:
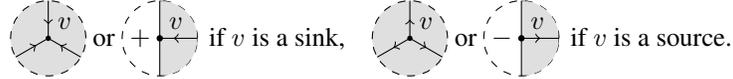
\[
\tikz[baseline=-.6ex]{
\draw [thin, dashed, fill=lightgray!50] (0,0) circle [radius=.5];
\draw[-<-=.5] (0:0) -- (90:.5);
\draw[-<-=.5] (0:0) -- (210:.5);
\draw[-<-=.5] (0:0) -- (-30:.5);
\node (v) at (0,0) [above right]{$v$};
\fill (0,0) circle [radius=1pt];
}
\text{\ or\ }
\tikz[baseline=-.6ex]{
\draw [thin, dashed] (0,0) circle [radius=.5];
\clip (0,0) circle [radius=.5];
\draw [thin, fill=lightgray!50] (0,-.5) rectangle (.5,.5);
\draw[-<-=.5] (0,0)--(.5,0);
\node (p) at (0,0) [left]{${+}$};
\node at (0,0) [above right]{$v$};
\draw [fill=black] (0,0) circle [radius=1pt];
}\text{\, if $v$ is a sink,}
\quad\tikz[baseline=-.6ex]{
\draw [thin, dashed, fill=lightgray!50] (0,0) circle [radius=.5];
\draw[->-=.5] (0:0) -- (90:.5);
\draw[->-=.5] (0:0) -- (210:.5);
\draw[->-=.5] (0:0) -- (-30:.5);
\node (v) at (0,0) [above right]{$v$};
\fill (0,0) circle [radius=1pt];
}
\text{\ or\ }
\tikz[baseline=-.6ex]{
\draw [thin, dashed] (0,0) circle [radius=.5];
\clip (0,0) circle [radius=.5];
\draw [thin, fill=lightgray!50] (0,-.5) rectangle (.5,.5);
\draw[->-=.5] (0,0)--(.5,0);
\node (p) at (0,0) [left]{${-}$};
\node at (0,0) [above right]{$v$};
\draw [fill=black] (0,0) circle [radius=1pt];
}\text{\, if $v$ is a source.}
\]
An {\em $A_2$ basis web} is the boundary-fixing isotopy class of a bipartite trivalent graph $G$ in $D_{\varepsilon}$, 
where any internal face of $D_{\varepsilon}\setminus G$ has at least six sides. 
Let us denote $B_\varepsilon$ the set of $A_2$ basis webs in $D_{\varepsilon}$.
For example,
$B_{(+,-,+,-,+,-)}$ has the following $A_2$ basis webs:
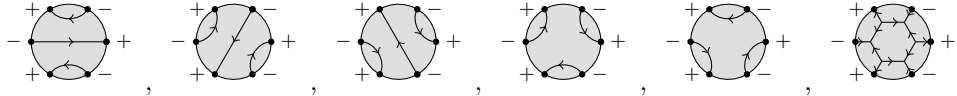
\[
\,\begin{tikzpicture}
\draw [thin, fill=lightgray!50] (0,0) circle [radius=.5];
\draw[-<-=.5] (0:.5) -- (180:.5);
\draw[->-=.5] (60:.5) to[out=-120, in=-60] (120:.5);
\draw[-<-=.5] (240:.5) to[out=60, in=120] (300:.5);
\foreach \i in {0,1,...,6} \draw[fill=black] ($(0,0) !1! \i*60:(.5,0)$) circle [radius=1pt];
\node at (0:.5) [right]{$\scriptstyle{+}$};
\node at (60:.5) [right]{$\scriptstyle{-}$};
\node at (120:.5) [left]{$\scriptstyle{+}$};
\node at (180:.5) [left]{$\scriptstyle{-}$};
\node at (240:.5) [left]{$\scriptstyle{+}$};
\node at (300:.5) [right]{$\scriptstyle{-}$};
\end{tikzpicture}\,,
\,\begin{tikzpicture}
\begin{scope}[rotate=60]
\draw [thin, fill=lightgray!50] (0,0) circle [radius=.5];
\draw[->-=.5] (0:.5) -- (180:.5);
\draw[-<-=.5] (60:.5) to[out=-120, in=-60] (120:.5);
\draw[->-=.5] (240:.5) to[out=60, in=120] (300:.5);
\foreach \i in {0,1,...,6} \draw[fill=black] ($(0,0) !1! \i*60:(.5,0)$) circle [radius=1pt];
\end{scope}
\node at (0:.5) [right]{$\scriptstyle{+}$};
\node at (60:.5) [right]{$\scriptstyle{-}$};
\node at (120:.5) [left]{$\scriptstyle{+}$};
\node at (180:.5) [left]{$\scriptstyle{-}$};
\node at (240:.5) [left]{$\scriptstyle{+}$};
\node at (300:.5) [right]{$\scriptstyle{-}$};
\end{tikzpicture}\,,
\,\begin{tikzpicture}
\begin{scope}[rotate=-60]
\draw [thin, fill=lightgray!50] (0,0) circle [radius=.5];
\draw[->-=.5] (0:.5) -- (180:.5);
\draw[-<-=.5] (60:.5) to[out=-120, in=-60] (120:.5);
\draw[->-=.5] (240:.5) to[out=60, in=120] (300:.5);
\foreach \i in {0,1,...,6} \draw[fill=black] ($(0,0) !1! \i*60:(.5,0)$) circle [radius=1pt];
\end{scope}
\node at (0:.5) [right]{$\scriptstyle{+}$};
\node at (60:.5) [right]{$\scriptstyle{-}$};
\node at (120:.5) [left]{$\scriptstyle{+}$};
\node at (180:.5) [left]{$\scriptstyle{-}$};
\node at (240:.5) [left]{$\scriptstyle{+}$};
\node at (300:.5) [right]{$\scriptstyle{-}$};
\end{tikzpicture}\,,
\,\begin{tikzpicture}
\draw [thin, fill=lightgray!50] (0,0) circle [radius=.5];
\draw[-<-=.5] (0:.5) to[out=180, in=-120] (60:.5);
\draw[rotate=120, -<-=.5] (0:.5) to[out=180, in=-120] (60:.5);
\draw[rotate=240, -<-=.5] (0:.5) to[out=180, in=-120] (60:.5);
\foreach \i in {0,1,...,6} \draw[fill=black] ($(0,0) !1! \i*60:(.5,0)$) circle [radius=1pt];
\node at (0:.5) [right]{$\scriptstyle{+}$};
\node at (60:.5) [right]{$\scriptstyle{-}$};
\node at (120:.5) [left]{$\scriptstyle{+}$};
\node at (180:.5) [left]{$\scriptstyle{-}$};
\node at (240:.5) [left]{$\scriptstyle{+}$};
\node at (300:.5) [right]{$\scriptstyle{-}$};
\end{tikzpicture}\,,
\,\begin{tikzpicture}
\begin{scope}[rotate=60]
\draw [thin, fill=lightgray!50] (0,0) circle [radius=.5];
\draw[->-=.5] (0:.5) to[out=180, in=-120] (60:.5);
\draw[rotate=120, ->-=.5] (0:.5) to[out=180, in=-120] (60:.5);
\draw[rotate=240, ->-=.5] (0:.5) to[out=180, in=-120] (60:.5);
\foreach \i in {0,1,...,6} \draw[fill=black] ($(0,0) !1! \i*60:(.5,0)$) circle [radius=1pt];
\end{scope}
\node at (0:.5) [right]{$\scriptstyle{+}$};
\node at (60:.5) [right]{$\scriptstyle{-}$};
\node at (120:.5) [left]{$\scriptstyle{+}$};
\node at (180:.5) [left]{$\scriptstyle{-}$};
\node at (240:.5) [left]{$\scriptstyle{+}$};
\node at (300:.5) [right]{$\scriptstyle{-}$};
\end{tikzpicture}\,,
\,\begin{tikzpicture}
\draw [thin, fill=lightgray!50] (0,0) circle [radius=.5];
\draw[-<-=.5] (0:.5) -- (0:.3);
\draw[rotate=60, ->-=.5] (0:.5) -- (0:.3);
\draw[rotate=120, -<-=.5] (0:.5) -- (0:.3);
\draw[rotate=180, ->-=.5] (0:.5) -- (0:.3);
\draw[rotate=240, -<-=.5] (0:.5) -- (0:.3);
\draw[rotate=300, ->-=.5] (0:.5) -- (0:.3);
\draw[->-=.5] (0:.3) -- (60:.3);
\draw[rotate=60, -<-=.5] (0:.3) -- (60:.3);
\draw[rotate=120, ->-=.5] (0:.3) -- (60:.3);
\draw[rotate=180, -<-=.5] (0:.3) -- (60:.3);
\draw[rotate=240, ->-=.5] (0:.3) -- (60:.3);
\draw[rotate=300, -<-=.5] (0:.3) -- (60:.3);
\foreach \i in {0,1,...,6} \draw[fill=black] ($(0,0) !1! \i*60:(.5,0)$) circle [radius=1pt];
\node at (0:.5) [right]{$\scriptstyle{+}$};
\node at (60:.5) [right]{$\scriptstyle{-}$};
\node at (120:.5) [left]{$\scriptstyle{+}$};
\node at (180:.5) [left]{$\scriptstyle{-}$};
\node at (240:.5) [left]{$\scriptstyle{+}$};
\node at (300:.5) [right]{$\scriptstyle{-}$};
\end{tikzpicture}\,.
\]
The {\em $A_2$ web space $W_\varepsilon$} is the $\mathbb{Q}(q^{\frac{1}{6}})$-vector space spanned by $B_\varepsilon$. 
A {\em tangled trivalent graph diagram} in $D_\varepsilon$ is an immersed bipartite uni-trivalent graph in $D_\varepsilon$ whose intersection points are only transverse double points of edges with crossing data 
\,\tikz[baseline=-.6ex, scale=.8]{
\draw [thin, dashed, fill=lightgray!50] (0,0) circle [radius=.5];
\draw[->-=.8] (-45:.5) -- (135:.5);
\draw[->-=.8, lightgray!50, double=black, double distance=0.4pt, ultra thick] (-135:.5) -- (45:.5);
}\, or 
\,\tikz[baseline=-.6ex, scale=.8]{
\draw [thin, dashed, fill=lightgray!50] (0,0) circle [radius=.5];
\draw[->-=.8] (-135:.5) -- (45:.5);
\draw[->-=.8, lightgray!50, double=black, double distance=0.4pt, ultra thick] (-45:.5) -- (135:.5);
}\, .
Tangled trivalent graph diagrams $G$ and $G'$ are regularly isotopic if $G$ is obtained from $G'$ by a finite sequence of boundary-fixing isotopies and Reidemeister moves, see Figure~\ref{Reidemeister}, with some direction of edges.
\begin{figure}
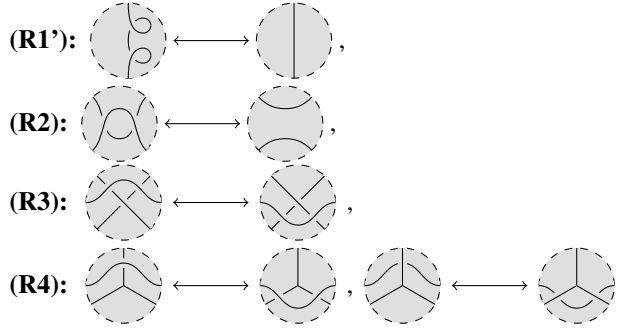

\begin{description}
\item[(R1')]
\tikz[baseline=-.6ex]{
\draw [thin, dashed, fill=lightgray!50] (0,0) circle [radius=.5];
\draw (.3,-.2)
to[out=south, in=east] (.2,-.3)
to[out=west, in=south] (.0,.0)
to[out=north, in=west] (.2,.3)
to[out=east, in=north] (.3,.2);
\draw[lightgray!50, double=black, double distance=0.4pt, ultra thick] (0,-.5) 
to[out=north, in=west] (.2,-.1)
to[out=east, in=north] (.3,-.2);
\draw[lightgray!50, double=black, double distance=0.4pt, ultra thick] (.3,.2)
to[out=south, in=east] (.2,.1)
to[out=west, in=south] (0,.5);
}
\tikz[baseline=-.6ex]{
\draw [<->, xshift=1.5cm] (1,0)--(2,0);
}
\tikz[baseline=-.6ex]{
\draw[xshift=3cm, thin, dashed, fill=lightgray!50] (0,0) circle [radius=.5];
\draw[xshift=3cm] (90:.5) to (-90:.5);
}\ ,
\item[(R2)]
\tikz[baseline=-.6ex]{
\draw [thin, dashed, fill=lightgray!50] (0,0) circle [radius=.5];
\draw (135:.5) to [out=south east, in=west](0,-.2) to [out=east, in=south west](45:.5);
\draw [lightgray!50, double=black, double distance=0.4pt, ultra thick](-135:.5) to [out=north east, in=left](0,.2) to [out=right, in=north west] (-45:.5);
}
\tikz[baseline=-.6ex]{
\draw [<->, xshift=1.5cm] (1,0)--(2,0);
}
\tikz[baseline=-.6ex]{
\draw[xshift=3cm, thin, dashed, fill=lightgray!50] (0,0) circle [radius=.5];
\draw[xshift=3cm] (135:.5) to [out=south east, in=west](0,.2) to [out=east, in=south west](45:.5);
\draw[xshift=3cm] (-135:.5) to [out=north east, in=left](0,-.2) to [out=right, in=north west] (-45:.5);
}\ ,
\item[(R3)]
\tikz[baseline=-.6ex]{
\draw [thin, dashed, fill=lightgray!50] (0,0) circle [radius=.5];
\draw (-135:.5) -- (45:.5);
\draw [lightgray!50, double=black, double distance=0.4pt, ultra thick] (135:.5) -- (-45:.5);
\draw[lightgray!50, double=black, double distance=0.4pt, ultra thick](180:.5) to [out=right, in=left](0,.3) to [out=right, in=left] (-0:.5);
}
\tikz[baseline=-.6ex]{
\draw[<->, xshift=1.5cm] (1,0)--(2,0);
}
\tikz[baseline=-.6ex]{
\draw [xshift=3cm, thin, dashed, fill=lightgray!50] (0,0) circle [radius=.5];
\draw [xshift=3cm] (-135:.5) -- (45:.5);
\draw [xshift=3cm, lightgray!50, double=black, double distance=0.4pt, ultra thick] (135:.5) -- (-45:.5);
\draw[xshift=3cm, lightgray!50, double=black, double distance=0.4pt, ultra thick](180:.5) to [out=right, in=left](0,-.3) to [out=right, in=left] (-0:.5);
}\ ,
\item[(R4)]
\tikz[baseline=-.6ex]{
\draw [thin, dashed, fill=lightgray!50] (0,0) circle [radius=.5];
\draw (0:0) -- (90:.5); 
\draw (0:0) -- (210:.5); 
\draw (0:0) -- (-30:.5);
\draw[lightgray!50, double=black, double distance=0.4pt, ultra thick](180:.5) to [out=right, in=left](0,.3) to [out=right, in=left] (-0:.5);
}
\tikz[baseline=-.6ex]{
\draw[<->, xshift=1.5cm] (1,0)--(2,0);
}
\tikz[baseline=-.6ex]{
\draw [thin, dashed, fill=lightgray!50] (0,0) circle [radius=.5];
\draw (0:0) -- (90:.5); 
\draw (0:0) -- (210:.5); 
\draw (0:0) -- (-30:.5);
\draw[lightgray!50, double=black, double distance=0.4pt, ultra thick](180:.5) to [out=right, in=left](0,-.3) to [out=right, in=left] (-0:.5);
}\ , 
\tikz[baseline=-.6ex]{
\draw [thin, dashed, fill=lightgray!50] (0,0) circle [radius=.5];
\draw[lightgray!50, double=black, double distance=0.4pt, ultra thick](180:.5) to [out=right, in=left](0,.3) to [out=right, in=left] (-0:.5);
\draw[lightgray!50, double=black, double distance=0.4pt, ultra thick] (0:0) -- (90:.5); 
\draw (0:0) -- (210:.5); 
\draw (0:0) -- (-30:.5);
}
\tikz[baseline=-.6ex]{
\draw[<->, xshift=1.5cm] (1,0)--(2,0);
}
\tikz[baseline=-.6ex]{
\draw [thin, dashed, fill=lightgray!50] (0,0) circle [radius=.5];
\draw[lightgray!50, double=black, double distance=0.4pt, ultra thick](180:.5) to [out=right, in=left](0,-.3) to [out=right, in=left] (-0:.5);
\draw (0:0) -- (90:.5); 
\draw[lightgray!50, double=black, double distance=0.4pt, ultra thick] (0:0) -- (210:.5); 
\draw[lightgray!50, double=black, double distance=0.4pt, ultra thick] (0:0) -- (-30:.5);
}.
\end{description}
\caption{The Reidemeister moves for tangled trivalent graph diagrams}
\label{Reidemeister}
\end{figure}

{\em Tangled trivalent graphs} in $D_\varepsilon$ are regular isotopy classes of tangled trivalent graph diagrams in $D_\varepsilon$. 
We denote $T_\varepsilon$ the set of tangled trivalent graphs in $D_\varepsilon$.

\begin{DEF}[The $A_2$ bracket~\cite{Kuperberg96}]\label{A2def}
We define a $\mathbb{Q}(q^{\frac{1}{6}})$-linear map $\langle\,\cdot\,\rangle_3\colon\mathbb{Q}(q^{\frac{1}{6}})T_\varepsilon\to W_\varepsilon$ by the following.
\begin{itemize}
\item 
$\Big\langle\,\tikz[baseline=-.6ex, scale=0.8]{
\draw [thin, dashed, fill=lightgray!50] (0,0) circle [radius=.5];
\draw[->-=.8] (-45:.5) -- (135:.5);
\draw[->-=.8, lightgray!50, double=black, double distance=0.4pt, ultra thick] (-135:.5) -- (45:.5);
}\,\Big\rangle_{\! 3}
=
q^{\frac{1}{3}}\Big\langle\,\tikz[baseline=-.6ex, scale=0.8]{
\draw[thin, dashed, fill=lightgray!50] (0,0) circle [radius=.5];
\draw[->-=.5] (-45:.5) to [out=north west, in=south](.2,0) to [out=north, in=south west](45:.5);
\draw[->-=.5] (-135:.5) to [out=north east, in=south](-.2,0) to [out=north, in=south east] (135:.5);
}\,\Big\rangle_{\! 3}
-
q^{-\frac{1}{6}}\Big\langle\,\tikz[baseline=-.6ex, scale=0.8]{
\draw[thin, dashed, fill=lightgray!50] (0,0) circle [radius=.5];
\draw[->-=.5] (-45:.5) -- (0,-.2);
\draw[->-=.5] (-135:.5) -- (0,-.2);
\draw[-<-=.5] (0,-.2) -- (0,.2);
\draw[-<-=.5] (45:.5) -- (0,.2);
\draw[-<-=.5] (135:.5) -- (0,.2);
}\,\Big\rangle_{\! 3}
$,
\item 
$\Big\langle\,\tikz[baseline=-.6ex, scale=0.8]{
\draw [thin, dashed, fill=lightgray!50] (0,0) circle [radius=.5];
\draw[->-=.8] (-135:.5) -- (45:.5);
\draw[->-=.8, lightgray!50, double=black, double distance=0.4pt, ultra thick] (-45:.5) -- (135:.5);
}\,\Big\rangle_{\! 3}
=
q^{-\frac{1}{3}}\Big\langle\,\tikz[baseline=-.6ex, scale=0.8]{
\draw[thin, dashed, fill=lightgray!50] (0,0) circle [radius=.5];
\draw[->-=.5] (-45:.5) to [out=north west, in=south](.2,0) to [out=north, in=south west](45:.5);
\draw[->-=.5] (-135:.5) to [out=north east, in=south](-.2,0) to [out=north, in=south east] (135:.5);
}\,\Big\rangle_{\! 3}
-
q^{\frac{1}{6}}\Big\langle\,\tikz[baseline=-.6ex, scale=0.8]{
\draw[thin, dashed, fill=lightgray!50] (0,0) circle [radius=.5];
\draw[->-=.5] (-45:.5) -- (0,-.2);
\draw[->-=.5] (-135:.5) -- (0,-.2);
\draw[-<-=.5] (0,-.2) -- (0,.2);
\draw[-<-=.5] (45:.5) -- (0,.2);
\draw[-<-=.5] (135:.5) -- (0,.2);
}\,\Big\rangle_{\! 3}
$,
\item 
$\Big\langle\,\tikz[baseline=-.6ex, rotate=90, scale=0.8]{
\draw[thin, dashed, fill=lightgray!50] (0,0) circle [radius=.5];
\draw[->-=.6] (-45:.5) -- (-45:.3);
\draw[->-=.6] (-135:.5) -- (-135:.3);
\draw[-<-=.6] (45:.5) -- (45:.3);
\draw[->-=.6] (135:.5) -- (135:.3);
\draw[->-=.5] (45:.3) -- (135:.3);
\draw[-<-=.5] (-45:.3) -- (-135:.3);
\draw[->-=.5] (45:.3) -- (-45:.3);
\draw[-<-=.5] (135:.3) -- (-135:.3);
}\,\Big\rangle_{\! 3}
=
\Big\langle\,\tikz[baseline=-.6ex, rotate=90, scale=0.8]{
\draw[thin, dashed, fill=lightgray!50] (0,0) circle [radius=.5];
\draw[->-=.5] (-45:.5) to [out=north west, in=south](.2,0) to [out=north, in=south west](45:.5);
\draw[-<-=.5] (-135:.5) to [out=north east, in=south](-.2,0) to [out=north, in=south east] (135:.5);
}\,\Big\rangle_{\! 3}
+
\Big\langle\,\tikz[rotate=90, baseline=-.6ex, rotate=90, scale=0.8]{
\draw[thin, dashed, fill=lightgray!50] (0,0) circle [radius=.5];
\draw[-<-=.5] (-45:.5) to [out=north west, in=south](.2,0) to [out=north, in=south west](45:.5);
\draw[->-=.5] (-135:.5) to [out=north east, in=south](-.2,0) to [out=north, in=south east] (135:.5);
}\,\Big\rangle_{\! 3}
$,
\item 
$\Big\langle\,\tikz[baseline=-.6ex, rotate=-90, scale=0.8]{
\draw[thin, dashed, fill=lightgray!50] (0,0) circle [radius=.5];
\draw[->-=.5] (0,-.5) -- (0,-.25);
\draw[->-=.5] (0,.25) -- (0,.5);
\draw[-<-=.5] (0,-.25) to [out=60, in=120, relative](0,.25);
\draw[-<-=.5] (0,-.25) to [out=-60, in=-120, relative](0,.25);
}\,\Big\rangle_{\! 3}
=
\left[2\right]\Big\langle\,\tikz[baseline=-.6ex, rotate=-90, scale=0.8]{
\draw[thin, dashed, fill=lightgray!50] (0,0) circle [radius=.5];
\draw[->-=.5] (0,-.5) -- (0,.5);
}\,\Big\rangle_{\! 3}
$,
\item 
$
\Big\langle G\sqcup
\,\tikz[baseline=-.6ex, scale=0.8]{
\draw[thin, dashed, fill=lightgray!50] (0,0) circle [radius=.5];
\draw[->-=.5] (0,0) circle [radius=.3];
}\,\Big\rangle_{\! 3}
=
\left[3\right] \langle G\rangle_{3},
$
\end{itemize}
where $\left[n\right]=\frac{q^{\frac{n}{2}}-q^{-\frac{n}{2}}}{q^{\frac{1}{2}}-q^{-\frac{1}{2}}}$ is a quantum integer.
\end{DEF}
We remark that this map is invariant under the Reidemeister moves for tangled trivalent graphs.

We next consider the $A_2$ web space $W_{n^{+}+n^{-}}=W_{({+},{+},\dots,{+},{-},{-},\dots,{-})}$. 
The $n$ marked points with sign ${+}$ lie in the right side and the $n$ marked points with sign ${-}$ in left side.
We define $A_2$ clasps 
$\tikz[baseline=-.6ex]{
\draw[->-=.8] (-.5,0) -- (.5,0);
\draw[fill=white] (-.1,-.3) rectangle (.1,.3);
\node at (.1,0) [above right]{${\scriptstyle n}$};
}\,\in W_{n^{+}+n^{-}}
$ 
inductively by the following. 
\begin{DEF}(The $A_2$ clasps)
\begin{align}
\tikz[baseline=-.6ex]{
\draw[->-=.8] (-.5,0) -- (.5,0);
\draw[fill=white] (-.1,-.3) rectangle (.1,.3);
\node at (.1,0) [above right]{${\scriptstyle 1}$};
}\,
&= 
\,\tikz[baseline=-.6ex]{
\draw[->-=.5] (-.5,0) -- (.5,0) node at (0,0) [above]{${\scriptstyle 1}$};
}\,\in W_{1^{+}+1^{-}}\notag\\
\tikz[baseline=-.6ex]{
\draw[->-=.8] (-.5,0) -- (.5,0);
\draw[fill=white] (-.1,-.3) rectangle (.1,.3);
\node at (.1,0) [above right]{${\scriptstyle n}$};
}\,
&=
\bigg\langle\,\tikz[baseline=-.6ex]{
\draw[->-=.8] (-.5,.1) -- (.5,.1);
\draw[->-=.5] (-.5,-.4) -- (.5,-.4);
\draw[fill=white] (-.1,-.2) rectangle (.1,.4);
\node at (0,0) [above right]{${\scriptstyle n-1}$};
\node at (.2,-.4) [above right]{${\scriptstyle 1}$};
}\,\bigg\rangle_{\! 3}
-\frac{\left[n-1\right]}{\left[n\right]}
\bigg\langle\,\tikz[baseline=-.6ex]{
\draw[->-=.5] (-.9,.1) -- (-.4,.1);
\draw[->-=.5] (-.3,.2) -- (.3,.2);
\draw[->-=.5] (.4,.1) -- (.9,.1);
\draw[->-=.5] (-.9,-.4) 
to[out=east, in=west] (-.3,-.4) 
to[out=east, in=south] (-.1,-.2);
\draw[-<-=.8] (-.1,-.2) 
to[out=north, in=east] (-.3,0);
\draw[-<-=.5] (.9,-.4) 
to[out=west, in=east] (.3,-.4)
to[out=west, in=south] (.1,-.2);
\draw[->-=.8] (.1,-.2)
to[out=north, in=west] (.3,0);
\draw[fill=white] (-.4,-.2) rectangle (-.3,.4);
\draw[fill=white] (.3,-.2) rectangle (.4,.4);
\draw[-<-=.5] (-.1,-.2) -- (.1,-.2);
\node at (-.3,0) [above left]{${\scriptscriptstyle n-1}$};
\node at (.3,0) [above right]{${\scriptscriptstyle n-1}$};
\node at (0,.1) [above]{${\scriptscriptstyle n-2}$};
\node at (-.6,-.3) {${\scriptscriptstyle 1}$};
\node at (.6,-.3) {${\scriptscriptstyle 1}$};
\node at (0,-.2) [below]{${\scriptscriptstyle 1}$};
\node at (-.3,0) [right]{${\scriptscriptstyle 1}$};
\node at (.3,0) [left]{${\scriptscriptstyle 1}$};
}\,\bigg\rangle_{\! 3} \in W_{n^{+}+n^{-}}
\end{align}
\end{DEF}
A strand decorated by a non-negative integer $n$ means $n$ parallelization of the strand. 
For example, 
$
\,\tikz[baseline=-.6ex]{
\draw [thin, dashed, fill=lightgray!50] (0,0) circle [radius=.5];
\draw (-.5,0)--(.5,0) node at (0,0) [above]{$n$};
}\,
=
\,\tikz[baseline=-.6ex]{
\draw [thin, dashed, fill=lightgray!50] (0,0) circle [radius=.5];
\draw (150:.5) -- (30:.5) (160:.5) -- (20:.5);
\node[rotate=90] at (0,0) {${\scriptstyle \cdots}$};
\draw (210:.5) -- (-30:.5);
\node at (.7,.0) {$\Big\}n$};
}\,$
,
$\,\tikz[baseline=-.6ex]{
\draw [thin, dashed] (0,0) circle [radius=.5];
\clip (0,0) circle [radius=.5];
\draw [thin, fill=lightgray!50] (0,-.5) rectangle (.5,.5);
\draw (0,0)--(.5,0);
\node at (.2,0) [above]{$n$};
\draw [fill=black] (0,0) circle [radius=1pt];
}\,
=
\,\tikz[baseline=-.6ex]{
\begin{scope}
\draw [thin, dashed] (0,0) circle [radius=.5];
\clip (0,0) circle [radius=.5];
\draw [thin, fill=lightgray!50] (0,-.5) rectangle (.5,.5);
\draw (0,.2) -- (.5,.2) (0,.1) -- (.5,.1) (0,-.2) -- (.5,-.2);
\end{scope}
\draw [fill=black] (0,.2) circle [radius=1pt] 
(0,.1) circle [radius=1pt] 
(0,-.2) circle [radius=1pt];
\node at (.1,0) [left]{$n\Big\{$};
\node[rotate=90] at (.2,-.05) {${\scriptstyle \cdots}$};
}\,$
, and
$\,\tikz[baseline=-.6ex]{
\draw [thin, dashed, fill=lightgray!50] (0,0) circle [radius=.5];
\draw (135:.5) -- (-45:.5);
\draw [lightgray!50, double=black, double distance=0.4pt, ultra thick] (-135:.5) -- (45:.5);
\node at (45:.5) [right]{$k$} node at (135:.5) [left]{$l$};
}\,
=
\,\tikz[baseline=-.6ex]{
\draw [thin, dashed, fill=lightgray!50] (0,0) circle [radius=.5];
\draw (165:.5) -- (-75:.5) (150:.5) -- (-60:.5) (120:.5) -- (-30:.5);
\node[rotate=45] at (-.3,.3) {${\scriptstyle{\cdots}}$};
\node[rotate=45] at (.3,-.3) {${\scriptstyle{\cdots}}$};
\draw [lightgray!50, double=black, double distance=0.4pt, thick] (-150:.5) -- (60:.5) (-120:.5) -- (30:.5) (-105:.5) -- (15:.5);
\node[rotate=-45] at (.3,.3) {${\scriptstyle{\cdots}}$};
\node[rotate=-45] at (-.3,-.3) {${\scriptstyle{\cdots}}$};
\node at (0,0) {${\scriptstyle{\cdots}}$};
\node at (45:.5) [above right]{$k$} node at (135:.5) [above left]{$l$};
\node[rotate=-135] at (40:.6) {$\Big\{$};
\node[rotate=-45] at (140:.6) {$\Big\{$};
}\, $.

$A_2$ clasps have the following properties.
\begin{LEM}[Properties of $A_2$ clasps]\label{A2claspprop}
For any positive integer $n$,
\begin{itemize} 
\item $\Big\langle\,\tikz[baseline=-.6ex]{
\draw[->-=.5] (-.6,0) -- (.6,0);
\draw[fill=white] (-.4,-.3) rectangle (-.2,.3);
\draw[fill=white] (.2,-.3) rectangle (.4,.3);
\node at (-.3,0) [above right]{${\scriptstyle n}$};
\node at (.3,0) [above right]{${\scriptstyle n}$};
}\,\Big\rangle_{\! 3}
=
\,\tikz[baseline=-.6ex]{
\draw[->-=.8] (-.5,0) -- (.5,0);
\draw[fill=white] (-.1,-.3) rectangle (.1,.3);
\node at (0,0) [above right] {${\scriptstyle n}$};
}\,
$,
\item $\Big\langle\,\tikz[baseline=-.6ex]{
\draw[->-=.5] (-.5,0) -- (-.1,0);
\draw[->-=.5] (0,.2) -- (.5,.2);
\draw[->-=.5] (0,-.2) -- (.5,-.2);
\draw[->-=.5] (0,.1) 
to[out=east, in=north] (.3,0);
\draw[->-=.5] (0,-.1)
to[out=east, in=south] (.3,0);
\draw[-<-=.5] (.3,0) -- (.5,0);
\draw[fill=white] (-.2,-.3) rectangle (0,.3);
\node at (.4,0) [right] {${\scriptstyle 1}$};
\node at (.4,.2) [right] {${\scriptstyle n-k-2}$};
\node at (.4,-.2) [right] {${\scriptstyle k}$};
}\,\Big\rangle_{\! 3}=0$\quad ($k=0,1,\dots,n-2$).
\end{itemize}
\end{LEM}

We also define the $A_2$ clasp of type $(n,m)$ according to Ohtsuki and Yamada~\cite{OhtsukiYamada97}.
\begin{DEF}[the $A_2$ clasp of type $(n,m)$]\label{doubleA2clasp}
\[
\Bigg\langle\,\tikz[baseline=-.6ex, scale=.8]{
\draw[-<-=.8] (-.6,.4) -- (.6,.4);
\draw[->-=.8] (-.6,-.4) -- (.6,-.4);
\draw[fill=white] (-.1,-.6) rectangle (.1,.6);
\draw (-.1,.0) -- (.1,.0);
\node at (.4,-.6)[right]{$\scriptstyle{m}$};
\node at (-.4,-.6)[left]{$\scriptstyle{m}$};
\node at (.4,.6)[right]{$\scriptstyle{n}$};
\node at (-.4,.6)[left]{$\scriptstyle{n}$};
}\,\Bigg\rangle_{\! 3}
=
\sum_{k=0}^{\min\{m,n\}}
(-1)^k
\frac{{n\brack k}{m\brack k}}{{n+m+1\brack k}}
\Bigg\langle\,\tikz[baseline=-.6ex]{
\draw
(-.4,.4) -- +(-.2,0)
(.4,-.4) -- +(.2,0)
(-.4,-.4) -- +(-.2,0)
(.4,.4) -- +(.2,0);
\draw[-<-=.5] (-.4,.5) -- (.4,.5);
\draw[->-=.5] (-.4,-.5) -- (.4,-.5);
\draw[-<-=.5] (-.4,.3) to[out=east, in=east] (-.4,-.3);
\draw[->-=.5] (.4,.3) to[out=west, in=west] (.4,-.3);
\draw[fill=white] (.4,-.6) rectangle +(.1,.4);
\draw[fill=white] (-.4,-.6) rectangle +(-.1,.4);
\draw[fill=white] (.4,.6) rectangle +(.1,-.4);
\draw[fill=white] (-.4,.6) rectangle +(-.1,-.4);
\node at (.4,-.6)[right]{$\scriptstyle{m}$};
\node at (-.4,-.6)[left]{$\scriptstyle{m}$};
\node at (.4,.6)[right]{$\scriptstyle{n}$};
\node at (-.4,.6)[left]{$\scriptstyle{n}$};
\node at (0,.5)[above]{$\scriptstyle{n-k}$};
\node at (0,-.5)[below]{$\scriptstyle{m-k}$};
\node at (-.2,0)[left]{$\scriptstyle{k}$};
\node at (.2,0)[right]{$\scriptstyle{k}$};
}\,\Bigg\rangle_{\! 3}
\]
where 
\[
 {n \brack k}=\frac{\left[n\right]!}{\left[k\right]!\left[n-k\right]!}=\frac{\{n\}!}{\{k\}!\{n-k\}!}
\]
for $k\leq n$.
\end{DEF}

\begin{LEM}[Property of $A_2$ clasps of type $(m,n)$]\label{doubleA2claspprop}
\[
\Big\langle\,\tikz[baseline=-.6ex]{
\draw[-<-=.5] (-.5,.2) -- (-.1,.2);
\draw[->-=.5] (-.5,-.2) -- (-.1,-.2);
\draw[-<-=.5] (0,.2) -- (.5,.2);
\draw[->-=.5] (0,-.2) -- (.5,-.2);
\draw[-<=.5] (0,.1) 
to[out=east, in=north] (.3,0);
\draw (0,-.1)
to[out=east, in=south] (.3,0);
\draw[fill=white] (-.2,-.3) rectangle (0,.3);
\draw (-.2,.0) -- (.0,.0);
\node at (.4,0) [right] {${\scriptstyle 1}$};
\node at (.4,.2) [right] {${\scriptstyle n-1}$};
\node at (.4,-.2) [right] {${\scriptstyle m-1}$};
}\,\Big\rangle_{\! 3}=0 .
\]
\end{LEM}

We use the following graphical notations to represent certain $A_2$ webs.
\begin{DEF}
For positive integers $n$ and $m$, 
\[
\,\tikz[baseline=-.6ex]{
\draw[->-=.1,->-=.9] (-.6,0) -- (.6,0);
\draw[->-=.1,->-=.9] (0,-.6) -- (0,.6);
\draw[fill=white] (-.15,-.3) rectangle (.15,.3);
\draw (-.15,.3) -- (.15,-.3);
\node at (-.4,0)[above]{$\scriptstyle{n}$};
\node at (.4,0)[above]{$\scriptstyle{n}$};
\node at (0,.5)[right]{$\scriptstyle{m}$};
\node at (0,-.5)[right]{$\scriptstyle{m}$};
}\,\in W_{n^{+}+m^{+}+n^{-}+m^{-}}
\]
is defined as follows:
$\,\tikz[baseline=-.6ex]{
\draw[->-=.1,->-=.9] (-.5,0) -- (.5,0);
\draw[->-=.1,->-=.9] (0,-.6) -- (0,.6);
\draw[fill=white] (-.1,-.3) rectangle (.1,.3);
\draw (-.1,.3) -- (.1,-.3);
\node at (-.3,0)[above]{$\scriptstyle{n}$}; 
\node at (.3,0)[above]{$\scriptstyle{n}$};
}\,
=
\,\tikz[baseline=-.6ex]{
\draw[->-=.5] (-.6,-.4) -- (.6,-.4);
\draw[->-=.6] (-.6,-.2) -- (.6,-.2);
\node[rotate=90] at (-.6,.1){$\scriptstyle{\cdots}$};
\node[rotate=90] at (.6,.1){$\scriptstyle{\cdots}$};
\draw[->-=.5] (-.6,.4) -- (.6,.4);
\draw[->-=.5] (-.4,-.6) -- (-.4,-.4);
\draw[->-=.5] (-.3,-.4) -- (-.3,-.2);
\draw[->-=.5] (-.2,-.2) -- (-.2,0);
\node at (0,.1){$\scriptstyle{\cdots}$};
\draw[->-=.5] (.3,.2) -- (.3,.4);
\draw[->-=.5] (.4,.4) -- (.4,.6);
}\,\in W_{n^{+}+1^{+}+n^{-}+1^{-}}$ for $m=1$,
$\,\tikz[baseline=-.6ex]{
\draw[->-=.1,->-=.9] (-.6,0) -- (.6,0);
\draw[->-=.1,->-=.9] (0,-.6) -- (0,.6);
\draw[fill=white] (-.15,-.3) rectangle (.15,.3);
\draw (-.15,.3) -- (.15,-.3);
\node at (-.4,0)[above]{$\scriptstyle{n}$};
\node at (.4,0)[above]{$\scriptstyle{n}$};
\node at (0,.5)[right]{$\scriptstyle{m}$};
\node at (0,-.5)[right]{$\scriptstyle{m}$};
}\,
=
\,\tikz[baseline=-.6ex]{
\draw[->-=.1] (-.6,0) -- (.9,0);
\draw[->-=.1,->-=.9] (0,-.6) -- (0,.6);
\draw[fill=white] (-.15,-.3) rectangle (.15,.3);
\draw (-.15,.3) -- (.15,-.3);
\node at (-.4,0)[above]{$\scriptstyle{n}$};
\node at (0,.5)[left]{$\scriptstyle{m-1}$};
\node at (0,-.5)[left]{$\scriptstyle{m-1}$};
\begin{scope}[xshift=.5cm]
\draw[->-=.1,->-=.9] (-.2,0) -- (.5,0);
\draw[->-=.1,->-=.9] (.1,-.6) -- (.1,.6);
\draw[fill=white] (0,-.3) rectangle (.2,.3);
\draw (0,.3) -- (.2,-.3);
\node at (-.2,0)[above]{$\scriptstyle{n}$}; 
\node at (.4,0)[above]{$\scriptstyle{n}$};
\end{scope}
}\,
$ for $m>1$.
We also define
$
\,\tikz[baseline=-.6ex]{
\draw[->-=.1,->-=.9] (-.6,0) -- (.6,0);
\draw[-<-=.1,-<-=.9] (0,-.6) -- (0,.6);
\draw[fill=white] (-.15,-.3) rectangle (.15,.3);
\draw (-.15,-.3) -- (.15,.3);
\node at (-.4,0)[above]{$\scriptstyle{n}$};
\node at (.4,0)[above]{$\scriptstyle{n}$};
\node at (0,.5)[right]{$\scriptstyle{m}$};
\node at (0,-.5)[right]{$\scriptstyle{m}$};
}\,\in W_{n^{+}+m^{-}+n^{-}+m^{+}}
$
in the same way.
\end{DEF}
\begin{DEF}
For positive integer $n$,
\[
\,\tikz[baseline=-.6ex]{
\draw (30:.5) -- (0,0);
\draw (150:.5) -- (0,0);
\draw[-<-=.2] (270:.5) -- (0,0);
\draw[->-=.2] (30:.5)
to[out=30, in=west] +(.3,.1);
\draw[->-=.2] (150:.5)
to[out=150, in=east] +(-.3,.1);
\draw[fill=white] (-30:.5) -- (90:.5) -- (210:.5) -- cycle;
\node at (30:.5) [above]{$\scriptstyle{n}$};
\node at (150:.5) [above]{$\scriptstyle{n}$};
\node at (270:.5) [left]{$\scriptstyle{n}$};
}\,\in W_{n^{+}+n^{+}+n^{+}}
\]
is defined as follows:
$\,\tikz[baseline=-.6ex]{
\draw[-<-=.5] (30:.4) -- (0,0);
\draw[-<-=.5] (150:.4) -- (0,0);
\draw[-<-=.5] (270:.4) -- (0,0);
}\,$ for $n=1$, 
$
\,\tikz[baseline=-.6ex]{
\draw (30:.4) -- (0,0);
\draw (150:.4) -- (0,0);
\draw[-<-=.2] (270:.4) -- (0,0);
\draw[->-=.2] (30:.4)
to[out=30, in=west] +(.3,.1);
\draw[->-=.2] (150:.4)
to[out=150, in=east] +(-.3,.1);
\draw[fill=white] (-30:.4) -- (90:.4) -- (210:.4) -- cycle;
\node at (30:.4) [above]{$\scriptstyle{n}$};
\node at (150:.4) [above]{$\scriptstyle{n}$};
\node at (270:.4) [left]{$\scriptstyle{n}$};
}\,
=
\,\tikz[baseline=-.6ex, xscale=-1]{
\draw (30:.2) -- (0,0);
\draw[->-=.8] (0,0) to[out=150, in=east] (170:.5);
\draw[-<-=.2] (270:.3) -- (0,0);
\draw[->-=.2, ->-=.9] (30:.2) -- +(.9,0);
\draw[fill=white] (-30:.3) -- (90:.3) -- (210:.3) -- cycle;
\draw[-<-=.2] (.6,-.4) -- (.6,.4);
\draw[-<-=.4, ->-=.9] (-.5,.4) -- (1.0,.4);
\draw[fill=white] (.5,-.2) rectangle (.7,.3);
\draw (.5,-.2) -- (.7,.3);
\node at (1,0) [left]{$\scriptstyle{n-1}$};
\node at (150:.5) [below right]{$\scriptstyle{n-1}$};
\node at (270:.3) [below]{$\scriptstyle{n-1}$};
\node at (150:.5) [above right]{$\scriptstyle{1}$};
\node at (1,.4) [left]{$\scriptstyle{1}$};
}\,$
for $n>1$.
We also define 
$\,\tikz[baseline=-.6ex]{
\draw (30:.4) -- (0,0);
\draw (150:.4) -- (0,0);
\draw[->-=.2] (270:.4) -- (0,0);
\draw[-<-=.2] (30:.4)
to[out=30, in=west] +(.3,.1);
\draw[-<-=.2] (150:.4)
to[out=150, in=east] +(-.3,.1);
\draw[fill=white] (-30:.4) -- (90:.4) -- (210:.4) -- cycle;
\node at (30:.4) [above]{$\scriptstyle{n}$};
\node at (150:.4) [above]{$\scriptstyle{n}$};
\node at (270:.4) [left]{$\scriptstyle{n}$};
}\,\in W_{n^{-}+n^{-}+n^{-}}$ in the same way.
\end{DEF}

We sometimes omit orientations of $A_2$ webs when the orientations obviously turn out from the previous $A_2$ webs.

We review some formulas for the $A_2$ bracket. 

\begin{LEM}[\cite{Yuasa17}]\label{coloredvertex}\ 
\begin{enumerate}
\item
$\,\tikz[baseline=-.6ex]{
\draw[->-=.5] (-.6,.0) -- (-.3,.0);
\draw (-.3,.0) -- (.3,.0);
\draw[->-=.8] (.3,.0) -- (.6,.0);
\draw[->-=.2, ->-=.9] (-.3,-.4) -- (-.3,.4);
\draw[->-=.2, ->-=.9] (.3,-.4) -- (.3,.4);
\draw[fill=white] (-.4,.2) rectangle +(.2,-.4);
\draw (-.4,.2) -- +(.2,-.4);
\draw[fill=white] (.2,.2) rectangle +(.2,-.4);
\draw (.2,.2) -- +(.2,-.4);
\node at (-.6,0)[left]{$\scriptstyle{n}$};
\node at (0,0)[above]{$\scriptstyle{n}$};
\node at (.6,0)[right]{$\scriptstyle{n}$};
\node at (-.3,.4)[above]{$\scriptstyle{m_1}$};
\node at (-.3,-.4)[below]{$\scriptstyle{m_1}$};
\node at (.3,.4)[above]{$\scriptstyle{m_2}$};
\node at (.3,-.4)[below]{$\scriptstyle{m_2}$};
}\,
=
\,\tikz[baseline=-.6ex]{
\draw[->-=.2, ->-=.9] (-.3,0) -- (.3,0);
\draw[->-=.2, ->-=.9] (0,-.4) -- (0,.4);
\draw[fill=white] (-.1,-.2) rectangle (.1,.2);
\draw (-.1,.2) -- (.1,-.2);
\node at (-.3,0)[left]{$\scriptstyle{n}$};
\node at (.3,0)[right]{$\scriptstyle{n}$};
\node at (0,.4)[above]{$\scriptstyle{m_1+m_2}$};
\node at (0,-.4)[below]{$\scriptstyle{m_1+m_2}$};
}\,
$,
\item
$\,\tikz[baseline=-.6ex]{
\draw[->-=.2, ->-=.9] (-.4,0) -- (.4,0);
\draw[->-=.2, ->-=.9] (0,-.4) -- (0,.4);
\draw[fill=white] (-.2,-.2) rectangle +(.4,.4);
\draw (-.2,.2) -- +(.4,-.4);
\node at (-.4,0)[left]{$\scriptstyle{n}$};
\node at (0,-.4)[below]{$\scriptstyle{n}$};
}\,
=
\,\tikz[baseline=-.6ex]{
\draw[->-=.5] (-.6,.0) -- (-.3,.0);
\draw (-.3,.0) -- (.3,.0);
\draw[->-=.8] (.3,.0) -- (.6,.0);
\draw[->-=.2] (-.3,-.4) -- (-.3,.0);
\draw[->-=.8] (.3,.0) -- (.3,.4);
\draw[fill=white] (-.5,-.2) -- (-.1,-.2) -- (-.1,.2) -- cycle;
\draw[fill=white, rotate=180] (-.5,-.2) -- (-.1,-.2) -- (-.1,.2) -- cycle;
\node at (-.6,0)[left]{$\scriptstyle{n}$};
\node at (.6,0)[right]{$\scriptstyle{n}$};
\node at (.3,.4)[above]{$\scriptstyle{n}$};
\node at (-.3,-.4)[below]{$\scriptstyle{n}$};
}$,
\item
$\,\tikz[baseline=-.6ex]{
\draw[->-=.2, -<-=.9] (.0,.4) -- (.6,.4);
\draw[->-=.2, ->-=.9] (.0,.0) -- (.6,.0);
\draw[-<-=.8] (.3,.3) [rounded corners]-- (.3,-.3) -- (.6,-.3);
\draw[fill=white] (.2,.3) -- (.4,.3) -- (.4,.5) -- cycle;
\draw[fill=white] (.2,.2) rectangle +(.2,-.4);
\draw (.2,.2) -- +(.2,-.4);
\node at (.6,.4)[right]{$\scriptstyle{m}$};
\node at (.6,.0)[right]{$\scriptstyle{n}$};
\node at (.6,-.3)[right]{$\scriptstyle{m}$};
}\,
=
\,\tikz[baseline=-.6ex, yscale=-1]{
\draw[->-=.2, -<-=.9] (.0,.4) -- (.6,.4);
\draw[->-=.2, ->-=.9] (.0,.0) -- (.6,.0);
\draw[-<-=.8] (.3,.3) [rounded corners]-- (.3,-.3) -- (.6,-.3);
\draw[fill=white] (.2,.3) -- (.4,.3) -- (.4,.5) -- cycle;
\draw[fill=white] (.2,.2) rectangle +(.2,-.4);
\draw (.2,.2) -- +(.2,-.4);
\node at (.6,.4)[right]{$\scriptstyle{m}$};
\node at (.6,.0)[right]{$\scriptstyle{n}$};
\node at (.6,-.3)[right]{$\scriptstyle{m}$};
}\,$,
\item
$\Big\langle\,\tikz[baseline=-.6ex]{
\draw[->-=.4] (.0,.4) -- (.4,.4);
\draw[-<-=.8] (.4,.4) -- (.4,-.4);
\draw[-<-=.8] (.3,.4) -- (.3,-.4);
\draw[->-=.2, ->-=.9] (.0,.0) -- (.7,0);
\draw[fill=white] (.2,.2) rectangle +(.3,-.4);
\draw (.2,.2) -- +(.3,-.4);
\node at (.0,.0)[left]{$\scriptstyle{n}$};
\node at (.0,.4)[left]{$\scriptstyle{1}$};
}\,\Big\rangle_3
=
\,\tikz[baseline=-.6ex]{
\draw[->-=.4] (.0,-.1) -- (.3,-.1);
\draw[-<-=.8] (.3,-.1) -- (.3,-.3);
\draw[-<-=.8] (.2,-.1) -- (.2,-.3);
\draw[->-=.5] (.0,.1) -- (.5,.1);
\node at (.0,.1)[left]{$\scriptstyle{n}$};
\node at (.0,-.1)[left]{$\scriptstyle{1}$};
}\,
+\sum_{i=0}^{n-1}
\,\tikz[baseline=-.6ex]{
\draw[->-=.5] (.0,.4) -- (.7,.4);
\draw[->-=.2, ->-=.9] (.0,-.05) -- (.7,-.05);
\draw[->-=.4] (.0,.3) -- (.3,.3);
\draw[->-=.4] (.0,.2) -- (.3,.2);
\draw[->-=.2] (.3,-.4) -- (.3,.3);
\draw[->-=.2] (.4,-.4) -- (.4,.2);
\draw[->-=.8] (.4,.2) -- (.7,.2);
\draw[fill=white] (.2,.1) rectangle +(.3,-.3);
\draw (.2,.1) -- +(.3,-.3);
\node at (.7,.4)[right]{$\scriptstyle{n-i-1}$};
\node at (.7,.2)[right]{$\scriptstyle{1}$};
\node at (.7,-.05)[right]{$\scriptstyle{i}$};
}$,
\item
$\bigg\langle\,\tikz[baseline=-.6ex]{
\draw[-<-=.8] (.0,.2) -- +(-.4,.0);
\draw[-<-=.7] (.6,.4) -- +(.4,.0);
\draw[->-=.7] (.6,.0) -- +(.4,.0);
\draw (.0,.4) -- (.6,.4);
\draw (.0,.0) -- (.6,.0);
\draw[-<-=.9] (.3,.3) [rounded corners]-- (.3,-.3) -- (1.0,-.3);
\draw[fill=white] (.2,.3) -- (.4,.3) -- (.4,.5) -- cycle;
\draw[fill=white] (.2,.2) rectangle +(.2,-.4);
\draw (.2,.2) -- +(.2,-.4);
\draw[fill=white] (.0,.5) rectangle (-.1,-.2);
\draw[fill=white] (.6,.2) rectangle +(.1,-.4);
\node at (-.4,.2)[left]{$\scriptstyle{n+m}$};
\node at (1.0,.4)[right]{$\scriptstyle{m}$};
\node at (1.0,.0)[right]{$\scriptstyle{n}$};
\node at (1.0,-.3)[right]{$\scriptstyle{m}$};
}\,\bigg\rangle_{\!3}
=\bigg\langle\,\tikz[baseline=-.6ex]{
\draw[-<-=.8] (.0,.2) -- +(-.4,.0);
\draw[-<-=.9] (.0,.4) -- (.6,.4);
\draw[->-=.9] (.0,.0) -- (.6,.0);
\draw[-<-=.9] (.3,.3) [rounded corners]-- (.3,-.3) -- (.6,-.3);
\draw[fill=white] (.2,.3) -- (.4,.3) -- (.4,.5) -- cycle;
\draw[fill=white] (.2,.2) rectangle +(.2,-.4);
\draw (.2,.2) -- +(.2,-.4);
\draw[fill=white] (.0,.5) rectangle (-.1,-.2);
\node at (-.4,.2)[left]{$\scriptstyle{n+m}$};
\node at (.6,.4)[right]{$\scriptstyle{m}$};
\node at (.6,.0)[right]{$\scriptstyle{n}$};
\node at (.6,-.3)[right]{$\scriptstyle{m}$};
}\,\bigg\rangle_{\!3}$,
\item
$\bigg\langle\,\tikz[baseline=-.6ex]{
\draw[-<-=.5] (-.4,.0) -- +(-.2,.0);
\draw[-<-=.5] (.4,.0) -- +(.2,.0);
\draw[-<-=.8] (.0,-.3) -- +(.0,-.3);
\draw (-.3,.0) -- (.0,.0);
\draw (.3,.0) -- (.0,.0);
\draw (.0,-.3) -- (.0,.0);
\draw[fill=white] (-.4,-.2) rectangle (-.3,.2);
\draw[fill=white] (.4,-.2) rectangle (.3,.2);
\draw[fill=white] (-.2,-.3) rectangle (.2,-.4);
\draw[fill=white] (-.2,-.2) -- (.2,-.2) -- (.2,.2) -- cycle;
\node at (-.6,0)[left]{$\scriptstyle{n}$};
\node at (.6,0)[right]{$\scriptstyle{n}$};
\node at (.0,-.6)[below]{$\scriptstyle{n}$};
}\bigg\rangle_3
=\bigg\langle\,\tikz[baseline=-.6ex]{
\draw[-<-=.5] (-.4,.0) -- +(-.2,.0);
\draw[-<-=.5] (.4,.0) -- +(.2,.0);
\draw[-<-=.8] (.0,-.3) -- +(.0,-.3);
\draw (-.3,.0) -- (.0,.0);
\draw (.4,.0) -- (.0,.0);
\draw (.0,-.3) -- (.0,.0);
\draw[fill=white] (-.4,-.2) rectangle (-.3,.2);
\draw[fill=white] (-.2,-.3) rectangle (.2,-.4);
\draw[fill=white] (-.2,-.2) -- (.2,-.2) -- (.2,.2) -- cycle;
\node at (-.6,0)[left]{$\scriptstyle{n}$};
\node at (.6,0)[right]{$\scriptstyle{n}$};
\node at (.0,-.6)[below]{$\scriptstyle{n}$};
}\bigg\rangle_3$.
\end{enumerate}
The above equations also hold for the opposite orientations.
\end{LEM}

\begin{LEM}\label{A2clasplem}\ 
For $k=0,1,\dots,n$,
\begin{enumerate}
\item $\Big\langle\,\tikz[baseline=-.6ex, scale=0.1]{
\draw[->-=.5] (-5,0) -- (-1,0);
\draw[->-=.8, white, double=black, double distance=0.4pt, ultra thick] 
(0,-2) to[out=right, in=left] (5,2);
\draw[->-=.8, white, double=black, double distance=0.4pt, ultra thick] 
(0,2) to[out=right, in=left] (5,-2);
\draw[fill=white] (-1,-3) rectangle (0,3);
\node at (-1,0) [above left]{${\scriptstyle n}$};
\node at (4,2) [right] {${\scriptstyle k}$};
\node at (4,-2) [right] {${\scriptstyle n-k}$};
}\,\Big\rangle_{\! 3}
=q^{\frac{k(n-k)}{3}}
\Big\langle\,\tikz[baseline=-.6ex, scale=0.1]{
\draw[->-=.2] (-5,0) -- (5,0);
\draw[fill=white] (-1,-3) rectangle (0,3);
\node at (1,0) [above right]{${\scriptstyle n}$};
}\,\Big\rangle_{\! 3}$
\item[]
and $\Big\langle\,\tikz[baseline=-.6ex, scale=0.1]{
\draw[->-=.5] (-5,0) -- (-1,0);
\draw[->-=.8, white, double=black, double distance=0.4pt, ultra thick] 
(0,2) to[out=right, in=left] (5,-2);
\draw[->-=.8, white, double=black, double distance=0.4pt, ultra thick] 
(0,-2) to[out=right, in=left] (5,2);
\draw[fill=white] (-1,-3) rectangle (0,3);
\node at (-1,0) [above left]{${\scriptstyle n}$};
\node at (4,2) [right] {${\scriptstyle k}$};
\node at (4,-2) [right] {${\scriptstyle n-k}$};
}\,\Big\rangle_{\! 3}
=q^{-\frac{k(n-k)}{3}}
\Big\langle\,\tikz[baseline=-.6ex, scale=0.1]{
\draw[->-=.2] (-5,0) -- (5,0);
\draw[fill=white] (-1,-3) rectangle (0,3);
\node at (1,0) [above right]{${\scriptstyle n}$};
}\,\Big\rangle_{\! 3},$
\item 
$\Big\langle\,\tikz[baseline=-.6ex, scale=0.1]{
\draw[->-=.2] (-4,2) -- (4,2);
\draw[->-=.8, rounded corners=.1cm] (-2,-2) rectangle (2,1);
\draw[fill=white] (-.5,0) rectangle (.5,3);
\node at (0,3)[above]{$\scriptstyle{n}$};
\node at (4,2)[right]{$\scriptstyle{n-k}$};
\node at (0,-2)[below]{$\scriptstyle{k}$};
}\,\Big\rangle_{\! 3}
=\frac{\left[n+1\right]\left[n+2\right]}{\left[n-k+1\right]\left[n-k+2\right]}
\Big\langle\,\tikz[baseline=-.6ex, scale=0.1]{
\draw[->-=.2] (-4,0) -- (4,0);
\draw[fill=white] (-.5,-2) rectangle (.5,2);
\node at (0,2)[above]{$\scriptstyle{n-k}$};
}\,\Big\rangle_{\! 3},$
\item 
$\Big\langle\,\tikz[baseline=-.6ex, scale=0.1]{
\draw[->-=.5] (-4,1) -- (0,1);
\draw[white, double=black, double distance=0.4pt, ultra thick] 
(2,-2) to[out=west, in=west]
(2,1) to[out=east, in=west]
(5,1);
\draw[->-=.8, white, double=black, double distance=0.4pt, ultra thick] 
(0,1) to[out=east, in=west]
(2,1) to[out=east, in=east]
(2,-2);
\draw[fill=white] (-1,-1) rectangle (0,3);
\node at (-1,-1)[below]{$\scriptstyle{n}$};;
}\,\Big\rangle_{\! 3}
=q^{\frac{n^2+3n}{3}}
\Big\langle\,\tikz[baseline=-.6ex, scale=0.1]{
\draw[->-=.2] (-4,0) -- (4,0);
\draw[fill=white] (-.5,-2) rectangle (.5,2);
\node at (0,2)[above]{$\scriptstyle{n}$};
}\,\Big\rangle_{\! 3}$, 
$\Big\langle\,\tikz[baseline=-.6ex, scale=0.1]{
\draw (-4,1) -- (0,1);
\draw[white, double=black, double distance=0.4pt, ultra thick] 
(0,1) to[out=east, in=west]
(2,1) to[out=east, in=east]
(2,-2);
\draw[->-=.8, white, double=black, double distance=0.4pt, ultra thick] 
(2,-2) to[out=west, in=west]
(2,1) to[out=east, in=west]
(5,1);
\draw[fill=white] (-1,-1) rectangle (0,3);
\node at (-1,-1)[below]{$\scriptstyle{n}$};;
}\,\Big\rangle_{\! 3}
=q^{-\frac{n^2+3n}{3}}
\Big\langle\,\tikz[baseline=-.6ex, scale=0.1]{
\draw[->-=.2] (-4,0) -- (4,0);
\draw[fill=white] (-.5,-2) rectangle (.5,2);
\node at (0,2)[above]{$\scriptstyle{n}$};
}\,\Big\rangle_{\! 3}$.
\end{enumerate}
\end{LEM}

\begin{proof}
It is easy to prove (1) -- (3). 
See, for example, \cite{OhtsukiYamada97}.
\end{proof}

\begin{LEM}\label{twistcoeff}\ 
\begin{enumerate}
\item 
$\Big\langle\,\tikz[baseline=-.6ex, scale=0.1]{
\draw[-<-=.5] (-4,2) -- (-1,2);
\draw[->-=.5] (-4,-2) -- (-1,-2);
\draw[-<-=.6, white, double=black, double distance=0.4pt, ultra thick] 
(0,2) to[out=right, in=left] (5,-2) -- (9,-2);
\draw[->-=.6, white, double=black, double distance=0.4pt, ultra thick] 
(0,-2) to[out=right, in=left] (5,2) -- (9,2);
\draw[fill=white] (-1,-3) rectangle (0,3);
\draw[fill=white] (7,1) rectangle (8,3);
\draw[fill=white] (7,-1) rectangle (8,-3);
\draw (-1,0) -- (0,0);
\node at (-2,2) [above]{${\scriptstyle n}$};
\node at (-2,-2) [below]{${\scriptstyle n}$};
\node at (9,2) [right]{${\scriptstyle n}$};
\node at (9,-2) [right]{${\scriptstyle n}$};
}\,\Big\rangle_{\! 3}
=(-1)^nq^{-\frac{n^2}{6}}
\Big\langle\,\tikz[baseline=-.6ex, scale=0.1]{
\draw[-<-=.5] (-4,2) -- (-1,2);
\draw[->-=.5] (-4,-2) -- (-1,-2);
\draw[->-=.7] (0,-2) to[out=east, in=south west] 
(3,0) to[out=north east , in=west] (9,2);
\draw[-<-=.7] (0,2) to[out=east, in=north west] 
(3,0) to[out=south east , in=west] (9,-2);
\draw[fill=white] (-1,-3) rectangle (0,3);
\draw[fill=white] (7,1) rectangle (8,3);
\draw[fill=white] (7,-1) rectangle (8,-3);
\draw[fill=white] (1,0) -- (3,2) -- (5,0) -- (3,-2) -- cycle;
\draw (-1,0) -- (0,0);
\draw (1,0) -- (5,0);
\node at (-2,2) [above]{${\scriptstyle n}$};
\node at (-2,-2) [below]{${\scriptstyle n}$};
\node at (9,2) [right]{${\scriptstyle n}$};
\node at (9,-2) [right]{${\scriptstyle n}$};
}\,\Big\rangle_{\! 3}$
\item[]
and $\Big\langle\,\tikz[baseline=-.6ex, scale=0.1]{
\draw[-<-=.5] (-4,2) -- (-1,2);
\draw[->-=.5] (-4,-2) -- (-1,-2);
\draw[->-=.6, white, double=black, double distance=0.4pt, ultra thick] 
(0,-2) to[out=right, in=left] (5,2) -- (9,2);
\draw[-<-=.6, white, double=black, double distance=0.4pt, ultra thick] 
(0,2) to[out=right, in=left] (5,-2) -- (9,-2);
\draw[fill=white] (-1,-3) rectangle (0,3);
\draw[fill=white] (7,1) rectangle (8,3);
\draw[fill=white] (7,-1) rectangle (8,-3);
\draw (-1,0) -- (0,0);
\node at (-2,2) [above]{${\scriptstyle n}$};
\node at (-2,-2) [below]{${\scriptstyle n}$};
\node at (9,2) [right]{${\scriptstyle n}$};
\node at (9,-2) [right]{${\scriptstyle n}$};
}\,\Big\rangle_{\! 3}
=(-1)^nq^{\frac{n^2}{6}}
\Big\langle\,\tikz[baseline=-.6ex, scale=0.1]{
\draw[-<-=.5] (-4,2) -- (-1,2);
\draw[->-=.5] (-4,-2) -- (-1,-2);
\draw[->-=.7] (0,-2) to[out=east, in=south west] 
(3,0) to[out=north east , in=west] (9,2);
\draw[-<-=.7] (0,2) to[out=east, in=north west] 
(3,0) to[out=south east , in=west] (9,-2);
\draw[fill=white] (-1,-3) rectangle (0,3);
\draw[fill=white] (7,1) rectangle (8,3);
\draw[fill=white] (7,-1) rectangle (8,-3);
\draw[fill=white] (1,0) -- (3,2) -- (5,0) -- (3,-2) -- cycle;
\draw (-1,0) -- (0,0);
\draw (1,0) -- (5,0);
\node at (-2,2) [above]{${\scriptstyle n}$};
\node at (-2,-2) [below]{${\scriptstyle n}$};
\node at (9,2) [right]{${\scriptstyle n}$};
\node at (9,-2) [right]{${\scriptstyle n}$};
}\,\Big\rangle_{\! 3}$
\item 
$\Big\langle\tikz[baseline=-.6ex, scale=0.1]{
\draw[-<-=.5] (-5,0) -- (1,0);
\draw[->-=.3, white, double=black, double distance=0.4pt, ultra thick] 
(0,0) to[out=south east, in=west] (3,-2)
to[out=east, in=west] (8,2) -- (10,2);
\draw[->-=.3, white, double=black, double distance=0.4pt, ultra thick] 
(0,0) to[out=north east, in=west] (3,2) 
to[out=east, in=west] (8,-2) -- (10,-2);
\draw[fill=white] (-4,-2) rectangle (-3,2);
\draw[fill=white] (8,1) rectangle (9,3);
\draw[fill=white] (8,-1) rectangle (9,-3);
\draw[fill=white] (0,-2) -- (0,2) -- (2,0) --cycle;
\node at (-5,0) [left]{${\scriptstyle n}$};
\node at (3,2) [above]{${\scriptstyle n}$};
\node at (3,-2) [below]{${\scriptstyle n}$};
}\Big\rangle_3
=
(-1)^nq^{-\frac{n^2+3n}{6}}\Big\langle\tikz[baseline=-.6ex, scale=0.1]{
\draw[-<-=.5] (-5,0) -- (1,0);
\draw[->-=.5, white, double=black, double distance=0.4pt, ultra thick] 
(0,0) to[out=north east, in=west] (3,2) -- (8,2);
\draw[->-=.5, white, double=black, double distance=0.4pt, ultra thick] 
(0,0) to[out=south east, in=west] (3,-2) -- (8,-2);
\draw[fill=white] (-4,-2) rectangle (-3,2);
\draw[fill=white] (6,1) rectangle (7,3);
\draw[fill=white] (6,-1) rectangle (7,-3);
\draw[fill=white] (0,-2) -- (0,2) -- (2,0) --cycle;
\node at (-5,0) [left]{${\scriptstyle n}$};
\node at (3,2) [above]{${\scriptstyle n}$};
\node at (3,-2) [below]{${\scriptstyle n}$};
}\Big\rangle_3
$
\item[]
and 
$\Big\langle\tikz[baseline=-.6ex, scale=0.1]{
\draw[-<-=.5] (-5,0) -- (1,0);
\draw[->-=.3, white, double=black, double distance=0.4pt, ultra thick] 
(0,0) to[out=north east, in=west] (3,2) 
to[out=east, in=west] (8,-2) -- (10,-2);
\draw[->-=.3, white, double=black, double distance=0.4pt, ultra thick] 
(0,0) to[out=south east, in=west] (3,-2)
to[out=east, in=west] (8,2) -- (10,2);
\draw[fill=white] (-4,-2) rectangle (-3,2);
\draw[fill=white] (8,1) rectangle (9,3);
\draw[fill=white] (8,-1) rectangle (9,-3);
\draw[fill=white] (0,-2) -- (0,2) -- (2,0) --cycle;
\node at (-5,0) [left]{${\scriptstyle n}$};
\node at (3,2) [above]{${\scriptstyle n}$};
\node at (3,-2) [below]{${\scriptstyle n}$};
}\Big\rangle_3
=
(-1)^nq^{\frac{n^2+3n}{6}}\Big\langle\tikz[baseline=-.6ex, scale=0.1]{
\draw[-<-=.5] (-5,0) -- (1,0);
\draw[->-=.5, white, double=black, double distance=0.4pt, ultra thick] 
(0,0) to[out=north east, in=west] (3,2) -- (8,2);
\draw[->-=.5, white, double=black, double distance=0.4pt, ultra thick] 
(0,0) to[out=south east, in=west] (3,-2) -- (8,-2);
\draw[fill=white] (-4,-2) rectangle (-3,2);
\draw[fill=white] (6,1) rectangle (7,3);
\draw[fill=white] (6,-1) rectangle (7,-3);
\draw[fill=white] (0,-2) -- (0,2) -- (2,0) --cycle;
\node at (-5,0) [left]{${\scriptstyle n}$};
\node at (3,2) [above]{${\scriptstyle n}$};
\node at (3,-2) [below]{${\scriptstyle n}$};
}\Big\rangle_3$.
\end{enumerate}
The above equations also hold for the opposite orientations.
\end{LEM}

\begin{proof}
(1) is derived by Lemma~\ref{doubleA2claspprop} and the colored $A_2$ skein relation in \cite{Yuasa17} (see Theorem~\ref{coloredA2} in Sect.~5).
We only prove the first equation of (2) by induction.
It is proven by straightforward calculation for $n=1$. 
Set $C_n=(-1)^nq^{-\frac{n^2+3n}{6}}$,
\begin{align*}
\Big\langle
\tikz[baseline=-.6ex, scale=0.1]{
\draw[-<-=.5] (-5,0) -- (1,0);
\draw[->-=.3, white, double=black, double distance=0.4pt, ultra thick] 
(0,0) to[out=south east, in=west] (3,-2)
to[out=east, in=west] (8,2) -- (10,2);
\draw[->-=.3, white, double=black, double distance=0.4pt, ultra thick] 
(0,0) to[out=north east, in=west] (3,2) 
to[out=east, in=west] (8,-2) -- (10,-2);
\draw[fill=white] (-4,-2) rectangle (-3,2);
\draw[fill=white] (8,1) rectangle (9,3);
\draw[fill=white] (8,-1) rectangle (9,-3);
\draw[fill=white] (0,-2) -- (0,2) -- (2,0) --cycle;
\node at (-5,0) [left]{${\scriptstyle n}$};
\node at (3,2) [above]{${\scriptstyle n}$};
\node at (3,-2) [below]{${\scriptstyle n}$};
}
\Big\rangle_3
&=
\Big\langle
\tikz[baseline=-.6ex, scale=0.1]{
\draw (-5,0) -- (-4,0);
\draw (9,2.5) -- (10,2.5);
\draw (9,-2.5) -- (10,-2.5);
\draw[-<-=.5] (-3,2) -- (1,2);
\draw[->-=.9, white, double=black, double distance=0.4pt, ultra thick] 
(3,-1.5) {[rounded corners]-- (3,-3)} to[out=north east, in=west] (8,3);
\draw[->-=.3, white, double=black, double distance=0.4pt, ultra thick] 
(1,2) -- (1,-2) to[out=south, in=west] (3,-5)
to[out=east, in=west] (8,2);
\draw[->-=.3, white, double=black, double distance=0.4pt, ultra thick] 
(1,2) to[out=north, in=west] (3,5) 
to[out=east, in=west] (8,-2);
\draw[-<-=.3] (-3,-1.5) -- (3,-1.5);
\draw[->-=.9, white, double=black, double distance=0.4pt, ultra thick] 
(3,-1.5) {[rounded corners]-- (3,0)} to[out=south east, in=west] (8,-3);
\draw[fill=white] (0,-1) rectangle (2,-2);
\draw (0,-1) -- (2,-2);
\draw[fill=white] (-4,-3) rectangle (-3,3);
\draw[fill=white] (8,1) rectangle (9,4);
\draw[fill=white] (8,-1) rectangle (9,-4);
\draw[fill=white] (0,1) -- (0,3) -- (2,2) --cycle;
\node at (-5,0) [left]{${\scriptstyle n}$};
\node at (3,5) [above]{${\scriptstyle n-1}$};
\node at (3,-5) [below]{${\scriptstyle n-1}$};
}
\Big\rangle_3
=C_1
\Big\langle
\tikz[baseline=-.6ex, scale=0.1]{
\draw (-5,0) -- (-4,0);
\draw (9,2.5) -- (10,2.5);
\draw (9,-2.5) -- (10,-2.5);
\draw[-<-=.5] (-3,2) -- (1,2);
\draw[->-=.9, white, double=black, double distance=0.4pt, ultra thick] 
(3,-1.5) -- (3,0) to[out=north, in=west] (8,3);
\draw[->-=.3, white, double=black, double distance=0.4pt, ultra thick] 
(1,2) -- (1,-2) to[out=south, in=west] (3,-5)
to[out=east, in=west] (8,2);
\draw[->-=.3, white, double=black, double distance=0.4pt, ultra thick] 
(1,2) to[out=north, in=west] (3,5) 
to[out=east, in=west] (8,-2);
\draw[-<-=.3] (-3,-1.5) -- (3,-1.5);
\draw[->-=.9, white, double=black, double distance=0.4pt, ultra thick] 
(3,-1.5) to[out=south, in=west] (8,-3);
\draw[fill=white] (0,-1) rectangle (2,-2);
\draw (0,-1) -- (2,-2);
\draw[fill=white] (-4,-3) rectangle (-3,3);
\draw[fill=white] (8,1) rectangle (9,4);
\draw[fill=white] (8,-1) rectangle (9,-4);
\draw[fill=white] (0,1) -- (0,3) -- (2,2) --cycle;
\node at (-5,0) [left]{${\scriptstyle n}$};
\node at (3,5) [above]{${\scriptstyle n-1}$};
\node at (3,-5) [below]{${\scriptstyle n-1}$};
}
\Big\rangle_3
=C_1q^{-\frac{2}{3}(n-1)}
\Big\langle
\tikz[baseline=-.6ex, scale=0.1]{
\draw (-5,0) -- (-4,0);
\draw (9,2.5) -- (10,2.5);
\draw (9,-2.5) -- (10,-2.5);
\draw[-<-=.5] (-3,2) -- (1,2);
\draw[white, double=black, double distance=0.4pt, ultra thick] 
(7,-1.5) -- (7,2) -- (8,2);
\draw[->-=.8, white, double=black, double distance=0.4pt, ultra thick] 
(1,2) -- (1,-2) to[out=south, in=west] (2,-4)
to[out=east, in=south] (3,0)
to[out=north, in=west] (8,3);
\draw[-<-=.2, white, double=black, double distance=0.4pt, ultra thick] 
(-3,-1.5) -- (7,-1.5);
\draw[->-=.8, white, double=black, double distance=0.4pt, ultra thick] 
(1,2) to[out=north, in=west] (3,4) 
to[out=east, in=west] (6,-3) -- (8,-3);
\draw[white, double=black, double distance=0.4pt, ultra thick] 
(7,-1.5) -- (7,-2) -- (8,-2);
\draw[fill=white] (0,-1) rectangle (2,-2);
\draw (0,-1) -- (2,-2);
\draw[fill=white] (-4,-3) rectangle (-3,3);
\draw[fill=white] (8,1) rectangle (9,4);
\draw[fill=white] (8,-1) rectangle (9,-4);
\draw[fill=white] (0,1) -- (0,3) -- (2,2) --cycle;
\node at (-5,0) [left]{${\scriptstyle n}$};
\node at (3,4) [above]{${\scriptstyle n-1}$};
\node at (3,-4) [below]{${\scriptstyle n-1}$};
}
\Big\rangle_3\\
&=C_1q^{-\frac{2}{3}(n-1)}(-q^{\frac{1}{6}})^{n-1}
\Big\langle
\tikz[baseline=-.6ex, scale=0.1]{
\draw (-5,0) -- (-4,0);
\draw (9,2.5) -- (10,2.5);
\draw (9,-2.5) -- (10,-2.5);
\draw[-<-=.5] (-3,2) -- (1,2);
\draw[white, double=black, double distance=0.4pt, ultra thick] 
(7,-1.5) -- (7,2) -- (8,2);
\draw[->-=.3, white, double=black, double distance=0.4pt, ultra thick] 
(1,2) to[out=south, in=west] (3,0)
to[out=east, in=west] (6,3) -- (8,3);
\draw[-<-=.2, white, double=black, double distance=0.4pt, ultra thick] 
(-3,-1.5) -- (7,-1.5);
\draw[->-=.2, white, double=black, double distance=0.4pt, ultra thick] 
(1,2) to[out=north, in=west] (3,4) 
to[out=east, in=west] (6,-3) -- (8,-3);
\draw[white, double=black, double distance=0.4pt, ultra thick] 
(7,-1.5) -- (7,-2) -- (8,-2);
\draw[fill=white] (-4,-3) rectangle (-3,3);
\draw[fill=white] (8,1) rectangle (9,4);
\draw[fill=white] (8,-1) rectangle (9,-4);
\draw[fill=white] (0,1) -- (0,3) -- (2,2) --cycle;
\node at (-5,0) [left]{${\scriptstyle n}$};
\node at (3,4) [above]{${\scriptstyle n-1}$};
\node at (3,-4) [below]{${\scriptstyle n-1}$};
}
\Big\rangle_3\\
&=C_1q^{-\frac{2}{3}(n-1)}(-q^{\frac{1}{6}})^{n-1}C_{n-1}
\Big\langle
\tikz[baseline=-.6ex, scale=0.1]{
\draw (-5,0) -- (-4,0);
\draw (9,2.5) -- (10,2.5);
\draw (9,-2.5) -- (10,-2.5);
\draw[-<-=.5] (-3,1) -- (1,1);
\draw[white, double=black, double distance=0.4pt, ultra thick] 
(7,-1.5) -- (7,2) -- (8,2);
\draw[->-=.5, white, double=black, double distance=0.4pt, ultra thick] 
(1,1) to[out=north, in=west] (8,3);
\draw[-<-=.2, white, double=black, double distance=0.4pt, ultra thick] 
(-3,-1.5) -- (7,-1.5);
\draw[->-=.5, white, double=black, double distance=0.4pt, ultra thick, rounded corners] 
(1,1) -- (1,-3) -- (8,-3);
\draw[white, double=black, double distance=0.4pt, ultra thick] 
(7,-1.5) -- (7,-2) -- (8,-2);
\draw[fill=white] (-4,-3) rectangle (-3,3);
\draw[fill=white] (8,1) rectangle (9,4);
\draw[fill=white] (8,-1) rectangle (9,-4);
\draw[fill=white] (0,0) -- (0,2) -- (2,1) --cycle;
\node at (-5,0) [left]{${\scriptstyle n}$};
\node at (3,3) [above]{${\scriptstyle n-1}$};
\node at (3,-4) [below]{${\scriptstyle n-1}$};
}
\Big\rangle_3\\
&=C_1q^{-\frac{2}{3}(n-1)}(-q^{\frac{1}{6}})^{n-1}C_{n-1}(-q^{\frac{1}{6}})^{n-1}
\Big\langle
\tikz[baseline=-.6ex, scale=0.1]{
\draw (-5,0) -- (-4,0);
\draw (9,2.5) -- (10,2.5);
\draw (9,-2.5) -- (10,-2.5);
\draw[-<-=.5] (-3,1) -- (1,1);
\draw[white, double=black, double distance=0.4pt, ultra thick] 
(7,-1.5) -- (7,2) -- (8,2);
\draw[->-=.5, white, double=black, double distance=0.4pt, ultra thick] 
(1,1) to[out=north, in=west] (8,3);
\draw[-<-=.2, white, double=black, double distance=0.4pt, ultra thick] 
(-3,-1.5) -- (7,-1.5);
\draw[->-=.5, white, double=black, double distance=0.4pt, ultra thick, rounded corners] 
(1,1) -- (1,-3) -- (8,-3);
\draw[white, double=black, double distance=0.4pt, ultra thick] 
(7,-1.5) -- (7,-2) -- (8,-2);
\draw[fill=white] (0,-1) rectangle (2,-2);
\draw (0,-1) -- (2,-2);
\draw[fill=white] (-4,-3) rectangle (-3,3);
\draw[fill=white] (8,1) rectangle (9,4);
\draw[fill=white] (8,-1) rectangle (9,-4);
\draw[fill=white] (0,0) -- (0,2) -- (2,1) --cycle;
\node at (-5,0) [left]{${\scriptstyle n}$};
\node at (3,3) [above]{${\scriptstyle n-1}$};
\node at (3,-4) [below]{${\scriptstyle n-1}$};
}
\Big\rangle_3
\end{align*}
The last equation is easily derived by using the $A_2$ skein relation $n-1$ times at the crossing. 
We applied the following calculation to the second line of the above equation.
\begin{align*}
\Big\langle
\tikz[baseline=-.6ex, scale=0.1]{
\draw[rounded corners] (-3.5,5) -- (-3.5,-3) -- (3.5,-3) -- (3.5,5);
\draw[->-=.1, white, double=black, double distance=0.4pt, ultra thick] 
(-7,0) -- (7,0);
\draw[fill=white] (-6,-1) rectangle (-1,1);
\draw (-6,1) -- (-1,-1);
\draw[fill=white] (-6,3) rectangle (-1,4);
\draw[fill=white] (6,3) rectangle (1,4);
\node at (-3.5,5) [above]{${\scriptstyle n-1}$};
\node at (0,-4) [below]{${\scriptstyle n-1}$};
}
\Big\rangle_3
&=
\Big\langle
\tikz[baseline=-.6ex, scale=0.1]{
\draw (-3.5,4) -- (-3.5,5);
\draw (3.5,4) -- (3.5,5);
\draw[rounded corners] (-5,4) -- (-5,-4) -- (5,-4) -- (5,4);
\draw (-2,3) -- (-2,0);
\draw[rounded corners] (-3,0) -- (-3,-2) -- (3,-2) -- (3,0) -- (3,3);
\draw (0,0) -- (-7,0);
\draw[-<-=.1, white, double=black, double distance=0.4pt, ultra thick] 
(0,0) -- (7,0);
\draw[fill=white] (-6,-1) rectangle (-4,1);
\draw (-6,1) -- (-4,-1);
\draw[fill=white] (-6,3) rectangle (-1,4);
\draw[fill=white] (6,3) rectangle (1,4);
\node at (-3.5,5) [above]{${\scriptstyle n-1}$};
\node at (0,-4) [below]{${\scriptstyle n-2}$};
}
\Big\rangle_3
=
q^{-\frac{1}{3}}
\Big\langle
\tikz[baseline=-.6ex, scale=0.1]{
\draw (-3.5,4) -- (-3.5,5);
\draw (3.5,4) -- (3.5,5);
\draw[->-=.5, rounded corners] (-5,4) -- (-5,-4) -- (5,-4) -- (5,4);
\draw (-2,3) -- (-2,0);
\draw[->-=.5, rounded corners] (-3,0) -- (-3,-2) -- (1,-2) -- (1,0) -- (0,0);
\draw (0,0) -- (-7,0);
\draw[-<-=.3, white, double=black, double distance=0.4pt, ultra thick, rounded corners] 
(2,3) -- (2,0) -- (7,0);
\draw[fill=white] (-6,-1) rectangle (-4,1);
\draw (-6,1) -- (-4,-1);
\draw[fill=white] (-6,3) rectangle (-1,4);
\draw[fill=white] (6,3) rectangle (1,4);
\node at (-3.5,5) [above]{${\scriptstyle n-1}$};
\node at (0,-4) [below]{${\scriptstyle n-2}$};
}
\Big\rangle_3
-q^{\frac{1}{6}}
\Big\langle
\tikz[baseline=-.6ex, scale=0.1]{
\draw (-3.5,4) -- (-3.5,5);
\draw (3.5,4) -- (3.5,5);
\draw[->-=.5, rounded corners] (-5,4) -- (-5,-4) -- (5,-4) -- (5,4);
\draw (-2,3) -- (-2,0);
\draw (2,3) -- (2,0);
\draw[->-=.5, rounded corners] (-3,0) -- (-3,-2) -- (3,-2) -- (3,0);
\draw (3,0) -- (-7,0);
\draw[-<-=.3, white, double=black, double distance=0.4pt, ultra thick] (3,0) -- (7,0);
\draw[fill=white] (-6,-1) rectangle (-4,1);
\draw (-6,1) -- (-4,-1);
\draw[fill=white] (-6,3) rectangle (-1,4);
\draw[fill=white] (6,3) rectangle (1,4);
\node at (-3.5,5) [above]{${\scriptstyle n-1}$};
\node at (0,-4) [below]{${\scriptstyle n-2}$};
}
\Big\rangle_3\\
&=
-q^{\frac{1}{6}}
\Big\langle
\tikz[baseline=-.6ex, scale=0.1]{
\draw (-3.5,4) -- (-3.5,5);
\draw (3.5,4) -- (3.5,5);
\draw[->-=.5, rounded corners] (-5,4) -- (-5,-4) -- (5,-4) -- (5,4);
\draw[->-=.5, rounded corners] (-2,3) -- (-2,2) -- (2,2) -- (2,3);
\draw (3,0) -- (-7,0);
\draw[-<-=.3, white, double=black, double distance=0.4pt, ultra thick] (3,0) -- (7,0);
\draw[fill=white] (-6,-1) rectangle (-4,1);
\draw (-6,1) -- (-4,-1);
\draw[fill=white] (-6,3) rectangle (-1,4);
\draw[fill=white] (6,3) rectangle (1,4);
\node at (-3.5,5) [above]{${\scriptstyle n-1}$};
\node at (0,-4) [below]{${\scriptstyle n-2}$};
\node at (0,2) [above]{${\scriptstyle 1}$};
\node at (0,0) [below]{${\scriptstyle 1}$};}
\Big\rangle_3
=\cdots
=(-q^{\frac{1}{6}})^{n-2}
\Big\langle
\tikz[baseline=-.6ex, scale=0.1]{
\draw (-3.5,4) -- (-3.5,5);
\draw (3.5,4) -- (3.5,5);
\draw[->-=.5, rounded corners] (-2,3) -- (-2,2) -- (2,2) -- (2,3);
\draw (-4,3) -- (-4,-1);
\draw[->-=.4, rounded corners] (-5,-1) -- (-5,-3) -- (5,-3) -- (5,-1) -- (5,3);
\draw (0,-1) -- (-7,-1);
\draw[-<-=.1, white, double=black, double distance=0.4pt, ultra thick] 
(0,-1) -- (7,-1);
\draw[fill=white] (-6,3) rectangle (-1,4);
\draw[fill=white] (6,3) rectangle (1,4);
\node at (-3.5,5) [above]{${\scriptstyle n-1}$};
\node at (0,3) [below]{${\scriptstyle n-2}$};
}
\Big\rangle_3\\
&=(-q^{\frac{1}{6}})^{n-1}
\Big\langle
\tikz[baseline=-.6ex, scale=0.1]{
\draw (-3.5,4) -- (-3.5,5);
\draw (3.5,4) -- (3.5,5);
\draw[->-=.5, rounded corners] (-3.5,3) -- (-3.5,2) -- (3.5,2) -- (3.5,3);
\draw[-<-=.1] (-7,-1) -- (7,-1);
\draw[fill=white] (-6,3) rectangle (-1,4);
\draw[fill=white] (6,3) rectangle (1,4);
\node at (0,3) [below]{${\scriptstyle n-1}$};
\node at (0,0) [below]{${\scriptstyle 1}$};
}
\Big\rangle_3\\
&\qquad+((-q^{\frac{1}{6}})^{n-2}q^{-\frac{1}{3}}\left[2\right]+(-q^{\frac{1}{6}})^{n-1})
\Big\langle
\tikz[baseline=-.6ex, scale=0.1]{
\draw (-3.5,4) -- (-3.5,5);
\draw (3.5,4) -- (3.5,5);
\draw[->-=.5, rounded corners] (-2,3) -- (-2,2) -- (2,2) -- (2,3);
\draw[-<-=.5, rounded corners] (-7,-1) -- (-5,-1) -- (-5,3);
\draw[->-=.5, rounded corners] (7,-1) -- (5,-1) -- (5,3);
\draw[fill=white] (-6,3) rectangle (-1,4);
\draw[fill=white] (6,3) rectangle (1,4);
\node at (0,3) [below]{${\scriptstyle n-2}$};
\node at (-6,-1) [below]{${\scriptstyle 1}$};
\node at (6,-1) [below]{${\scriptstyle 1}$};
}
\Big\rangle_3\\
&=(-q^{\frac{1}{6}})^{n-1}
\Big\langle
\tikz[baseline=-.6ex, scale=0.1]{
\draw (-3.5,4) -- (-3.5,5);
\draw (3.5,4) -- (3.5,5);
\draw[->-=.5, rounded corners] (-3.5,3) -- (-3.5,2) -- (3.5,2) -- (3.5,3);
\draw[-<-=.1] (-7,-1) -- (7,-1);
\draw[fill=white] (-6,3) rectangle (-1,4);
\draw[fill=white] (6,3) rectangle (1,4);
\node at (0,3) [below]{${\scriptstyle n-1}$};
\node at (0,0) [below]{${\scriptstyle 1}$};
}
\Big\rangle_3
-(-q^{\frac{1}{6}})^{n-1}(q^{-1})
\Big\langle
\tikz[baseline=-.6ex, scale=0.1]{
\draw (-3.5,4) -- (-3.5,5);
\draw (3.5,4) -- (3.5,5);
\draw[->-=.5, rounded corners] (-2,3) -- (-2,2) -- (2,2) -- (2,3);
\draw[-<-=.5, rounded corners] (-7,-1) -- (-5,-1) -- (-5,3);
\draw[->-=.5, rounded corners] (7,-1) -- (5,-1) -- (5,3);
\draw[fill=white] (-6,3) rectangle (-1,4);
\draw[fill=white] (6,3) rectangle (1,4);
\node at (0,3) [below]{${\scriptstyle n-2}$};
\node at (-6,-1) [below]{${\scriptstyle 1}$};
\node at (6,-1) [below]{${\scriptstyle 1}$};
}
\Big\rangle_3\ .
\end{align*}

We can confirm that the coefficient turns out to be $C_n=(-1)^nq^{-\frac{n^2+3n}{6}}$ and the $A_2$ web in the last term is the same to 
$\Big\langle
\tikz[baseline=-.6ex, scale=0.1]{
\draw[-<-=.5] (-5,0) -- (1,0);
\draw[->-=.5, white, double=black, double distance=0.4pt, ultra thick] 
(0,0) to[out=north east, in=west] (3,2) -- (8,2);
\draw[->-=.5, white, double=black, double distance=0.4pt, ultra thick] 
(0,0) to[out=south east, in=west] (3,-2) -- (8,-2);
\draw[fill=white] (-4,-2) rectangle (-3,2);
\draw[fill=white] (6,1) rectangle (7,3);
\draw[fill=white] (6,-1) rectangle (7,-3);
\draw[fill=white] (0,-2) -- (0,2) -- (2,0) --cycle;
\node at (-5,0) [left]{${\scriptstyle n}$};
\node at (3,2) [above]{${\scriptstyle n}$};
\node at (3,-2) [below]{${\scriptstyle n}$};
}
\Big\rangle_3$ by definition.
\end{proof}

\section{The $A_2$ colored Kauffman-Vogel polynomial}
In this section, 
we will give definitions of invariants of oriented and unoriented $4$-valent rigid vertex graphs by using clasped $A_2$ webs. 
These invariants are a generalization of Kauffman-Vogel polynomials.
Kauffman and Vogel defined three variables polynomial invariants of $4$-valent rigid vertex graphs. 
A one-variable specialization of the Kauffman-Vogel polynomial was given by using the Kauffman bracket and the Jones-Wenzl idempotents (see Chapter~{4.3} in~\cite{KauffmanLins94}). 
The one variable Kauffman-Vogel polynomial was generalized by Elhamdadi and Hajij~\cite{ElhamdadiHajij17}. 
This polynomial is colored by positive even integers. 
Our invariants for oriented and unoriented $4$-valent rigid vertex graphs are an $A_2$ version of the colored one-variable Kauffman-Vogel polynomials.
\subsection{Invariants of oriented $4$-valent rigid vertex graphs}
Let $G$ be an oriented $4$-valent rigid vertex graph diagram. 
\begin{DEF}\label{oriKVdef}
We define $\left[G\right]_{2n}$ by the following rules:
\begin{enumerate}
\item 
$\left[\tikz[baseline=-.6ex, scale=0.1]{
\draw[->-=.5] (-5,0) -- (5,0);
}\right]_{2n}
=
\Big\langle\tikz[baseline=-.6ex, scale=0.1]{
\draw[->-=.8] (-5,0) -- (5,0);
\draw[fill=white] (-1,-3) rectangle (1,3);
\node at (1,0) [above right]{${\scriptstyle 2n}$};
}\Big\rangle_3
$,
\item 
$
\left[\,\tikz[baseline=-.6ex, scale=.1]{
\draw (-4,4) -- +(-2,0);
\draw[->-=.8, white, double=black, double distance=0.4pt, ultra thick] 
(4,-4) to[out=west, in=east] (-4,4);
\draw (4,-4) -- +(2,0);
\draw (-4,-4) -- +(-2,0);
\draw[->-=.8, white, double=black, double distance=0.4pt, ultra thick] 
(-4,-4) to[out=east, in=west] (4,4);
\draw (4,4) -- +(2,0);
}\,\right]_{2n}
=
\Bigg\langle\,\tikz[baseline=-.6ex, scale=.08]{
\draw (-4,4) -- +(-2,0);
\draw[->-=.8, white, double=black, double distance=0.4pt, ultra thick] 
(4,-4) to[out=west, in=east] (-4,4);
\draw (4,-4) -- +(2,0);
\draw (-4,-4) -- +(-2,0);
\draw[->-=.8, white, double=black, double distance=0.4pt, ultra thick] 
(-4,-4) to[out=east, in=west] (4,4);
\draw (4,4) -- +(2,0);
\draw[fill=white] (4,-6) rectangle +(1,4);
\draw[fill=white] (-4,-6) rectangle +(-1,4);
\draw[fill=white] (4,6) rectangle +(1,-4);
\draw[fill=white] (-4,6) rectangle +(-1,-4);
\node at (4,-6)[right]{$\scriptstyle{2n}$};
\node at (-4,-6)[left]{$\scriptstyle{2n}$};
\node at (4,6)[right]{$\scriptstyle{2n}$};
\node at (-4,6)[left]{$\scriptstyle{2n}$};
}\,\Bigg\rangle_{\! 3}$
and 
$\left[\,\tikz[baseline=-.6ex, scale=.1]{
\draw (-4,4) -- +(-2,0);
\draw[->-=.8, white, double=black, double distance=0.4pt, ultra thick] 
(-4,-4) to[out=east, in=west] (4,4);
\draw[->-=.8, white, double=black, double distance=0.4pt, ultra thick] 
(4,-4) to[out=west, in=east] (-4,4);
\draw (4,-4) -- +(2,0);
\draw (-4,-4) -- +(-2,0);
\draw (4,4) -- +(2,0);
}\,\right]_{2n}
=
\Bigg\langle\,\tikz[baseline=-.6ex, scale=.08]{
\draw (-4,4) -- +(-2,0);
\draw[->-=.8, white, double=black, double distance=0.4pt, ultra thick] 
(-4,-4) to[out=east, in=west] (4,4);
\draw[->-=.8, white, double=black, double distance=0.4pt, ultra thick] 
(4,-4) to[out=west, in=east] (-4,4);
\draw (4,-4) -- +(2,0);
\draw (-4,-4) -- +(-2,0);
\draw (4,4) -- +(2,0);
\draw[fill=white] (4,-6) rectangle +(1,4);
\draw[fill=white] (-4,-6) rectangle +(-1,4);
\draw[fill=white] (4,6) rectangle +(1,-4);
\draw[fill=white] (-4,6) rectangle +(-1,-4);
\node at (4,-6)[right]{$\scriptstyle{2n}$};
\node at (-4,-6)[left]{$\scriptstyle{2n}$};
\node at (4,6)[right]{$\scriptstyle{2n}$};
\node at (-4,6)[left]{$\scriptstyle{2n}$};
}\,\Bigg\rangle_{\! 3}$,
\item 
$
\left[\,\tikz[baseline=-.6ex, scale=.1]{
\draw (-4,4) -- +(-2,0);
\draw (4,-4) -- +(2,0);
\draw (-4,-4) -- +(-2,0);
\draw (4,4) -- +(2,0);
\draw[->-=.8] (0,0) -- (2,3) to[out=north east, in=west] (4,4);
\draw[->-=.8] (0,0) -- (-2,3) to[out=north west, in=east] (-4,4);
\draw[-<-=.8] (0,0) -- (2,-3) to[out=south east, in=west] (4,-4);
\draw[-<-=.8] (0,0) -- (-2,-3) to[out=south west, in=east] (-4,-4);
\draw[fill=cyan] (0,0) circle [radius=.8];
}\,\right]_{2n}
=
\Bigg\langle\,\tikz[baseline=-.6ex, scale=.1]{
\draw 
(-5,4) -- +(-2,0)
(5,-4) -- +(2,0)
(-5,-4) -- +(-2,0)
(5,4) -- +(2,0);
\draw[-<-=.5] (-5,3) to[out=east, in=east] (-5,-3);
\draw[-<-=.5] (5,3) to[out=west, in=west] (5,-3);
\draw[-<-=.5] (-5,5) to[out=east, in=north west] (0,0);
\draw[-<-=.5] (0,0) to[out=south east, in=west] (5,-5);
\draw[-<-=.5] (5,5) to[out=west, in=north east] (0,0);
\draw[-<-=.5] (0,0) to[out=south west, in=east] (-5,-5);
\draw[fill=white] (2,0) -- (0,2) -- (-2,0) -- (0,-2) -- cycle;
\draw (-2,0) -- (2,0);
\draw[fill=white] (5,-6) rectangle +(1,4);
\draw[fill=white] (-5,-6) rectangle +(-1,4);
\draw[fill=white] (5,6) rectangle +(1,-4);
\draw[fill=white] (-5,6) rectangle +(-1,-4);
\node at (5,-6)[right]{$\scriptstyle{2n}$};
\node at (-5,-6)[left]{$\scriptstyle{2n}$};
\node at (5,6)[right]{$\scriptstyle{2n}$};
\node at (-5,6)[left]{$\scriptstyle{2n}$};
\node at (-3,0)[left]{$\scriptstyle{n}$};
\node at (3,0)[right]{$\scriptstyle{n}$};
\node at (0,5)[right]{$\scriptstyle{n}$};
\node at (0,5)[left]{$\scriptstyle{n}$};
\node at (0,-5)[right]{$\scriptstyle{n}$};
\node at (0,-5)[left]{$\scriptstyle{n}$};
}\,\Bigg\rangle_{\! 3}
$
and 
$\left[\,\tikz[baseline=-.6ex, scale=.1]{
\draw (-4,4) -- +(-2,0);
\draw (4,-4) -- +(2,0);
\draw (-4,-4) -- +(-2,0);
\draw (4,4) -- +(2,0);
\draw[-<-=.8] (0,0) -- (2,3) to[out=north east, in=west] (4,4);
\draw[->-=.8] (0,0) -- (-2,3) to[out=north west, in=east] (-4,4);
\draw[->-=.8] (0,0) -- (2,-3) to[out=south east, in=west] (4,-4);
\draw[-<-=.8] (0,0) -- (-2,-3) to[out=south west, in=east] (-4,-4);
\draw[fill=cyan] (0,0) circle [radius=.8];
}\,\right]_{2n}
=
\Bigg\langle\,\tikz[baseline=-.6ex, scale=.08]{
\draw 
(-4,4) -- +(-2,0)
(4,-4) -- +(2,0)
(-4,-4) -- +(-2,0)
(4,4) -- +(2,0);
\draw[-<-=.5] (-4,5) -- (4,5);
\draw[->-=.5] (-4,-5) -- (4,-5);
\draw[-<-=.5] (-4,3) to[out=east, in=east] (-4,-3);
\draw[->-=.5] (4,3) to[out=west, in=west] (4,-3);
\draw[fill=white] (4,-6) rectangle +(1,4);
\draw[fill=white] (-4,-6) rectangle +(-1,4);
\draw[fill=white] (4,6) rectangle +(1,-4);
\draw[fill=white] (-4,6) rectangle +(-1,-4);
\node at (4,-6)[right]{$\scriptstyle{2n}$};
\node at (-4,-6)[left]{$\scriptstyle{2n}$};
\node at (4,6)[right]{$\scriptstyle{2n}$};
\node at (-4,6)[left]{$\scriptstyle{2n}$};
\node at (0,5)[above]{$\scriptstyle{n}$};
\node at (0,-5)[below]{$\scriptstyle{n}$};
\node at (-2,0)[left]{$\scriptstyle{n}$};
\node at (2,0)[right]{$\scriptstyle{n}$};
}\,\Bigg\rangle_{\! 3}$.
\end{enumerate}
\end{DEF}

\begin{THM}\label{oriKVthm}
$\left[G\right]_{2n}$ is invariant under the Reidemeister moves (RI) -- (RV).
\end{THM}
\begin{LEM}\label{4valenttwist}\ 
\begin{itemize}
\item 
$\Bigg\langle\,
\tikz[baseline=-.6ex, scale=0.1]{
\draw (0,0) -- (45:5);
\draw (0,0) -- (135:5);
\draw (0,0) -- (-45:5);
\draw (0,0) -- (-135:5);
\draw[-<-=.8, white, double=black, double distance=0.4pt, ultra thick] 
(45:5) to[out=north east, in=north west] (10,-4);
\draw[->-=.8, white, double=black, double distance=0.4pt, ultra thick] 
(-45:5) to[out=south east, in=south west] (10,4);
\draw[fill=white] (-2,0) -- (0,2) -- (2,0) -- (0,-2) -- cycle;
\draw (-2,0) -- (2,0);
\begin{scope}[rotate around={45:(0,0)}]
\draw[fill=white] (3,-1) rectangle (4,1);
\end{scope}
\begin{scope}[rotate around={135:(0,0)}]
\draw[fill=white] (3,-1) rectangle (4,1);
\end{scope}
\begin{scope}[rotate around={-45:(0,0)}]
\draw[fill=white] (3,-1) rectangle (4,1);
\end{scope}
\begin{scope}[rotate around={-135:(0,0)}]
\draw[fill=white] (3,-1) rectangle (4,1);
\end{scope}
\node at (-2,3) [above]{${\scriptstyle n}$};
\node at (-2,-3) [below]{${\scriptstyle n}$};
\node at (2,3) [above]{${\scriptstyle n}$};
\node at (2,-3) [below]{${\scriptstyle n}$};}
\,\Bigg\rangle_{\! 3}
=(-1)^nq^{\frac{n^2+3n}{6}}
\Bigg\langle\,
\tikz[baseline=-.6ex, scale=0.1, yshift=-1cm]{
\draw (0,0) -- (45:5);
\draw (0,0) -- (135:5);
\draw (0,0) -- (-45:5);
\draw (0,0) -- (-135:5);
\draw[-<-=.8, white, double=black, double distance=0.4pt, ultra thick] 
(45:5) to[out=north east, in=south east] (-4,10);
\draw[->-=.8, white, double=black, double distance=0.4pt, ultra thick] 
(135:5) to[out=north west, in=south west] (4,10);
\draw[fill=white] (-2,0) -- (0,2) -- (2,0) -- (0,-2) -- cycle;
\draw (0,-2) -- (0,2);
\begin{scope}[rotate around={45:(0,0)}]
\draw[fill=white] (3,-1) rectangle (4,1);
\end{scope}
\begin{scope}[rotate around={135:(0,0)}]
\draw[fill=white] (3,-1) rectangle (4,1);
\end{scope}
\begin{scope}[rotate around={-45:(0,0)}]
\draw[fill=white] (3,-1) rectangle (4,1);
\end{scope}
\begin{scope}[rotate around={-135:(0,0)}]
\draw[fill=white] (3,-1) rectangle (4,1);
\end{scope}
\node at (-2,3) [above]{${\scriptstyle n}$};
\node at (-2,-3) [below]{${\scriptstyle n}$};
\node at (2,3) [above]{${\scriptstyle n}$};
\node at (2,-3) [below]{${\scriptstyle n}$};}
\,\Bigg\rangle_{\! 3}
=
\Bigg\langle\,
\tikz[baseline=-.6ex, scale=0.1]{
\draw (0,0) -- (45:5);
\draw (0,0) -- (135:5);
\draw (0,0) -- (-45:5);
\draw (0,0) -- (-135:5);
\draw[-<-=.8, white, double=black, double distance=0.4pt, ultra thick] 
(-135:5) to[out=south west, in=south east] (-10,4);
\draw[->-=.8, white, double=black, double distance=0.4pt, ultra thick] 
(135:5) to[out=north west, in=north east] (-10,-4);
\draw[fill=white] (-2,0) -- (0,2) -- (2,0) -- (0,-2) -- cycle;
\draw (-2,0) -- (2,0);
\begin{scope}[rotate around={45:(0,0)}]
\draw[fill=white] (3,-1) rectangle (4,1);
\end{scope}
\begin{scope}[rotate around={135:(0,0)}]
\draw[fill=white] (3,-1) rectangle (4,1);
\end{scope}
\begin{scope}[rotate around={-45:(0,0)}]
\draw[fill=white] (3,-1) rectangle (4,1);
\end{scope}
\begin{scope}[rotate around={-135:(0,0)}]
\draw[fill=white] (3,-1) rectangle (4,1);
\end{scope}
\node at (-2,3) [above]{${\scriptstyle n}$};
\node at (-2,-3) [below]{${\scriptstyle n}$};
\node at (2,3) [above]{${\scriptstyle n}$};
\node at (2,-3) [below]{${\scriptstyle n}$};}
\,\Bigg\rangle_{\! 3}
$
\item 
$\Bigg\langle\,
\tikz[baseline=-.6ex, scale=0.1]{
\draw (0,0) -- (45:5);
\draw (0,0) -- (135:5);
\draw (0,0) -- (-45:5);
\draw (0,0) -- (-135:5);
\draw[->-=.8, white, double=black, double distance=0.4pt, ultra thick] 
(-45:5) to[out=south east, in=south west] (10,4);
\draw[-<-=.8, white, double=black, double distance=0.4pt, ultra thick] 
(45:5) to[out=north east, in=north west] (10,-4);
\draw[fill=white] (-2,0) -- (0,2) -- (2,0) -- (0,-2) -- cycle;
\draw (-2,0) -- (2,0);
\begin{scope}[rotate around={45:(0,0)}]
\draw[fill=white] (3,-1) rectangle (4,1);
\end{scope}
\begin{scope}[rotate around={135:(0,0)}]
\draw[fill=white] (3,-1) rectangle (4,1);
\end{scope}
\begin{scope}[rotate around={-45:(0,0)}]
\draw[fill=white] (3,-1) rectangle (4,1);
\end{scope}
\begin{scope}[rotate around={-135:(0,0)}]
\draw[fill=white] (3,-1) rectangle (4,1);
\end{scope}
\node at (-2,3) [above]{${\scriptstyle n}$};
\node at (-2,-3) [below]{${\scriptstyle n}$};
\node at (2,3) [above]{${\scriptstyle n}$};
\node at (2,-3) [below]{${\scriptstyle n}$};}
\,\Bigg\rangle_{\! 3}
=(-1)^nq^{-\frac{n^2+3n}{6}}
\Bigg\langle\,
\tikz[baseline=-.6ex, scale=0.1, yshift=-1cm]{
\draw (0,0) -- (45:5);
\draw (0,0) -- (135:5);
\draw (0,0) -- (-45:5);
\draw (0,0) -- (-135:5);
\draw[->-=.8, white, double=black, double distance=0.4pt, ultra thick] 
(135:5) to[out=north west, in=south west] (4,10);
\draw[-<-=.8, white, double=black, double distance=0.4pt, ultra thick] 
(45:5) to[out=north east, in=south east] (-4,10);
\draw[fill=white] (-2,0) -- (0,2) -- (2,0) -- (0,-2) -- cycle;
\draw (0,-2) -- (0,2);
\begin{scope}[rotate around={45:(0,0)}]
\draw[fill=white] (3,-1) rectangle (4,1);
\end{scope}
\begin{scope}[rotate around={135:(0,0)}]
\draw[fill=white] (3,-1) rectangle (4,1);
\end{scope}
\begin{scope}[rotate around={-45:(0,0)}]
\draw[fill=white] (3,-1) rectangle (4,1);
\end{scope}
\begin{scope}[rotate around={-135:(0,0)}]
\draw[fill=white] (3,-1) rectangle (4,1);
\end{scope}
\node at (-2,3) [above]{${\scriptstyle n}$};
\node at (-2,-3) [below]{${\scriptstyle n}$};
\node at (2,3) [above]{${\scriptstyle n}$};
\node at (2,-3) [below]{${\scriptstyle n}$};}
\,\Bigg\rangle_{\! 3}
=
\Bigg\langle\,
\tikz[baseline=-.6ex, scale=0.1]{
\draw (0,0) -- (45:5);
\draw (0,0) -- (135:5);
\draw (0,0) -- (-45:5);
\draw (0,0) -- (-135:5);
\draw[->-=.8, white, double=black, double distance=0.4pt, ultra thick] 
(135:5) to[out=north west, in=north east] (-10,-4);
\draw[-<-=.8, white, double=black, double distance=0.4pt, ultra thick] 
(-135:5) to[out=south west, in=south east] (-10,4);
\draw[fill=white] (-2,0) -- (0,2) -- (2,0) -- (0,-2) -- cycle;
\draw (-2,0) -- (2,0);
\begin{scope}[rotate around={45:(0,0)}]
\draw[fill=white] (3,-1) rectangle (4,1);
\end{scope}
\begin{scope}[rotate around={135:(0,0)}]
\draw[fill=white] (3,-1) rectangle (4,1);
\end{scope}
\begin{scope}[rotate around={-45:(0,0)}]
\draw[fill=white] (3,-1) rectangle (4,1);
\end{scope}
\begin{scope}[rotate around={-135:(0,0)}]
\draw[fill=white] (3,-1) rectangle (4,1);
\end{scope}
\node at (-2,3) [above]{${\scriptstyle n}$};
\node at (-2,-3) [below]{${\scriptstyle n}$};
\node at (2,3) [above]{${\scriptstyle n}$};
\node at (2,-3) [below]{${\scriptstyle n}$};}
\,\Bigg\rangle_{\! 3}$
\end{itemize}
\end{LEM}

\begin{proof}
By using the Reidemeister moves for tangled trivalent graph diagrams and Lemma~\ref{twistcoeff}~(2), 
\begin{align*}
\Bigg\langle\,
\tikz[baseline=-.6ex, scale=0.1, yshift=-1cm]{
\draw (0,0) -- (45:5);
\draw (0,0) -- (135:5);
\draw (0,0) -- (-45:5);
\draw (0,0) -- (-135:5);
\draw[-<-=.8, white, double=black, double distance=0.4pt, ultra thick] 
(45:5) to[out=north east, in=south east] (-4,10);
\draw[->-=.8, white, double=black, double distance=0.4pt, ultra thick] 
(135:5) to[out=north west, in=south west] (4,10);
\draw[fill=white] (-2,0) -- (0,2) -- (2,0) -- (0,-2) -- cycle;
\draw (0,-2) -- (0,2);
\begin{scope}[rotate around={45:(0,0)}]
\draw[fill=white] (3,-1) rectangle (4,1);
\end{scope}
\begin{scope}[rotate around={135:(0,0)}]
\draw[fill=white] (3,-1) rectangle (4,1);
\end{scope}
\begin{scope}[rotate around={-45:(0,0)}]
\draw[fill=white] (3,-1) rectangle (4,1);
\end{scope}
\begin{scope}[rotate around={-135:(0,0)}]
\draw[fill=white] (3,-1) rectangle (4,1);
\end{scope}
\node at (-2,3) [above]{${\scriptstyle n}$};
\node at (-2,-3) [below]{${\scriptstyle n}$};
\node at (2,3) [above]{${\scriptstyle n}$};
\node at (2,-3) [below]{${\scriptstyle n}$};}
\,\Bigg\rangle_{\! 3}
&=\Bigg\langle\,
\tikz[baseline=-.6ex, scale=.1]{
\draw (4,9) -- (4,10);
\draw (-4,9) -- (-4,10);
\draw (3,-9) -- (3,-10);
\draw (-3,-9) -- (-3,-10);
\draw[-<-=.4, white, double=black, double distance=0.4pt, ultra thick] 
(2,0) to[out=north east, in=south] (3,3)
to[out=north, in=south] (-4,8);
\draw[->-=.4, white, double=black, double distance=0.4pt, ultra thick] 
(-2,0) to[out=north west, in=south] (-3,3)
to[out=north, in=south] (4,8);
\draw[-<-=.5, white, double=black, double distance=0.4pt, ultra thick] 
(2,0) to[out=south east, in=north] (3,-3) -- (3,-8);
\draw[->-=.5, white, double=black, double distance=0.4pt, ultra thick] 
(-2,0) to[out=south west, in=north] (-3,-3) -- (-3,-8);
\draw (-2,0) -- (2,0);
\draw[fill=white] (4,0) -- (2,2) -- (2,-2) -- cycle; 
\draw[fill=white] (-4,0) -- (-2,2) -- (-2,-2) -- cycle;
\draw[fill=white] (2,8) rectangle (6,9);
\draw[fill=white] (-2,8) rectangle (-6,9);
\draw[fill=white] (1,-8) rectangle (5,-9);
\draw[fill=white] (-1,-8) rectangle (-5,-9);
\draw[fill=white] (-.5,-2) rectangle (.5,2);
\node at (-3,3)[left]{$\scriptstyle{n}$};
\node at (3,3)[right]{$\scriptstyle{n}$};
\node at (-3,-3)[left]{$\scriptstyle{n}$};
\node at (3,-3)[right]{$\scriptstyle{n}$};
}
\,\Bigg\rangle_{\! 3}
=\Bigg\langle\,
\tikz[baseline=-.6ex, scale=.1]{
\draw (4,9) -- (4,10);
\draw (-4,9) -- (-4,10);
\draw (3,-9) -- (3,-10);
\draw (-3,-9) -- (-3,-10);
\draw[-<-=.5, white, double=black, double distance=0.4pt, ultra thick] 
(2,-2) to[out=north, in=south] (3,0)
to[out=north, in=south] (-4,5) -- (-4,8);
\draw[->-=.8, white, double=black, double distance=0.4pt, ultra thick] 
(3,6) to[out=north east, in=south] (4,8);
\draw[-<-=.5, white, double=black, double distance=0.4pt, ultra thick] 
(2,-2) to[out=south, in=north] (3,-4) -- (3,-8);
\draw[->-=.5, white, double=black, double distance=0.4pt, ultra thick] 
(3,6) to[out=west, in=north] (-3,2) -- (-3,-8);
\draw[-<-=.5, white, double=black, double distance=0.4pt, ultra thick, rounded corners] 
(2,-2) -- (-2,-2) -- (-2,1) -- (3,1) -- (3,6);
\draw[fill=white] (3,-2) -- (1,0) -- (1,-4) -- cycle; 
\draw[fill=white] (3,7) -- (1,5) -- (5,5) -- cycle;
\draw[fill=white] (2,8) rectangle (6,9);
\draw[fill=white] (-2,8) rectangle (-6,9);
\draw[fill=white] (1,-8) rectangle (5,-9);
\draw[fill=white] (-1,-8) rectangle (-5,-9);
\draw[fill=white] (1,3) rectangle (5,4);
}
\,\Bigg\rangle_{\! 3}\\
&=(-1)^nq^{-\frac{n^2+3n}{6}}\Bigg\langle\,
\tikz[baseline=-.6ex, scale=.1]{
\draw (4,9) -- (4,10);
\draw (-4,9) -- (-4,10);
\draw (3,-9) -- (3,-10);
\draw (-3,-9) -- (-3,-10);
\draw[-<-=.5, white, double=black, double distance=0.4pt, ultra thick, rounded corners] 
(2,-2) -- (-1,-2) -- (-1,3) -- (-4,5) -- (-4,8);
\draw[->-=.8, white, double=black, double distance=0.4pt, ultra thick] 
(3,6) to[out=north east, in=south] (4,8);
\draw[-<-=.5, white, double=black, double distance=0.4pt, ultra thick] 
(2,-2) to[out=south, in=north] (3,-4) -- (3,-8);
\draw[->-=.5, white, double=black, double distance=0.4pt, ultra thick] 
(3,6) to[out=west, in=north] (-3,2) -- (-3,-8);
\draw[-<-=.5, white, double=black, double distance=0.4pt, ultra thick] 
(2,-2) to[out=north east, in=south] (3,3) -- (3,6);
\draw[fill=white] (3,-2) -- (1,0) -- (1,-4) -- cycle; 
\draw[fill=white] (3,7) -- (1,5) -- (5,5) -- cycle;
\draw[fill=white] (2,8) rectangle (6,9);
\draw[fill=white] (-2,8) rectangle (-6,9);
\draw[fill=white] (1,-8) rectangle (5,-9);
\draw[fill=white] (-1,-8) rectangle (-5,-9);
\draw[fill=white] (1,3) rectangle (5,4);
}
\,\Bigg\rangle_{\! 3}
=(-1)^nq^{-\frac{n^2+3n}{6}}\Bigg\langle\,
\tikz[baseline=-.6ex, scale=0.1]{
\draw (0,0) -- (45:5);
\draw (0,0) -- (135:5);
\draw (0,0) -- (-45:5);
\draw (0,0) -- (-135:5);
\draw[-<-=.8, white, double=black, double distance=0.4pt, ultra thick] 
(-135:5) to[out=south west, in=south east] (-10,4);
\draw[->-=.8, white, double=black, double distance=0.4pt, ultra thick] 
(135:5) to[out=north west, in=north east] (-10,-4);
\draw[fill=white] (-2,0) -- (0,2) -- (2,0) -- (0,-2) -- cycle;
\draw (-2,0) -- (2,0);
\begin{scope}[rotate around={45:(0,0)}]
\draw[fill=white] (3,-1) rectangle (4,1);
\end{scope}
\begin{scope}[rotate around={135:(0,0)}]
\draw[fill=white] (3,-1) rectangle (4,1);
\end{scope}
\begin{scope}[rotate around={-45:(0,0)}]
\draw[fill=white] (3,-1) rectangle (4,1);
\end{scope}
\begin{scope}[rotate around={-135:(0,0)}]
\draw[fill=white] (3,-1) rectangle (4,1);
\end{scope}
\node at (-2,3) [above]{${\scriptstyle n}$};
\node at (-2,-3) [below]{${\scriptstyle n}$};
\node at (2,3) [above]{${\scriptstyle n}$};
\node at (2,-3) [below]{${\scriptstyle n}$};}
\,\Bigg\rangle_{\! 3}\ .
\end{align*}
The other identities are also proven in the same way.
\end{proof}

\begin{proof}[Proof of Theorem~\ref{oriKVthm}]
The invariance under (RI) -- (RIV) is guaranteed by the invariance of $A_2$ webs under the Reidemeister moves (R1) -- (R4) for tangled trivalent graph diagrams. 
Thus we show the invariance under the first move of (RV):
\[
\left[\tikz[baseline=-.6ex, scale=0.1]{
\draw[->-=.5] (0,0) -- (45:5);
\draw[->-=.5] (0,0) -- (135:5); 
\draw[-<-=.5] (0,0) -- (-45:5);
\draw[-<-=.5] (0,0) -- (-135:5); 
\draw[fill=cyan] (0,0) circle [radius=.8];}\right]_{2n}
=
\left[\tikz[baseline=-.6ex, scale=0.1]{
\draw[->-=.9, white, double=black, double distance=0.4pt, ultra thick, rounded corners] 
(0,0) -- (45:3) -- (135:5);
\draw[->-=.9, white, double=black, double distance=0.4pt, ultra thick, rounded corners] 
(0,0) -- (135:3) -- (45:5);
\draw[-<-=.9, white, double=black, double distance=0.4pt, ultra thick, rounded corners] 
(0,0) -- (-45:3) -- (-135:5);
\draw[-<-=.9, white, double=black, double distance=0.4pt, ultra thick, rounded corners] 
(0,0) -- (-135:3) -- (-45:5);
\draw[fill=cyan] (0,0) circle [radius=.8];}\right]_{2n},\ 
\left[\tikz[baseline=-.6ex, scale=0.1]{
\draw[-<-=.5] (0,0) -- (45:5);
\draw[->-=.5] (0,0) -- (135:5); 
\draw[-<-=.5] (0,0) -- (-45:5);
\draw[->-=.5] (0,0) -- (-135:5); 
\draw[fill=cyan] (0,0) circle [radius=.8];}\right]_{2n}
=
\left[\tikz[baseline=-.6ex, scale=0.1]{
\draw[->-=.9, white, double=black, double distance=0.4pt, ultra thick, rounded corners] 
(0,0) -- (45:3) -- (135:5);
\draw[-<-=.9, white, double=black, double distance=0.4pt, ultra thick, rounded corners] 
(0,0) -- (135:3) -- (45:5);
\draw[->-=.9, white, double=black, double distance=0.4pt, ultra thick, rounded corners] 
(0,0) -- (-45:3) -- (-135:5);
\draw[-<-=.9, white, double=black, double distance=0.4pt, ultra thick, rounded corners] 
(0,0) -- (-135:3) -- (-45:5);
\draw[fill=cyan] (0,0) circle [radius=.8];}\right]_{2n}, 
\text{\ and\ }
\left[\tikz[baseline=-.6ex, scale=0.1]{
\draw[-<-=.5] (0,0) -- (45:5);
\draw[->-=.5] (0,0) -- (135:5); 
\draw[->-=.5] (0,0) -- (-45:5);
\draw[-<-=.5] (0,0) -- (-135:5); 
\draw[fill=cyan] (0,0) circle [radius=.8];}\right]_{2n}
=
\left[\tikz[baseline=-.6ex, scale=0.1]{
\draw[->-=.9, white, double=black, double distance=0.4pt, ultra thick, rounded corners] 
(0,0) -- (45:3) -- (135:5);
\draw[-<-=.9, white, double=black, double distance=0.4pt, ultra thick, rounded corners] 
(0,0) -- (135:3) -- (45:5);
\draw[-<-=.9, white, double=black, double distance=0.4pt, ultra thick, rounded corners] 
(0,0) -- (-45:3) -- (-135:5);
\draw[->-=.9, white, double=black, double distance=0.4pt, ultra thick, rounded corners] 
(0,0) -- (-135:3) -- (-45:5);
\draw[fill=cyan] (0,0) circle [radius=.8];}\right]_{2n}\ .\]
Other cases can be obtained by changing the orientation of the edges or the over/under information at the crossing points in the above diagrams. 
These cases can be proven in the same way as the proof of the above equations.
Therefore, we only show the above three equations.
Let us denote the first equation of (2) in Lemma~\ref{twistcoeff} by $C_n=(-1)^nq^{-\frac{n^2+3n}{6}}$.
\begin{align*}
\left[\tikz[baseline=-.6ex, scale=0.1]{
\draw[->-=.9, white, double=black, double distance=0.4pt, ultra thick, rounded corners] 
(0,0) -- (45:3) -- (135:5);
\draw[->-=.9, white, double=black, double distance=0.4pt, ultra thick, rounded corners] 
(0,0) -- (135:3) -- (45:5);
\draw[-<-=.9, white, double=black, double distance=0.4pt, ultra thick, rounded corners] 
(0,0) -- (-45:3) -- (-135:5);
\draw[-<-=.9, white, double=black, double distance=0.4pt, ultra thick, rounded corners] 
(0,0) -- (-135:3) -- (-45:5);
\draw[fill=cyan] (0,0) circle [radius=.8];}\right]_{2n}
&=
\Bigg\langle\,
\tikz[baseline=-.6ex, scale=.1]{
\draw[->-=.2, white, double=black, double distance=0.4pt, ultra thick]
(3.5,5) to[out=north, in=south] (-3.5,10) -- (-3.5,12);
\draw[->-=.2, white, double=black, double distance=0.4pt, ultra thick]
(-3.5,5) to[out=north, in=south] (3.5,10) -- (3.5,12);
\draw[-<-=.2, white, double=black, double distance=0.4pt, ultra thick]
(3.5,-5) to[out=south, in=north] (-3.5,-10) -- (-3.5,-12);
\draw[-<-=.2, white, double=black, double distance=0.4pt, ultra thick]
(-3.5,-5) to[out=south, in=north] (3.5,-10) -- (3.5,-12);
\draw[-<-=.5] (-4,4) -- (-4,-4);
\draw[-<-=.5] (4,4) -- (4,-4);
\draw[->-=.5] (0,0) to[out=north east, in=south] (3,4);
\draw[->-=.5] (0,0) to[out=north west, in=south] (-3,4);
\draw[-<-=.5] (0,0) to[out=south east, in=north] (3,-4);
\draw[-<-=.5] (0,0) to[out=south west, in=north] (-3,-4);
\draw[fill=white] (2,0) -- (0,2) -- (-2,0) -- (0,-2) -- cycle;
\draw (-2,0) -- (2,0);
\draw[fill=white] (2,10) rectangle (5,11);
\draw[fill=white] (-2,10) rectangle (-5,11);
\draw[fill=white] (2,-10) rectangle (5,-11);
\draw[fill=white] (-2,-10) rectangle (-5,-11);
\draw[fill=white] (2,4) rectangle (5,5);
\draw[fill=white] (-2,4) rectangle (-5,5);
\draw[fill=white] (2,-4) rectangle (5,-5);
\draw[fill=white] (-2,-4) rectangle (-5,-5);
\node at (4,-6)[right]{$\scriptstyle{2n}$};
\node at (-4,-6)[left]{$\scriptstyle{2n}$};
\node at (4,6)[right]{$\scriptstyle{2n}$};
\node at (-4,6)[left]{$\scriptstyle{2n}$};
\node at (-4,0)[left]{$\scriptstyle{n}$};
\node at (4,0)[right]{$\scriptstyle{n}$};
}
\,\Bigg\rangle_{\! 3}
=\Bigg\langle\,
\tikz[baseline=-.6ex, scale=.1]{
\draw (-3.5,10) -- (-3.5,12);
\draw (3.5,10) -- (3.5,12);
\draw (-3.5,-10) -- (-3.5,-12);
\draw (3.5,-10) -- (3.5,-12);
\draw[triple={[line width=1.4pt, white] in [line width=2.2pt, black] in [line width=5.4pt, white]}]
(3.5,5) to[out=north, in=south] (-3.5,10);
\draw[triple={[line width=1.4pt, white] in [line width=2.2pt, black] in [line width=5.4pt, white]}]
(-3.5,5) to[out=north, in=south] (3.5,10);
\draw[triple={[line width=1.4pt, white] in [line width=2.2pt, black] in [line width=5.4pt, white]}]
(3.5,-5) to[out=south, in=north] (-3.5,-10);
\draw[triple={[line width=1.4pt, white] in [line width=2.2pt, black] in [line width=5.4pt, white]}]
(-3.5,-5) to[out=south, in=north] (3.5,-10);
\draw[-<-=.5] (-3.8,5) -- (-3.8,-5);
\draw[-<-=.5] (3.8,5) -- (3.8,-5);
\draw[->-=.5] (0,0) to[out=north east, in=south] (3.2,5);
\draw[->-=.5] (0,0) to[out=north west, in=south] (-3.2,5);
\draw[-<-=.5] (0,0) to[out=south east, in=north] (3.2,-5);
\draw[-<-=.5] (0,0) to[out=south west, in=north] (-3.2,-5);
\draw[fill=white] (2,0) -- (0,2) -- (-2,0) -- (0,-2) -- cycle;
\draw (-2,0) -- (2,0);
\draw[fill=white] (2,10) rectangle (5,11);
\draw[fill=white] (-2,10) rectangle (-5,11);
\draw[fill=white] (2,-10) rectangle (5,-11);
\draw[fill=white] (-2,-10) rectangle (-5,-11);
\node at (-4,0)[left]{$\scriptstyle{n}$};
\node at (4,0)[right]{$\scriptstyle{n}$};
}
\,\Bigg\rangle_{\! 3}
=\Bigg\langle\,
\tikz[baseline=-.6ex, scale=.1]{
\draw (3,9) -- (3,10);
\draw (-3,9) -- (-3,10);
\draw (3,-9) -- (3,-10);
\draw (-3,-9) -- (-3,-10);
\draw[-<-=.5, white, double=black, double distance=0.4pt, ultra thick] 
(-2,8) to[out=south, in=north]
(-4,4) -- (-4,-4) 
to[out=south, in=north] (-2,-8);
\draw[->-=.4, white, double=black, double distance=0.4pt, ultra thick] 
(0,0) to[out=north east, in=south] (2,3)
to[out=north, in=south] (-4,8);
\draw[->-=.4, white, double=black, double distance=0.4pt, ultra thick] 
(0,0) to[out=north west, in=south] (-2,3)
to[out=north, in=south] (4,8);
\draw[-<-=.4, white, double=black, double distance=0.4pt, ultra thick] 
(0,0) to[out=south east, in=north] (2,-3)
to[out=south, in=north] (-4,-8);
\draw[-<-=.4, white, double=black, double distance=0.4pt, ultra thick] 
(0,0) to[out=south west, in=north] (-2,-3)
to[out=south, in=north] (4,-8);
\draw[-<-=.5, white, double=black, double distance=0.4pt, ultra thick] 
(2,8) to[out=south, in=north]
(4,4) -- (4,-4)
to[out=south, in=north] (2,-8);
\draw[fill=white] (2,0) -- (0,2) -- (-2,0) -- (0,-2) -- cycle;
\draw (-2,0) -- (2,0);
\draw[fill=white] (1,8) rectangle (5,9);
\draw[fill=white] (-1,8) rectangle (-5,9);
\draw[fill=white] (1,-8) rectangle (5,-9);
\draw[fill=white] (-1,-8) rectangle (-5,-9);
\node at (-4,0)[left]{$\scriptstyle{n}$};
\node at (4,0)[right]{$\scriptstyle{n}$};
}
\,\Bigg\rangle_{\! 3}
\\ 
&=(q^{\frac{n^2}{3}})^2(q^{-\frac{n^2}{3}})^2\Bigg\langle\,
\tikz[baseline=-.6ex, scale=.1]{
\draw (3,9) -- (3,10);
\draw (-3,9) -- (-3,10);
\draw (3,-9) -- (3,-10);
\draw (-3,-9) -- (-3,-10);
\draw[-<-=.5, white, double=black, double distance=0.4pt, ultra thick] 
(-4,8) -- (-4,-8); 
\draw[->-=.4, white, double=black, double distance=0.4pt, ultra thick] 
(0,0) to[out=north east, in=south] (2,3)
to[out=north, in=south] (-2,8);
\draw[->-=.4, white, double=black, double distance=0.4pt, ultra thick] 
(0,0) to[out=north west, in=south] (-2,3)
to[out=north, in=south] (2,8);
\draw[-<-=.4, white, double=black, double distance=0.4pt, ultra thick] 
(0,0) to[out=south east, in=north] (2,-3)
to[out=south, in=north] (-2,-8);
\draw[-<-=.4, white, double=black, double distance=0.4pt, ultra thick] 
(0,0) to[out=south west, in=north] (-2,-3)
to[out=south, in=north] (2,-8);
\draw[-<-=.5, white, double=black, double distance=0.4pt, ultra thick] 
(4,8) -- (4,-8);
\draw[fill=white] (2,0) -- (0,2) -- (-2,0) -- (0,-2) -- cycle;
\draw (-2,0) -- (2,0);
\draw[fill=white] (1,8) rectangle (5,9);
\draw[fill=white] (-1,8) rectangle (-5,9);
\draw[fill=white] (1,-8) rectangle (5,-9);
\draw[fill=white] (-1,-8) rectangle (-5,-9);
\node at (-4,0)[left]{$\scriptstyle{n}$};
\node at (4,0)[right]{$\scriptstyle{n}$};
}
\,\Bigg\rangle_{\! 3}
=\Bigg\langle\,
\tikz[baseline=-.6ex, scale=.1]{
\draw (3,9) -- (3,10);
\draw (-3,9) -- (-3,10);
\draw (3,-9) -- (3,-10);
\draw (-3,-9) -- (-3,-10);
\draw[-<-=.5, white, double=black, double distance=0.4pt, ultra thick] 
(-4,8) -- (-4,-8); 
\draw[-<-=.5, white, double=black, double distance=0.4pt, ultra thick] 
(4,8) -- (4,-8);
\draw[->-=.4, white, double=black, double distance=0.4pt, ultra thick] 
(0,2) to[out=north east, in=south] (2,5)
to[out=north, in=south] (-2,8);
\draw[->-=.4, white, double=black, double distance=0.4pt, ultra thick] 
(0,2) to[out=north west, in=south] (-2,5)
to[out=north, in=south] (2,8);
\draw[-<-=.4, white, double=black, double distance=0.4pt, ultra thick] 
(0,-2) to[out=south east, in=north] (2,-5)
to[out=south, in=north] (-2,-8);
\draw[-<-=.4, white, double=black, double distance=0.4pt, ultra thick] 
(0,-2) to[out=south west, in=north] (-2,-5)
to[out=south, in=north] (2,-8);
\draw (0,2) -- (0,-2);
\draw[fill=white] (2,2) -- (0,4) -- (-2,2) -- cycle;
\draw[fill=white] (2,-2) -- (0,-4) -- (-2,-2) -- cycle;
\draw (-2,0) -- (2,0);
\draw[fill=white] (1,8) rectangle (5,9);
\draw[fill=white] (-1,8) rectangle (-5,9);
\draw[fill=white] (1,-8) rectangle (5,-9);
\draw[fill=white] (-1,-8) rectangle (-5,-9);
\draw[fill=white] (-2,-.5) rectangle (2,.5);
\node at (-4,0)[left]{$\scriptstyle{n}$};
\node at (4,0)[right]{$\scriptstyle{n}$};
}
\,\Bigg\rangle_{\! 3}\\
&=C_nC_n^{-1}\Bigg\langle\,
\tikz[baseline=-.6ex, scale=.1]{
\draw (3,6) -- (3,7);
\draw (-3,6) -- (-3,7);
\draw (3,-6) -- (3,-7);
\draw (-3,-6) -- (-3,-7);
\draw[-<-=.5, white, double=black, double distance=0.4pt, ultra thick] 
(-4,5) -- (-4,-5); 
\draw[-<-=.5, white, double=black, double distance=0.4pt, ultra thick] 
(4,5) -- (4,-5);
\draw[->-=.8, white, double=black, double distance=0.4pt, ultra thick] 
(0,2) to[out=north east, in=south] (2,5);
\draw[->-=.8, white, double=black, double distance=0.4pt, ultra thick] 
(0,2) to[out=north west, in=south] (-2,5);
\draw[-<-=.8, white, double=black, double distance=0.4pt, ultra thick] 
(0,-2) to[out=south east, in=north] (2,-5);
\draw[-<-=.8, white, double=black, double distance=0.4pt, ultra thick] 
(0,-2) to[out=south west, in=north] (-2,-5);
\draw (0,2) -- (0,-2);
\draw[fill=white] (2,2) -- (0,4) -- (-2,2) -- cycle;
\draw[fill=white] (2,-2) -- (0,-4) -- (-2,-2) -- cycle;
\draw (-2,0) -- (2,0);
\draw[fill=white] (1,5) rectangle (5,6);
\draw[fill=white] (-1,5) rectangle (-5,6);
\draw[fill=white] (1,-5) rectangle (5,-6);
\draw[fill=white] (-1,-5) rectangle (-5,-6);
\draw[fill=white] (-2,-.5) rectangle (2,.5);
\node at (-4,0)[left]{$\scriptstyle{n}$};
\node at (4,0)[right]{$\scriptstyle{n}$};
}
\,\Bigg\rangle_{\! 3}
=\left[
\tikz[baseline=-.6ex, scale=0.1]{
\draw[->-=.5] (0,0) -- (45:5);
\draw[->-=.5] (0,0) -- (135:5); 
\draw[-<-=.5] (0,0) -- (-45:5);
\draw[-<-=.5] (0,0) -- (-135:5); 
\draw[fill=cyan] (0,0) circle [radius=.8];}
\right]_{2n}.
\end{align*}
We used Lemma~\ref{A2clasplem}~(1) substituting $n$ for $2n$ and $k$ for $n$ in the second line of the above identities.
\begin{align*}
\left[\tikz[baseline=-.6ex, scale=0.1]{
\draw[->-=.9, white, double=black, double distance=0.4pt, ultra thick, rounded corners] 
(0,0) -- (45:3) -- (135:5);
\draw[-<-=.9, white, double=black, double distance=0.4pt, ultra thick, rounded corners] 
(0,0) -- (135:3) -- (45:5);
\draw[->-=.9, white, double=black, double distance=0.4pt, ultra thick, rounded corners] 
(0,0) -- (-45:3) -- (-135:5);
\draw[-<-=.9, white, double=black, double distance=0.4pt, ultra thick, rounded corners] 
(0,0) -- (-135:3) -- (-45:5);
\draw[fill=cyan] (0,0) circle [radius=.8];}\right]_{2n}
&=
\Bigg\langle\,
\tikz[baseline=-.6ex, scale=.1]{
\draw[->-=.2, white, double=black, double distance=0.4pt, ultra thick]
(3.5,5) to[out=north, in=south] (-3.5,10) -- (-3.5,12);
\draw[-<-=.2, white, double=black, double distance=0.4pt, ultra thick]
(-3.5,5) to[out=north, in=south] (3.5,10) -- (3.5,12);
\draw[->-=.2, white, double=black, double distance=0.4pt, ultra thick]
(3.5,-5) to[out=south, in=north] (-3.5,-10) -- (-3.5,-12);
\draw[-<-=.2, white, double=black, double distance=0.4pt, ultra thick]
(-3.5,-5) to[out=south, in=north] (3.5,-10) -- (3.5,-12);
\draw[->-=.5] (-3,5) to[out=south, in=south] (3,5);
\draw[->-=.5] (-3,-5) to[out=north, in=north] (3,-5);
\draw[->-=.5] (0,0) to[out=north east, in=south] (4,4);
\draw[-<-=.5] (0,0) to[out=north west, in=south] (-4,4);
\draw[->-=.5] (0,0) to[out=south east, in=north] (4,-4);
\draw[-<-=.5] (0,0) to[out=south west, in=north] (-4,-4);
\draw[fill=white] (2,0) -- (0,2) -- (-2,0) -- (0,-2) -- cycle;
\draw (0,-2) -- (0,2);
\draw[fill=white] (2,10) rectangle (5,11);
\draw[fill=white] (-2,10) rectangle (-5,11);
\draw[fill=white] (2,-10) rectangle (5,-11);
\draw[fill=white] (-2,-10) rectangle (-5,-11);
\draw[fill=white] (2,4) rectangle (5,5);
\draw[fill=white] (-2,4) rectangle (-5,5);
\draw[fill=white] (2,-4) rectangle (5,-5);
\draw[fill=white] (-2,-4) rectangle (-5,-5);
\node at (4,-6)[right]{$\scriptstyle{2n}$};
\node at (-4,-6)[left]{$\scriptstyle{2n}$};
\node at (4,6)[right]{$\scriptstyle{2n}$};
\node at (-4,6)[left]{$\scriptstyle{2n}$};
\node at (-2.5,2)[left]{$\scriptstyle{n}$};
\node at (2.5,2)[right]{$\scriptstyle{n}$};
\node at (-2.5,-2)[left]{$\scriptstyle{n}$};
\node at (2.5,-2)[right]{$\scriptstyle{n}$};
}
\,\Bigg\rangle_{\! 3}
=\Bigg\langle\,
\tikz[baseline=-.6ex, scale=.1]{
\draw (-3.5,10) -- (-3.5,12);
\draw (3.5,10) -- (3.5,12);
\draw (-3.5,-10) -- (-3.5,-12);
\draw (3.5,-10) -- (3.5,-12);
\draw[triple={[line width=1.4pt, white] in [line width=2.2pt, black] in [line width=5.4pt, white]}]
(3.5,5) to[out=north, in=south] (-3.5,10);
\draw[triple={[line width=1.4pt, white] in [line width=2.2pt, black] in [line width=5.4pt, white]}]
(-3.5,5) to[out=north, in=south] (3.5,10);
\draw[triple={[line width=1.4pt, white] in [line width=2.2pt, black] in [line width=5.4pt, white]}]
(3.5,-5) to[out=south, in=north] (-3.5,-10);
\draw[triple={[line width=1.4pt, white] in [line width=2.2pt, black] in [line width=5.4pt, white]}]
(-3.5,-5) to[out=south, in=north] (3.5,-10);
\draw[->-=.5] (-3.2,5) to[out=south, in=south] (3.2,5);
\draw[->-=.5] (-3.2,-5) to[out=north, in=north] (3.2,-5);
\draw[->-=.5] (0,0) to[out=north east, in=south] (3.8,5);
\draw[-<-=.5] (0,0) to[out=north west, in=south] (-3.8,5);
\draw[->-=.5] (0,0) to[out=south east, in=north] (3.8,-5);
\draw[-<-=.5] (0,0) to[out=south west, in=north] (-3.8,-5);
\draw[fill=white] (2,0) -- (0,2) -- (-2,0) -- (0,-2) -- cycle;
\draw (0,-2) -- (0,2);
\draw[fill=white] (2,10) rectangle (5,11);
\draw[fill=white] (-2,10) rectangle (-5,11);
\draw[fill=white] (2,-10) rectangle (5,-11);
\draw[fill=white] (-2,-10) rectangle (-5,-11);
\node at (-2,2)[left]{$\scriptstyle{n}$};
\node at (2,2)[right]{$\scriptstyle{n}$};
\node at (-2,-2)[left]{$\scriptstyle{n}$};
\node at (2,-2)[right]{$\scriptstyle{n}$};
}
\,\Bigg\rangle_{\! 3}\\
&=q^{\frac{n^2+3n}{3}}q^{-\frac{n^2+3n}{3}}\Bigg\langle\,
\tikz[baseline=-.6ex, scale=.1]{
\draw (3,9) -- (3,10);
\draw (-3,9) -- (-3,10);
\draw (3,-9) -- (3,-10);
\draw (-3,-9) -- (-3,-10);
\draw[white, double=black, double distance=0.4pt, ultra thick] 
(0,7) to[out=east, in=south] (4,8);
\draw[white, double=black, double distance=0.4pt, ultra thick] 
(0,-7) to[out=east, in=north] (4,-8);
\draw[->-=.4, white, double=black, double distance=0.4pt, ultra thick] 
(0,0) to[out=north east, in=south] (2,3)
to[out=north, in=south] (-2,8);
\draw[-<-=.4, white, double=black, double distance=0.4pt, ultra thick] 
(0,0) to[out=north west, in=south] (-2,3)
to[out=north, in=south] (2,8);
\draw[->-=.4, white, double=black, double distance=0.4pt, ultra thick] 
(0,0) to[out=south east, in=north] (2,-3)
to[out=south, in=north] (-2,-8);
\draw[-<-=.4, white, double=black, double distance=0.4pt, ultra thick] 
(0,0) to[out=south west, in=north] (-2,-3)
to[out=south, in=north] (2,-8);
\draw[-<-=1, white, double=black, double distance=0.4pt, ultra thick] 
(-4,8) to[out=south, in=west] (0,7);
\draw[-<-=1, white, double=black, double distance=0.4pt, ultra thick] 
(-4,-8) to[out=north, in=west] (0,-7);
\draw[fill=white] (2,0) -- (0,2) -- (-2,0) -- (0,-2) -- cycle;
\draw (0,-2) -- (0,2);
\draw[fill=white] (1,8) rectangle (5,9);
\draw[fill=white] (-1,8) rectangle (-5,9);
\draw[fill=white] (1,-8) rectangle (5,-9);
\draw[fill=white] (-1,-8) rectangle (-5,-9);
\node at (-2,2)[left]{$\scriptstyle{n}$};
\node at (2,2)[right]{$\scriptstyle{n}$};
\node at (-2,-2)[left]{$\scriptstyle{n}$};
\node at (2,-2)[right]{$\scriptstyle{n}$};
}
\,\Bigg\rangle_{\! 3}
=(q^{\frac{n^2}{3}})^2(q^{-\frac{n^2}{3}})^2\Bigg\langle\,
\tikz[baseline=-.6ex, scale=.1]{
\draw (3,9) -- (3,10);
\draw (-3,9) -- (-3,10);
\draw (3,-9) -- (3,-10);
\draw (-3,-9) -- (-3,-10);
\draw[->-=.4, white, double=black, double distance=0.4pt, ultra thick] 
(0,0) to[out=north east, in=south] (2,3)
to[out=north, in=south] (-4,8);
\draw[-<-=.4, white, double=black, double distance=0.4pt, ultra thick] 
(0,0) to[out=north west, in=south] (-2,3)
to[out=north, in=south] (4,8);
\draw[->-=.4, white, double=black, double distance=0.4pt, ultra thick] 
(0,0) to[out=south east, in=north] (2,-3)
to[out=south, in=north] (-4,-8);
\draw[-<-=.4, white, double=black, double distance=0.4pt, ultra thick] 
(0,0) to[out=south west, in=north] (-2,-3)
to[out=south, in=north] (4,-8);
\draw[-<-=.5, white, double=black, double distance=0.4pt, ultra thick] 
(-2,8) to[out=south, in=south] (2,8);
\draw[-<-=.5, white, double=black, double distance=0.4pt, ultra thick] 
(-2,-8) to[out=north, in=north] (2,-8);
\draw[fill=white] (2,0) -- (0,2) -- (-2,0) -- (0,-2) -- cycle;
\draw (0,-2) -- (0,2);
\draw[fill=white] (1,8) rectangle (5,9);
\draw[fill=white] (-1,8) rectangle (-5,9);
\draw[fill=white] (1,-8) rectangle (5,-9);
\draw[fill=white] (-1,-8) rectangle (-5,-9);
\node at (-2,2)[left]{$\scriptstyle{n}$};
\node at (2,2)[right]{$\scriptstyle{n}$};
\node at (-2,-2)[left]{$\scriptstyle{n}$};
\node at (2,-2)[right]{$\scriptstyle{n}$};
}
\,\Bigg\rangle_{\! 3}\\
&=\Bigg\langle\,
\tikz[baseline=-.6ex, scale=.1]{
\draw (3,9) -- (3,10);
\draw (-3,9) -- (-3,10);
\draw (3,-9) -- (3,-10);
\draw (-3,-9) -- (-3,-10);
\draw[-<-=.5, white, double=black, double distance=0.4pt, ultra thick] 
(0,3) to[out=north east, in=south] (4,8);
\draw[->-=.5, white, double=black, double distance=0.4pt, ultra thick] 
(0,3) to[out=north west, in=south] (-4,8);
\draw[->-=.5, white, double=black, double distance=0.4pt, ultra thick] 
(-1,3) to[out=south west, in=north west] (0,-1) 
to[out=south east, in=north] (2,-3)
to[out=south, in=north] (-4,-8);
\draw[-<-=.5, white, double=black, double distance=0.4pt, ultra thick] 
(1,3) to[out=south east, in=north east] (0,-1) 
to[out=south west, in=north] (-2,-3)
to[out=south, in=north] (4,-8);
\draw[-<-=.5, white, double=black, double distance=0.4pt, ultra thick] 
(-2,8) to[out=south, in=south] (2,8);
\draw[-<-=.5, white, double=black, double distance=0.4pt, ultra thick] 
(-2,-8) to[out=north, in=north] (2,-8);
\draw[fill=white] (2,3) -- (0,5) -- (-2,3) -- (0,1) -- cycle;
\draw (0,1) -- (0,5);
\draw[fill=white] (1,8) rectangle (5,9);
\draw[fill=white] (-1,8) rectangle (-5,9);
\draw[fill=white] (1,-8) rectangle (5,-9);
\draw[fill=white] (-1,-8) rectangle (-5,-9);
\node at (-2,4)[left]{$\scriptstyle{n}$};
\node at (2,4)[right]{$\scriptstyle{n}$};
\node at (-2,-2)[left]{$\scriptstyle{n}$};
\node at (2,-2)[right]{$\scriptstyle{n}$};
}
\,\Bigg\rangle_{\! 3}
=\Bigg\langle\,
\tikz[baseline=-.6ex, scale=.1]{
\draw (3.5,5) -- (3.5,6);
\draw (-3.5,5) -- (-3.5,6);
\draw (3.5,-5) -- (3.5,-6);
\draw (-3.5,-5) -- (-3.5,-6);
\draw[-<-=.5] (-3,5) to[out=south, in=south] (3,5);
\draw[-<-=.5] (-3,-5) to[out=north, in=north] (3,-5);
\draw[-<-=.5] (0,0) to[out=north east, in=south] (4,4);
\draw[->-=.5] (0,0) to[out=north west, in=south] (-4,4);
\draw[-<-=.5] (0,0) to[out=south east, in=north] (4,-4);
\draw[->-=.5] (0,0) to[out=south west, in=north] (-4,-4);
\draw[fill=white] (2,0) -- (0,2) -- (-2,0) -- (0,-2) -- cycle;
\draw (0,-2) -- (0,2);
\draw[fill=white] (2,4) rectangle (5,5);
\draw[fill=white] (-2,4) rectangle (-5,5);
\draw[fill=white] (2,-4) rectangle (5,-5);
\draw[fill=white] (-2,-4) rectangle (-5,-5);
\node at (4,-6)[right]{$\scriptstyle{2n}$};
\node at (-4,-6)[left]{$\scriptstyle{2n}$};
\node at (4,6)[right]{$\scriptstyle{2n}$};
\node at (-4,6)[left]{$\scriptstyle{2n}$};
\node at (-2.5,2)[left]{$\scriptstyle{n}$};
\node at (2.5,2)[right]{$\scriptstyle{n}$};
\node at (-2.5,-2)[left]{$\scriptstyle{n}$};
\node at (2.5,-2)[right]{$\scriptstyle{n}$};
}
\,\Bigg\rangle_{\! 3}
=\left[
\tikz[baseline=-.6ex, scale=0.1]{
\draw[-<-=.5] (0,0) -- (45:5);
\draw[->-=.5] (0,0) -- (135:5); 
\draw[-<-=.5] (0,0) -- (-45:5);
\draw[->-=.5] (0,0) -- (-135:5); 
\draw[fill=cyan] (0,0) circle [radius=.8];}
\right]_{2n}.
\end{align*}
We used Lemma~\ref{A2clasplem} (1), (3) and Lemma~\ref{4valenttwist} in the second line. 
\begin{align*}
\left[\tikz[baseline=-.6ex, scale=0.1]{
\draw[->-=.9, white, double=black, double distance=0.4pt, ultra thick, rounded corners] 
(0,0) -- (45:3) -- (135:5);
\draw[-<-=.9, white, double=black, double distance=0.4pt, ultra thick, rounded corners] 
(0,0) -- (135:3) -- (45:5);
\draw[-<-=.9, white, double=black, double distance=0.4pt, ultra thick, rounded corners] 
(0,0) -- (-45:3) -- (-135:5);
\draw[->-=.9, white, double=black, double distance=0.4pt, ultra thick, rounded corners] 
(0,0) -- (-135:3) -- (-45:5);
\draw[fill=cyan] (0,0) circle [radius=.8];}\right]_{2n}
&=
\Bigg\langle\,
\tikz[baseline=-.6ex, scale=.1]{
\draw[->-=.2, white, double=black, double distance=0.4pt, ultra thick]
(3.5,5) to[out=north, in=south] (-3.5,10) -- (-3.5,12);
\draw[-<-=.2, white, double=black, double distance=0.4pt, ultra thick]
(-3.5,5) to[out=north, in=south] (3.5,10) -- (3.5,12);
\draw[-<-=.2, white, double=black, double distance=0.4pt, ultra thick]
(3.5,-5) to[out=south, in=north] (-3.5,-10) -- (-3.5,-12);
\draw[->-=.2, white, double=black, double distance=0.4pt, ultra thick]
(-3.5,-5) to[out=south, in=north] (3.5,-10) -- (3.5,-12);
\draw[->-=.5] (-3,5) to[out=south, in=south] (3,5);
\draw[-<-=.5] (-3,-5) to[out=north, in=north] (3,-5);
\draw[->-=.5] (4,-5) -- (4,5);
\draw[-<-=.5] (-4,-5) -- (-4,5);
\draw[fill=white] (2,10) rectangle (5,11);
\draw[fill=white] (-2,10) rectangle (-5,11);
\draw[fill=white] (2,-10) rectangle (5,-11);
\draw[fill=white] (-2,-10) rectangle (-5,-11);
\draw[fill=white] (2,4) rectangle (5,5);
\draw[fill=white] (-2,4) rectangle (-5,5);
\draw[fill=white] (2,-4) rectangle (5,-5);
\draw[fill=white] (-2,-4) rectangle (-5,-5);
\node at (4,-6)[right]{$\scriptstyle{2n}$};
\node at (-4,-6)[left]{$\scriptstyle{2n}$};
\node at (4,6)[right]{$\scriptstyle{2n}$};
\node at (-4,6)[left]{$\scriptstyle{2n}$};
\node at (-4,0)[left]{$\scriptstyle{n}$};
\node at (4,0)[right]{$\scriptstyle{n}$};
\node at (0,3)[below]{$\scriptstyle{n}$};
\node at (0,-3)[above]{$\scriptstyle{n}$};
}
\,\Bigg\rangle_{\! 3}
=\Bigg\langle\,
\tikz[baseline=-.6ex, scale=.1]{
\draw (-3.5,10) -- (-3.5,12);
\draw (3.5,10) -- (3.5,12);
\draw (-3.5,-10) -- (-3.5,-12);
\draw (3.5,-10) -- (3.5,-12);
\draw[triple={[line width=1.4pt, white] in [line width=2.2pt, black] in [line width=5.4pt, white]}]
(3.5,5) to[out=north, in=south] (-3.5,10);
\draw[triple={[line width=1.4pt, white] in [line width=2.2pt, black] in [line width=5.4pt, white]}]
(-3.5,5) to[out=north, in=south] (3.5,10);
\draw[triple={[line width=1.4pt, white] in [line width=2.2pt, black] in [line width=5.4pt, white]}]
(3.5,-5) to[out=south, in=north] (-3.5,-10);
\draw[triple={[line width=1.4pt, white] in [line width=2.2pt, black] in [line width=5.4pt, white]}]
(-3.5,-5) to[out=south, in=north] (3.5,-10);
\draw[->-=.5] (-3.2,5) to[out=south, in=south] (3.2,5);
\draw[-<-=.5] (-3.2,-5) to[out=north, in=north] (3.2,-5);
\draw[->-=.5] (3.8,-5) -- (3.8,5);
\draw[-<-=.5] (-3.8,-5) -- (-3.8,5);
\draw[fill=white] (2,10) rectangle (5,11);
\draw[fill=white] (-2,10) rectangle (-5,11);
\draw[fill=white] (2,-10) rectangle (5,-11);
\draw[fill=white] (-2,-10) rectangle (-5,-11);
\node at (-4,0)[left]{$\scriptstyle{n}$};
\node at (4,0)[right]{$\scriptstyle{n}$};
\node at (0,3)[below]{$\scriptstyle{n}$};
\node at (0,-3)[above]{$\scriptstyle{n}$};
}
\,\Bigg\rangle_{\! 3}\\
&=q^{\frac{n^2+3n}{3}}q^{-\frac{n^2+3n}{3}}\Bigg\langle\,
\tikz[baseline=-.6ex, scale=.1]{
\draw (3,9) -- (3,10);
\draw (-3,9) -- (-3,10);
\draw (3,-9) -- (3,-10);
\draw (-3,-9) -- (-3,-10);
\draw[white, double=black, double distance=0.4pt, ultra thick] 
(0,7) to[out=east, in=south] (4,8);
\draw[white, double=black, double distance=0.4pt, ultra thick] 
(0,-7) to[out=east, in=north] (4,-8);
\draw[->-=.1, white, double=black, double distance=0.4pt, ultra thick] 
(2,0) -- (2,3) to[out=north, in=south] (-2,8);
\draw[-<-=.1, white, double=black, double distance=0.4pt, ultra thick] 
(-2,0) -- (-2,3) to[out=north, in=south] (2,8);
\draw[white, double=black, double distance=0.4pt, ultra thick] 
(2,0) -- (2,-3) to[out=south, in=north] (-2,-8);
\draw[white, double=black, double distance=0.4pt, ultra thick] 
(-2,0) -- (-2,-3) to[out=south, in=north] (2,-8);
\draw[-<-=1, white, double=black, double distance=0.4pt, ultra thick] 
(-4,8) to[out=south, in=west] (0,7);
\draw[->-=1, white, double=black, double distance=0.4pt, ultra thick] 
(-4,-8) to[out=north, in=west] (0,-7);
\draw[fill=white] (1,8) rectangle (5,9);
\draw[fill=white] (-1,8) rectangle (-5,9);
\draw[fill=white] (1,-8) rectangle (5,-9);
\draw[fill=white] (-1,-8) rectangle (-5,-9);
\node at (-2,0)[left]{$\scriptstyle{n}$};
\node at (2,0)[right]{$\scriptstyle{n}$};
}
\,\Bigg\rangle_{\! 3}
=\Bigg\langle\,
\tikz[baseline=-.6ex, scale=.1]{
\draw (3,9) -- (3,10);
\draw (-3,9) -- (-3,10);
\draw (3,-9) -- (3,-10);
\draw (-3,-9) -- (-3,-10);
\draw[white, double=black, double distance=0.4pt, ultra thick] 
(0,7) to[out=east, in=south] (4,8);
\draw[white, double=black, double distance=0.4pt, ultra thick] 
(0,-7) to[out=east, in=north] (4,-8);
\draw[->-=.5, white, double=black, double distance=0.4pt, ultra thick] 
(-2,-8) -- (-2,8);
\draw[-<-=.5, white, double=black, double distance=0.4pt, ultra thick] 
(2,-8) -- (2,8);
\draw[-<-=1, white, double=black, double distance=0.4pt, ultra thick] 
(-4,8) to[out=south, in=west] (0,7);
\draw[->-=1, white, double=black, double distance=0.4pt, ultra thick] 
(-4,-8) to[out=north, in=west] (0,-7);
\draw[fill=white] (1,8) rectangle (5,9);
\draw[fill=white] (-1,8) rectangle (-5,9);
\draw[fill=white] (1,-8) rectangle (5,-9);
\draw[fill=white] (-1,-8) rectangle (-5,-9);
\node at (-2,0)[left]{$\scriptstyle{n}$};
\node at (2,0)[right]{$\scriptstyle{n}$};
}
\,\Bigg\rangle_{\! 3}\\
&=(q^{\frac{n^2}{3}})^2(q^{-\frac{n^2}{3}})^2\Bigg\langle\,
\tikz[baseline=-.6ex, scale=.1]{
\draw (3,6) -- (3,7);
\draw (-3,6) -- (-3,7);
\draw (3,-6) -- (3,-7);
\draw (-3,-6) -- (-3,-7);
\draw[-<-=.5] (4,-5) -- (4,5);
\draw[->-=.5] (-4,-5) -- (-4,5);
\draw[-<-=.5, white, double=black, double distance=0.4pt, ultra thick] 
(-2,5) to[out=south, in=south] (2,5);
\draw[->-=.5, white, double=black, double distance=0.4pt, ultra thick] 
(-2,-5) to[out=north, in=north] (2,-5);
\draw[fill=white] (1,5) rectangle (5,6);
\draw[fill=white] (-1,5) rectangle (-5,6);
\draw[fill=white] (1,-5) rectangle (5,-6);
\draw[fill=white] (-1,-5) rectangle (-5,-6);
\node at (-4,0)[left]{$\scriptstyle{n}$};
\node at (4,0)[right]{$\scriptstyle{n}$};
\node at (0,4)[below]{$\scriptstyle{n}$};
\node at (0,-4)[above]{$\scriptstyle{n}$};
}
\,\Bigg\rangle_{\! 3}
=\left[\tikz[baseline=-.6ex, scale=0.1]{
\draw[-<-=.5] (0,0) -- (45:5);
\draw[->-=.5] (0,0) -- (135:5); 
\draw[->-=.5] (0,0) -- (-45:5);
\draw[-<-=.5] (0,0) -- (-135:5); 
\draw[fill=cyan] (0,0) circle [radius=.8];}\right]_{2n}\ .
\end{align*}
\end{proof}

\begin{RMK}\label{general}\ 
\begin{itemize}
\item 
If $G$ is a singular link, 
that is, 
all $4$-valent vertices of $G$ are 
$
\tikz[baseline=-.6ex, scale=.8]{
\draw [thin, dashed, fill=white] (0,0) circle [radius=.5];
\draw[-<-=.5] (45:.5) -- (0:0);
\draw[-<-=.5] (135:.5) -- (0:0);
\draw[->-=.5] (-45:.5) -- (0:0);
\draw[->-=.5] (-135:.5) -- (0:0);
\draw[fill=cyan] (0,0) circle [radius=.08];}
$, 
then the coloring of the invariant need not be even. 
This means that we can define $\left[G\right]_m$ for all positive integers $m$ if $G$ is a singular link.
This invariant is considered the $\mathfrak{sl}_3$ colored Jones polynomials for singular links.
\item
We can also define $\left[G\right]_{2n}^{(k)}$ by replacing Definition~\ref{oriKVdef}~(3) with 
\begin{itemize}
\item[(3-k)] 
$
\left[\,\tikz[baseline=-.6ex, scale=.1]{
\draw (-4,4) -- +(-2,0);
\draw (4,-4) -- +(2,0);
\draw (-4,-4) -- +(-2,0);
\draw (4,4) -- +(2,0);
\draw[->-=.8] (0,0) -- (2,3) to[out=north east, in=west] (4,4);
\draw[->-=.8] (0,0) -- (-2,3) to[out=north west, in=east] (-4,4);
\draw[-<-=.8] (0,0) -- (2,-3) to[out=south east, in=west] (4,-4);
\draw[-<-=.8] (0,0) -- (-2,-3) to[out=south west, in=east] (-4,-4);
\draw[fill=cyan] (0,0) circle [radius=.8];
}\,\right]_{2n}
=
\Bigg\langle\,\tikz[baseline=-.6ex, scale=.1]{
\draw 
(-5,4) -- +(-2,0)
(5,-4) -- +(2,0)
(-5,-4) -- +(-2,0)
(5,4) -- +(2,0);
\draw[-<-=.5] (-5,3) to[out=east, in=east] (-5,-3);
\draw[-<-=.5] (5,3) to[out=west, in=west] (5,-3);
\draw[-<-=.5] (-5,5) to[out=east, in=north west] (0,0);
\draw[-<-=.5] (0,0) to[out=south east, in=west] (5,-5);
\draw[-<-=.5] (5,5) to[out=west, in=north east] (0,0);
\draw[-<-=.5] (0,0) to[out=south west, in=east] (-5,-5);
\draw[fill=white] (2,0) -- (0,2) -- (-2,0) -- (0,-2) -- cycle;
\draw (-2,0) -- (2,0);
\draw[fill=white] (5,-6) rectangle +(1,4);
\draw[fill=white] (-5,-6) rectangle +(-1,4);
\draw[fill=white] (5,6) rectangle +(1,-4);
\draw[fill=white] (-5,6) rectangle +(-1,-4);
\node at (5,-6)[right]{$\scriptstyle{2n}$};
\node at (-5,-6)[left]{$\scriptstyle{2n}$};
\node at (5,6)[right]{$\scriptstyle{2n}$};
\node at (-5,6)[left]{$\scriptstyle{2n}$};
\node at (-3,0)[left]{$\scriptstyle{2n-k}$};
\node at (3,0)[right]{$\scriptstyle{2n-k}$};
\node at (0,5)[right]{$\scriptstyle{k}$};
\node at (0,5)[left]{$\scriptstyle{k}$};
\node at (0,-5)[right]{$\scriptstyle{k}$};
\node at (0,-5)[left]{$\scriptstyle{k}$};
}\,\Bigg\rangle_{\! 3}
$
and 
$\left[\,\tikz[baseline=-.6ex, scale=.1]{
\draw (-4,4) -- +(-2,0);
\draw (4,-4) -- +(2,0);
\draw (-4,-4) -- +(-2,0);
\draw (4,4) -- +(2,0);
\draw[-<-=.8] (0,0) -- (2,3) to[out=north east, in=west] (4,4);
\draw[->-=.8] (0,0) -- (-2,3) to[out=north west, in=east] (-4,4);
\draw[->-=.8] (0,0) -- (2,-3) to[out=south east, in=west] (4,-4);
\draw[-<-=.8] (0,0) -- (-2,-3) to[out=south west, in=east] (-4,-4);
\draw[fill=cyan] (0,0) circle [radius=.8];
}\,\right]_{2n}
=
\Bigg\langle\,\tikz[baseline=-.6ex, scale=.08]{
\draw 
(-4,4) -- +(-2,0)
(4,-4) -- +(2,0)
(-4,-4) -- +(-2,0)
(4,4) -- +(2,0);
\draw[-<-=.5] (-4,5) -- (4,5);
\draw[->-=.5] (-4,-5) -- (4,-5);
\draw[-<-=.5] (-4,3) to[out=east, in=east] (-4,-3);
\draw[->-=.5] (4,3) to[out=west, in=west] (4,-3);
\draw[fill=white] (4,-6) rectangle +(1,4);
\draw[fill=white] (-4,-6) rectangle +(-1,4);
\draw[fill=white] (4,6) rectangle +(1,-4);
\draw[fill=white] (-4,6) rectangle +(-1,-4);
\node at (4,-6)[right]{$\scriptstyle{2n}$};
\node at (-4,-6)[left]{$\scriptstyle{2n}$};
\node at (4,6)[right]{$\scriptstyle{2n}$};
\node at (-4,6)[left]{$\scriptstyle{2n}$};
\node at (0,5)[above]{$\scriptstyle{n}$};
\node at (0,-5)[below]{$\scriptstyle{n}$};
\node at (-2,0)[left]{$\scriptstyle{n}$};
\node at (2,0)[right]{$\scriptstyle{n}$};
}\,\Bigg\rangle_{\! 3}$,
\end{itemize}
for $k=0,1,\dots,2n$.
\end{itemize}
\end{RMK}

\subsection{Invariant of unoriented $4$-valent rigid vertex graphs}
For unoriented $4$-valent rigid vertex graph, 
we will define the invariant by using the colored trivalent graphs.
Firstly, 
we represent two types of clasped $A_2$ web by using colored trivalent graphs with white and black vertices. 
In general, a diagrammatic expression of a colored trivalent graph for a $A_2$ web is given by Kim~\cite{Kim06}.

We denote
$
\tikz[baseline=-.6ex, scale=.5]{
\draw[-<-=.8] (-.6,.4) -- (.6,.4);
\draw[->-=.8] (-.6,-.4) -- (.6,-.4);
\draw[fill=white] (-.1,-.6) rectangle (.1,.6);
\draw (-.1,.0) -- (.1,.0);
\node at (.4,-.6)[right]{$\scriptstyle{n}$};
\node at (-.4,-.6)[left]{$\scriptstyle{n}$};
\node at (.4,.6)[right]{$\scriptstyle{n}$};
\node at (-.4,.6)[left]{$\scriptstyle{n}$};}
$
by
$
\tikz[baseline=-.6ex, scale=0.5]{
\draw[triple={[line width=1.4pt, white] in [line width=2.2pt, black] in [line width=5.4pt, white]}]
(-1,0)--(1,0);
\node at (0,0) [above]{$\scriptstyle{n}$};
}
$
and 
$
\tikz[baseline=-.6ex, scale=.5]{
\draw[->-=.8] (-.6,.4) -- (.6,.4);
\draw[-<-=.8] (-.6,-.4) -- (.6,-.4);
\draw[fill=white] (-.1,-.6) rectangle (.1,.6);
\draw (-.1,.0) -- (.1,.0);
\node at (.4,-.6)[right]{$\scriptstyle{n}$};
\node at (-.4,-.6)[left]{$\scriptstyle{n}$};
\node at (.4,.6)[right]{$\scriptstyle{n}$};
\node at (-.4,.6)[left]{$\scriptstyle{n}$};}
$
by
$
\tikz[baseline=-.6ex, scale=0.5]{
\draw[-|-=.5, triple={[line width=1.4pt, white] in [line width=2.2pt, black] in [line width=5.4pt, white]}]
(-1,0)--(1,0);
\node at (0,0) [above]{$\scriptstyle{n}$};
}
$
.

\begin{DEF}
Let $n$ be a non-negative integer. 
For $0\leq i\leq n$, 
we define two types of trivalent vertices as follows.
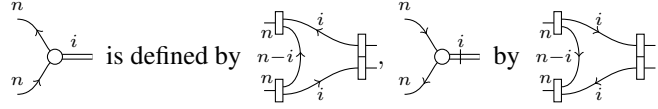
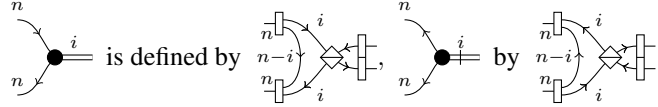
\[
\tikz[baseline=-.6ex, scale=0.5]{
\draw[triple={[line width=1.4pt, white] in [line width=2.2pt, black] in [line width=5.4pt, white]}]
(0,0) -- (1,0);
\draw[-<-=.5] (-1,1) to[out=east, in=north west] (0,0);
\draw[->-=.5] (-1,-1) to[out=east, in=south west] (0,0);
\draw[fill=white] (0,0)  circle (.2);
\node at (-1,1) [above]{$\scriptstyle{n}$};
\node at (-1,-1) [above]{$\scriptstyle{n}$};
\node at (.5,0) [above]{$\scriptstyle{i}$};}
\text{\ is defined by\ }
\tikz[baseline=-.6ex, scale=0.5]{
\draw (-1,.9) -- +(-.5,0);
\draw (-1,-.9) -- +(-.5,0);
\draw (1,.3) -- +(.5,0);
\draw (1,-.3) -- +(.5,0);
\draw[-<-=.5] (-1,1) to[out=east, in=west] (1,.3);
\draw[->-=.5] (-1,-1) to[out=east, in=west] (1,-.3);
\draw[-<-=.5] (-1,.8) to[out=east, in=east] (-1,.-.8);
\draw[fill=white] (-1.2,.6) rectangle (-1,1.2);
\draw[fill=white] (-1.2,-.6) rectangle (-1,-1.2);
\draw[fill=white] (1,-.6) rectangle (1.2,.6);
\draw (1,.0) -- (1.2,.0);
\node at (-1,.6)[left]{$\scriptstyle{n}$};
\node at (-1,-.6)[left]{$\scriptstyle{n}$};
\node at (-1.2,0){$\scriptstyle{n-i}$};
\node at (0,1){$\scriptstyle{i}$};
\node at (0,-1){$\scriptstyle{i}$};}
,
\tikz[baseline=-.6ex, scale=0.5]{
\draw[-|-=.5, triple={[line width=1.4pt, white] in [line width=2.2pt, black] in [line width=5.4pt, white]}]
(0,0) -- (1,0);
\draw[->-=.5] (-1,1) to[out=east, in=north west] (0,0);
\draw[-<-=.5] (-1,-1) to[out=east, in=south west] (0,0);
\draw[fill=white] (0,0)  circle (.2);
\node at (-1,1) [above]{$\scriptstyle{n}$};
\node at (-1,-1) [above]{$\scriptstyle{n}$};
\node at (.5,0) [above]{$\scriptstyle{i}$};}
\text{\ by\ }
\tikz[baseline=-.6ex, scale=0.5]{
\draw (-1,.9) -- +(-.5,0);
\draw (-1,-.9) -- +(-.5,0);
\draw (1,.3) -- +(.5,0);
\draw (1,-.3) -- +(.5,0);
\draw[->-=.5] (-1,1) to[out=east, in=west] (1,.3);
\draw[-<-=.5] (-1,-1) to[out=east, in=west] (1,-.3);
\draw[->-=.5] (-1,.8) to[out=east, in=east] (-1,.-.8);
\draw[fill=white] (-1.2,.6) rectangle (-1,1.2);
\draw[fill=white] (-1.2,-.6) rectangle (-1,-1.2);
\draw[fill=white] (1,-.6) rectangle (1.2,.6);
\draw (1,.0) -- (1.2,.0);
\node at (-1,.6)[left]{$\scriptstyle{n}$};
\node at (-1,-.6)[left]{$\scriptstyle{n}$};
\node at (-1.2,0){$\scriptstyle{n-i}$};
\node at (0,1){$\scriptstyle{i}$};
\node at (0,-1){$\scriptstyle{i}$};}
\]
and 
\[
\tikz[baseline=-.6ex, scale=0.5]{
\draw[triple={[line width=1.4pt, white] in [line width=2.2pt, black] in [line width=5.4pt, white]}]
(0,0) -- (1,0);
\draw[->-=.5] (-1,1) to[out=east, in=north west] (0,0);
\draw[-<-=.5] (-1,-1) to[out=east, in=south west] (0,0);
\draw[fill=black] (0,0)  circle (.2);
\node at (-1,1) [above]{$\scriptstyle{n}$};
\node at (-1,-1) [above]{$\scriptstyle{n}$};
\node at (.5,0) [above]{$\scriptstyle{i}$};}
\text{\ is defined by\ }
\tikz[baseline=-.6ex, scale=0.5]{
\draw (-1,.9) -- +(-.5,0);
\draw (-1,-.9) -- +(-.5,0);
\draw (1,.3) -- +(.5,0);
\draw (1,-.3) -- +(.5,0);
\draw[-<-=.3, -<-=.9] (-1,-1) to[out=east, in=west] (1,.3);
\draw[->-=.3, ->-=.9] (-1,1) to[out=east, in=west] (1,-.3);
\draw[->-=.5] (-1,.8) to[out=east, in=east] (-1,.-.8);
\draw [fill=white] (0,0) -- (.3,.3) -- (.6,0) -- (.3,-.3) -- cycle;
\draw (0,0) -- (.6,0);
\draw[fill=white] (-1.2,.6) rectangle (-1,1.2);
\draw[fill=white] (-1.2,-.6) rectangle (-1,-1.2);
\draw[fill=white] (1,-.6) rectangle (1.2,.6);
\draw (1,.0) -- (1.2,.0);
\node at (-1,.6)[left]{$\scriptstyle{n}$};
\node at (-1,-.6)[left]{$\scriptstyle{n}$};
\node at (-1.2,0){$\scriptstyle{n-i}$};
\node at (0,1){$\scriptstyle{i}$};
\node at (0,-1){$\scriptstyle{i}$};}
,
\tikz[baseline=-.6ex, scale=0.5]{
\draw[-|-=.5, triple={[line width=1.4pt, white] in [line width=2.2pt, black] in [line width=5.4pt, white]}]
(0,0) -- (1,0);
\draw[-<-=.5] (-1,1) to[out=east, in=north west] (0,0);
\draw[-<-=.5] (-1,-1) to[out=east, in=south west] (0,0);
\draw[fill=black] (0,0)  circle (.2);
\node at (-1,1) [above]{$\scriptstyle{n}$};
\node at (-1,-1) [above]{$\scriptstyle{n}$};
\node at (.5,0) [above]{$\scriptstyle{i}$};}
\text{\ by\ }
\tikz[baseline=-.6ex, scale=0.5]{
\draw (-1,.9) -- +(-.5,0);
\draw (-1,-.9) -- +(-.5,0);
\draw (1,.3) -- +(.5,0);
\draw (1,-.3) -- +(.5,0);
\draw[->-=.3, ->-=.9] (-1,-1) to[out=east, in=west] (1,.3);
\draw[-<-=.3, -<-=.9] (-1,1) to[out=east, in=west] (1,-.3);
\draw[-<-=.5] (-1,.8) to[out=east, in=east] (-1,.-.8);
\draw [fill=white] (0,0) -- (.3,.3) -- (.6,0) -- (.3,-.3) -- cycle;
\draw (0,0) -- (.6,0);
\draw[fill=white] (-1.2,.6) rectangle (-1,1.2);
\draw[fill=white] (-1.2,-.6) rectangle (-1,-1.2);
\draw[fill=white] (1,-.6) rectangle (1.2,.6);
\draw (1,.0) -- (1.2,.0);
\node at (-1,.6)[left]{$\scriptstyle{n}$};
\node at (-1,-.6)[left]{$\scriptstyle{n}$};
\node at (-1.2,0){$\scriptstyle{n-i}$};
\node at (0,1){$\scriptstyle{i}$};
\node at (0,-1){$\scriptstyle{i}$};}
\]
\end{DEF}

Let $\bar{G}$ be an unoriented $4$-valent rigid vertex graph diagram. 
\begin{DEF}\label{unoriKVdef}
We define a polynomial $\left[\bar{G}\right]_{(n,n)}$ by the following rules:
\begin{enumerate}
\item 
$\left[\tikz[baseline=-.6ex, scale=0.1]{
\draw (-5,0) -- (5,0);
}\right]_{(n,n)}
=
\Big\langle
\tikz[baseline=-.6ex, scale=0.1]{
\draw[triple={[line width=1.4pt, white] in [line width=2.2pt, black] in [line width=5.4pt, white]}] (-5,0) -- (5,0);
\node at (0,0)[above]{$\scriptstyle{n}$};
}\Big\rangle_3
$,
\item 
$
\left[\,\tikz[baseline=-.6ex, scale=.1]{
\draw (-4,4) -- +(-2,0);
\draw[white, double=black, double distance=0.4pt, ultra thick] 
(4,-4) to[out=west, in=east] (-4,4);
\draw (4,-4) -- +(2,0);
\draw (-4,-4) -- +(-2,0);
\draw[white, double=black, double distance=0.4pt, ultra thick] 
(-4,-4) to[out=east, in=west] (4,4);
\draw (4,4) -- +(2,0);
}\,\right]_{(n,n)}
=
\Bigg\langle\,\tikz[baseline=-.6ex, scale=.08]{
\draw (-4,4) -- +(-2,0);
\draw[triple={[line width=1.4pt, white] in [line width=2.2pt, black] in [line width=5.4pt, white]}] 
(6,-4) -- (4,-4) to[out=west, in=east] (-4,4) -- (-6,4);
\draw[triple={[line width=1.4pt, white] in [line width=2.2pt, black] in [line width=5.4pt, white]}] 
(-6,-4) -- (-4,-4) to[out=east, in=west] (4,4) -- (6,4);
\node at (4,-6)[right]{$\scriptstyle{n}$};
\node at (-4,-6)[left]{$\scriptstyle{n}$};
\node at (4,6)[right]{$\scriptstyle{n}$};
\node at (-4,6)[left]{$\scriptstyle{n}$};
}\,\Bigg\rangle_{\! 3}$,
\item 
$\left[\,\tikz[baseline=-.6ex, scale=.1]{
\draw (-4,4) -- +(-2,0);
\draw (4,-4) -- +(2,0);
\draw (-4,-4) -- +(-2,0);
\draw (4,4) -- +(2,0);
\draw (0,0) -- (2,3) to[out=north east, in=west] (4,4);
\draw (0,0) -- (-2,3) to[out=north west, in=east] (-4,4);
\draw (0,0) -- (2,-3) to[out=south east, in=west] (4,-4);
\draw (0,0) -- (-2,-3) to[out=south west, in=east] (-4,-4);
\draw[fill=cyan] (0,0) circle [radius=.8];
}\,\right]_{(n,n)}
=
\Bigg\langle\,\tikz[baseline=-.6ex, scale=.1]{
\draw[triple={[line width=1.4pt, white] in [line width=2.2pt, black] in [line width=5.4pt, white]}] 
(3,3) to[out=north east, in=west] (5,4) -- (7,4);
\draw[triple={[line width=1.4pt, white] in [line width=2.2pt, black] in [line width=5.4pt, white]}] 
(-3,3) to[out=north west, in=east] (-5,4) -- (-7,4);
\draw[triple={[line width=1.4pt, white] in [line width=2.2pt, black] in [line width=5.4pt, white]}] 
(3,-3) to[out=south east, in=west] (5,-4) -- (7,-4);
\draw[triple={[line width=1.4pt, white] in [line width=2.2pt, black] in [line width=5.4pt, white]}] 
(-3,-3) to[out=south west, in=east] (-5,-4) -- (-7,-4);
\draw[->-=.5] (3,3) -- (-3,3);
\draw[->-=.5] (-3,3) -- (-3,-3);
\draw[->-=.5] (-3,-3) -- (3,-3);
\draw[->-=.5] (3,-3) -- (3,3);
\draw[fill=white] (3,3) circle [radius=.8];
\draw[fill=white] (-3,3) circle [radius=.8];
\draw[fill=white] (3,-3) circle [radius=.8];
\draw[fill=white] (-3,-3) circle [radius=.8];
\node at (7,-4)[right]{$\scriptstyle{n}$};
\node at (-7,-4)[left]{$\scriptstyle{n}$};
\node at (7,4)[right]{$\scriptstyle{n}$};
\node at (-7,4)[left]{$\scriptstyle{n}$};
\node at (0,3)[above]{$\scriptstyle{n}$};
\node at (0,-3)[below]{$\scriptstyle{n}$};
\node at (3,0)[right]{$\scriptstyle{n}$};
\node at (-3,0)[left]{$\scriptstyle{n}$};
}\,\Bigg\rangle_{\! 3}
+
\Bigg\langle\,\tikz[baseline=-.6ex, scale=.1]{
\draw[triple={[line width=1.4pt, white] in [line width=2.2pt, black] in [line width=5.4pt, white]}] 
(3,3) to[out=north east, in=west] (5,4) -- (7,4);
\draw[triple={[line width=1.4pt, white] in [line width=2.2pt, black] in [line width=5.4pt, white]}] 
(-3,3) to[out=north west, in=east] (-5,4) -- (-7,4);
\draw[triple={[line width=1.4pt, white] in [line width=2.2pt, black] in [line width=5.4pt, white]}] 
(3,-3) to[out=south east, in=west] (5,-4) -- (7,-4);
\draw[triple={[line width=1.4pt, white] in [line width=2.2pt, black] in [line width=5.4pt, white]}] 
(-3,-3) to[out=south west, in=east] (-5,-4) -- (-7,-4);
\draw[-<-=.5] (3,3) -- (-3,3);
\draw[-<-=.5] (-3,3) -- (-3,-3);
\draw[-<-=.5] (-3,-3) -- (3,-3);
\draw[-<-=.5] (3,-3) -- (3,3);
\draw[fill=black] (3,3) circle [radius=.8];
\draw[fill=black] (-3,3) circle [radius=.8];
\draw[fill=black] (3,-3) circle [radius=.8];
\draw[fill=black] (-3,-3) circle [radius=.8];
\node at (7,-4)[right]{$\scriptstyle{n}$};
\node at (-7,-4)[left]{$\scriptstyle{n}$};
\node at (7,4)[right]{$\scriptstyle{n}$};
\node at (-7,4)[left]{$\scriptstyle{n}$};
\node at (0,3)[above]{$\scriptstyle{n}$};
\node at (0,-3)[below]{$\scriptstyle{n}$};
\node at (3,0)[right]{$\scriptstyle{n}$};
\node at (-3,0)[left]{$\scriptstyle{n}$};
}\,\Bigg\rangle_{\! 3}$.
\end{enumerate}
\end{DEF}

\begin{THM}\label{unoriKVthm}
$\left[\bar{G}\right]_{(n,n)}$ is invariant under the Reidemeister moves (RI) -- (RV).
\end{THM}
\begin{proof}
We show the invariance under the Reidemeister move (RV).
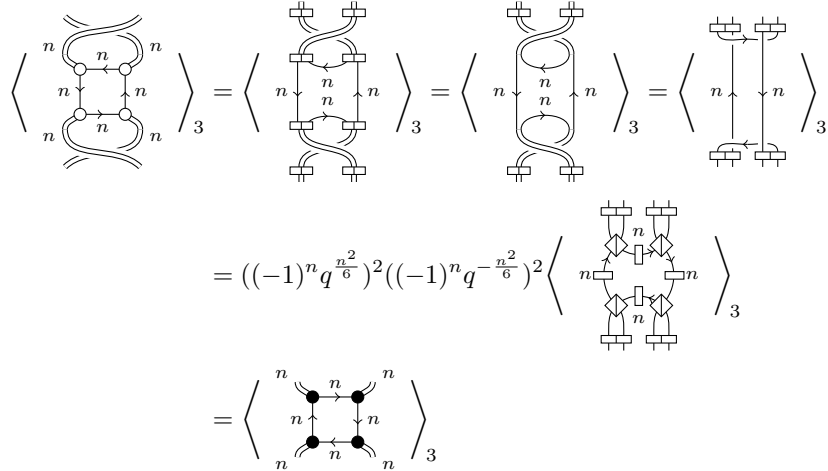
\begin{align*}
\Bigg\langle\,
\tikz[baseline=-.6ex, scale=.1]{
\draw[triple={[line width=1.4pt, white] in [line width=2.2pt, black] in [line width=5.4pt, white]}] 
(3,3) to[out=north east, in=south] (5,5);
\draw[triple={[line width=1.4pt, white] in [line width=2.2pt, black] in [line width=5.4pt, white]}] 
(-3,3) to[out=north west, in=south] (-5,5);
\draw[triple={[line width=1.4pt, white] in [line width=2.2pt, black] in [line width=5.4pt, white]}] 
(3,-3) to[out=south east, in=north] (5,-5);
\draw[triple={[line width=1.4pt, white] in [line width=2.2pt, black] in [line width=5.4pt, white]}] 
(-3,-3) to[out=south west, in=north] (-5,-5);
\draw[triple={[line width=1.4pt, white] in [line width=2.2pt, black] in [line width=5.4pt, white]}] 
(5,5) to[out=north, in=south east] (-5,10);
\draw[triple={[line width=1.4pt, white] in [line width=2.2pt, black] in [line width=5.4pt, white]}] 
(-5,5) to[out=north, in=south west] (5,10);
\draw[triple={[line width=1.4pt, white] in [line width=2.2pt, black] in [line width=5.4pt, white]}] 
(5,-5) to[out=south, in=north east] (-5,-10);
\draw[triple={[line width=1.4pt, white] in [line width=2.2pt, black] in [line width=5.4pt, white]}] 
(-5,-5) to[out=south, in=north west] (5,-10);
\draw[->-=.5] (3,3) -- (-3,3);
\draw[->-=.5] (-3,3) -- (-3,-3);
\draw[->-=.5] (-3,-3) -- (3,-3);
\draw[->-=.5] (3,-3) -- (3,3);
\draw[fill=white] (3,3) circle [radius=.8];
\draw[fill=white] (-3,3) circle [radius=.8];
\draw[fill=white] (3,-3) circle [radius=.8];
\draw[fill=white] (-3,-3) circle [radius=.8];
\node at (5,-6)[right]{$\scriptstyle{n}$};
\node at (-5,-6)[left]{$\scriptstyle{n}$};
\node at (5,6)[right]{$\scriptstyle{n}$};
\node at (-5,6)[left]{$\scriptstyle{n}$};
\node at (0,3)[above]{$\scriptstyle{n}$};
\node at (0,-3)[below]{$\scriptstyle{n}$};
\node at (3,0)[right]{$\scriptstyle{n}$};
\node at (-3,0)[left]{$\scriptstyle{n}$};}
\,\Bigg\rangle_{\! 3}
&=\Bigg\langle\,
\tikz[baseline=-.6ex, scale=.1]{
\draw[triple={[line width=1.4pt, white] in [line width=2.2pt, black] in [line width=5.4pt, white]}]
(3.5,5) to[out=north, in=south] (-3.5,10) -- (-3.5,12);
\draw[triple={[line width=1.4pt, white] in [line width=2.2pt, black] in [line width=5.4pt, white]}]
(-3.5,5) to[out=north, in=south] (3.5,10) -- (3.5,12);
\draw[triple={[line width=1.4pt, white] in [line width=2.2pt, black] in [line width=5.4pt, white]}]
(3.5,-5) to[out=south, in=north] (-3.5,-10) -- (-3.5,-12);
\draw[triple={[line width=1.4pt, white] in [line width=2.2pt, black] in [line width=5.4pt, white]}]
(-3.5,-5) to[out=south, in=north] (3.5,-10) -- (3.5,-12);
\draw[-<-=.5] (-3,5) to[out=south, in=south] (3,5);
\draw[->-=.5] (-3,-5) to[out=north, in=north] (3,-5);
\draw[->-=.5] (4,-5) -- (4,5);
\draw[-<-=.5] (-4,-5) -- (-4,5);
\draw[fill=white] (2,10) rectangle (5,11);
\draw[fill=white] (-2,10) rectangle (-5,11);
\draw[fill=white] (2,-10) rectangle (5,-11);
\draw[fill=white] (-2,-10) rectangle (-5,-11);
\draw (3.5,10) -- (3.5,11);
\draw (-3.5,10) -- (-3.5,11);
\draw (3.5,-10) -- (3.5,-11);
\draw (-3.5,-10) -- (-3.5,-11);
\draw[fill=white] (2,4) rectangle (5,5);
\draw[fill=white] (-2,4) rectangle (-5,5);
\draw[fill=white] (2,-4) rectangle (5,-5);
\draw[fill=white] (-2,-4) rectangle (-5,-5);
\draw (3.5,4) -- (3.5,5);
\draw (-3.5,4) -- (-3.5,5);
\draw (3.5,-4) -- (3.5,-5);
\draw (-3.5,-4) -- (-3.5,-5);
\node at (-4,0)[left]{$\scriptstyle{n}$};
\node at (4,0)[right]{$\scriptstyle{n}$};
\node at (0,3)[below]{$\scriptstyle{n}$};
\node at (0,-3)[above]{$\scriptstyle{n}$};}
\,\Bigg\rangle_{\! 3}
=\Bigg\langle\,
\tikz[baseline=-.6ex, scale=.1]{
\draw[triple={[line width=1.4pt, white] in [line width=2.2pt, black] in [line width=5.4pt, white]}]
(3.5,5) to[out=north, in=south] (-3.5,10) -- (-3.5,12);
\draw[triple={[line width=1.4pt, white] in [line width=2.2pt, black] in [line width=5.4pt, white]}]
(-3.5,5) to[out=north, in=south] (3.5,10) -- (3.5,12);
\draw[triple={[line width=1.4pt, white] in [line width=2.2pt, black] in [line width=5.4pt, white]}]
(3.5,-5) to[out=south, in=north] (-3.5,-10) -- (-3.5,-12);
\draw[triple={[line width=1.4pt, white] in [line width=2.2pt, black] in [line width=5.4pt, white]}]
(-3.5,-5) to[out=south, in=north] (3.5,-10) -- (3.5,-12);
\draw[-<-=.5] (-3.2,5) to[out=south, in=south] (3.2,5);
\draw[->-=.5] (-3.2,-5) to[out=north, in=north] (3.2,-5);
\draw[->-=.5] (3.8,-5) -- (3.8,5);
\draw[-<-=.5] (-3.8,-5) -- (-3.8,5);
\draw[fill=white] (2,10) rectangle (5,11);
\draw[fill=white] (-2,10) rectangle (-5,11);
\draw[fill=white] (2,-10) rectangle (5,-11);
\draw[fill=white] (-2,-10) rectangle (-5,-11);
\draw (3.5,10) -- (3.5,11);
\draw (-3.5,10) -- (-3.5,11);
\draw (3.5,-10) -- (3.5,-11);
\draw (-3.5,-10) -- (-3.5,-11);
\node at (-4,0)[left]{$\scriptstyle{n}$};
\node at (4,0)[right]{$\scriptstyle{n}$};
\node at (0,3)[below]{$\scriptstyle{n}$};
\node at (0,-3)[above]{$\scriptstyle{n}$};}
\,\Bigg\rangle_{\! 3}
=\Bigg\langle\,
\tikz[baseline=-.6ex, scale=.1]{
\draw (2,9) -- (2,10);
\draw (-2,9) -- (-2,10);
\draw (2,-9) -- (2,-10);
\draw (-2,-9) -- (-2,-10);
\draw (4,9) -- (4,10);
\draw (-4,9) -- (-4,10);
\draw (4,-9) -- (4,-10);
\draw (-4,-9) -- (-4,-10);
\draw[white, double=black, double distance=0.4pt, ultra thick] 
(0,7) to[out=east, in=south] (4,8);
\draw[white, double=black, double distance=0.4pt, ultra thick] 
(0,-7) to[out=east, in=north] (4,-8);
\draw[->-=.5, white, double=black, double distance=0.4pt, ultra thick] 
(-2,-8) -- (-2,8);
\draw[-<-=.5, white, double=black, double distance=0.4pt, ultra thick] 
(2,-8) -- (2,8);
\draw[->-=1, white, double=black, double distance=0.4pt, ultra thick] 
(-4,8) to[out=south, in=west] (0,7);
\draw[-<-=1, white, double=black, double distance=0.4pt, ultra thick] 
(-4,-8) to[out=north, in=west] (0,-7);
\draw[fill=white] (1,8) rectangle (5,9);
\draw[fill=white] (-1,8) rectangle (-5,9);
\draw[fill=white] (1,-8) rectangle (5,-9);
\draw[fill=white] (-1,-8) rectangle (-5,-9);
\draw (3,8) -- (3,9);
\draw (-3,8) -- (-3,9);
\draw (3,-8) -- (3,-9);
\draw (-3,-8) -- (-3,-9);
\node at (-2,0)[left]{$\scriptstyle{n}$};
\node at (2,0)[right]{$\scriptstyle{n}$};
}
\,\Bigg\rangle_{\! 3}\\
&=((-1)^nq^{\frac{n^2}{6}})^2((-1)^nq^{-\frac{n^2}{6}})^2
\Bigg\langle\,
\tikz[baseline=-.6ex, scale=.1]{
\draw (2,9) -- (2,10);
\draw (-2,9) -- (-2,10);
\draw (2,-9) -- (2,-10);
\draw (-2,-9) -- (-2,-10);
\draw (4,9) -- (4,10);
\draw (-4,9) -- (-4,10);
\draw (4,-9) -- (4,-10);
\draw (-4,-9) -- (-4,-10);
\draw (4,8) to[out=south, in=north east] (3,4);
\draw (2,8) to[out=south, in=north west] (3,4);
\draw (-4,8) to[out=south, in=north west] (-3,4);
\draw (-2,8) to[out=south, in=north east] (-3,4);
\draw (4,-8) to[out=north, in=south east] (3,-4);
\draw (2,-8) to[out=north, in=south west] (3,-4);
\draw (-4,-8) to[out=north, in=south west] (-3,-4);
\draw (-2,-8) to[out=north, in=south east] (-3,-4);
\draw[-<-=.3] (3,4) to[out=south west, in=south east] (-3,4);
\draw[->-=.3] (3,-4) to[out=north west, in=north east] (-3,-4);
\draw[->-=.3] (3,4) to[out=south east, in=north east] (3,-4);
\draw[-<-=.3] (-3,4) to[out=south west, in=north west] (-3,-4);
\draw[fill=white] (3,5.5) -- (4.5,4) -- (3,2.5) -- (1.5,4) -- cycle;
\draw[fill=white] (-3,5.5) -- (-4.5,4) -- (-3,2.5) -- (-1.5,4) -- cycle;
\draw[fill=white] (3,-5.5) -- (4.5,-4) -- (3,-2.5) -- (1.5,-4) -- cycle;
\draw[fill=white] (-3,-5.5) -- (-4.5,-4) -- (-3,-2.5) -- (-1.5,-4) -- cycle;
\draw (3,5.5) -- (3,2.5);
\draw (-3,5.5) -- (-3,2.5);
\draw (3,-5.5) -- (3,-2.5);
\draw (-3,-5.5) -- (-3,-2.5);
\draw[fill=white] (1,8) rectangle (5,9);
\draw[fill=white] (-1,8) rectangle (-5,9);
\draw[fill=white] (1,-8) rectangle (5,-9);
\draw[fill=white] (-1,-8) rectangle (-5,-9);
\draw (3,8) -- (3,9);
\draw (-3,8) -- (-3,9);
\draw (3,-8) -- (3,-9);
\draw (-3,-8) -- (-3,-9);
\draw[fill=white] (6,-.5) rectangle (3.5,.5);
\draw[fill=white] (-6,-.5) rectangle (-3.5,.5);
\draw[fill=white] (-.5,4) rectangle (.5,1.5);
\draw[fill=white] (-.5,-4) rectangle (.5,-1.5);
\node at (-5,0)[left]{$\scriptstyle{n}$};
\node at (5,0)[right]{$\scriptstyle{n}$};
\node at (0,4)[above]{$\scriptstyle{n}$};
\node at (0,-4)[below]{$\scriptstyle{n}$};
}
\,\Bigg\rangle_{\! 3}\\
&=\Bigg\langle\,
\tikz[baseline=-.6ex, scale=.1]{
\draw[triple={[line width=1.4pt, white] in [line width=2.2pt, black] in [line width=5.4pt, white]}] 
(3,3) to[out=north east, in=south] (5,5);
\draw[triple={[line width=1.4pt, white] in [line width=2.2pt, black] in [line width=5.4pt, white]}] 
(-3,3) to[out=north west, in=south] (-5,5);
\draw[triple={[line width=1.4pt, white] in [line width=2.2pt, black] in [line width=5.4pt, white]}] 
(3,-3) to[out=south east, in=north] (5,-5);
\draw[triple={[line width=1.4pt, white] in [line width=2.2pt, black] in [line width=5.4pt, white]}] 
(-3,-3) to[out=south west, in=north] (-5,-5);
\draw[-<-=.5] (3,3) -- (-3,3);
\draw[-<-=.5] (-3,3) -- (-3,-3);
\draw[-<-=.5] (-3,-3) -- (3,-3);
\draw[-<-=.5] (3,-3) -- (3,3);
\draw[fill=black] (3,3) circle [radius=.8];
\draw[fill=black] (-3,3) circle [radius=.8];
\draw[fill=black] (3,-3) circle [radius=.8];
\draw[fill=black] (-3,-3) circle [radius=.8];
\node at (5,-6)[right]{$\scriptstyle{n}$};
\node at (-5,-6)[left]{$\scriptstyle{n}$};
\node at (5,6)[right]{$\scriptstyle{n}$};
\node at (-5,6)[left]{$\scriptstyle{n}$};
\node at (0,3)[above]{$\scriptstyle{n}$};
\node at (0,-3)[below]{$\scriptstyle{n}$};
\node at (3,0)[right]{$\scriptstyle{n}$};
\node at (-3,0)[left]{$\scriptstyle{n}$};}
\,\Bigg\rangle_{\! 3}
\end{align*}
The above calculation is similar to the final calculation of the proof of Theorem~\ref{oriKVthm}. 
We used Lemma~\ref{twistcoeff}~(1) in the second line.
In the same way, 
we can show 
\[
\Bigg\langle\,
\tikz[baseline=-.6ex, scale=.1]{
\draw[triple={[line width=1.4pt, white] in [line width=2.2pt, black] in [line width=5.4pt, white]}] 
(3,3) to[out=north east, in=south] (5,5);
\draw[triple={[line width=1.4pt, white] in [line width=2.2pt, black] in [line width=5.4pt, white]}] 
(-3,3) to[out=north west, in=south] (-5,5);
\draw[triple={[line width=1.4pt, white] in [line width=2.2pt, black] in [line width=5.4pt, white]}] 
(3,-3) to[out=south east, in=north] (5,-5);
\draw[triple={[line width=1.4pt, white] in [line width=2.2pt, black] in [line width=5.4pt, white]}] 
(-3,-3) to[out=south west, in=north] (-5,-5);
\draw[triple={[line width=1.4pt, white] in [line width=2.2pt, black] in [line width=5.4pt, white]}] 
(-5,5) to[out=north, in=south west] (5,10);
\draw[triple={[line width=1.4pt, white] in [line width=2.2pt, black] in [line width=5.4pt, white]}] 
(5,5) to[out=north, in=south east] (-5,10);
\draw[triple={[line width=1.4pt, white] in [line width=2.2pt, black] in [line width=5.4pt, white]}] 
(-5,-5) to[out=south, in=north west] (5,-10);
\draw[triple={[line width=1.4pt, white] in [line width=2.2pt, black] in [line width=5.4pt, white]}] 
(5,-5) to[out=south, in=north east] (-5,-10);
\draw[->-=.5] (3,3) -- (-3,3);
\draw[->-=.5] (-3,3) -- (-3,-3);
\draw[->-=.5] (-3,-3) -- (3,-3);
\draw[->-=.5] (3,-3) -- (3,3);
\draw[fill=white] (3,3) circle [radius=.8];
\draw[fill=white] (-3,3) circle [radius=.8];
\draw[fill=white] (3,-3) circle [radius=.8];
\draw[fill=white] (-3,-3) circle [radius=.8];
\node at (5,-6)[right]{$\scriptstyle{n}$};
\node at (-5,-6)[left]{$\scriptstyle{n}$};
\node at (5,6)[right]{$\scriptstyle{n}$};
\node at (-5,6)[left]{$\scriptstyle{n}$};
\node at (0,3)[above]{$\scriptstyle{n}$};
\node at (0,-3)[below]{$\scriptstyle{n}$};
\node at (3,0)[right]{$\scriptstyle{n}$};
\node at (-3,0)[left]{$\scriptstyle{n}$};}
\,\Bigg\rangle_{\! 3}
=\Bigg\langle\,
\tikz[baseline=-.6ex, scale=.1]{
\draw[triple={[line width=1.4pt, white] in [line width=2.2pt, black] in [line width=5.4pt, white]}] 
(3,3) to[out=north east, in=south] (5,5);
\draw[triple={[line width=1.4pt, white] in [line width=2.2pt, black] in [line width=5.4pt, white]}] 
(-3,3) to[out=north west, in=south] (-5,5);
\draw[triple={[line width=1.4pt, white] in [line width=2.2pt, black] in [line width=5.4pt, white]}] 
(3,-3) to[out=south east, in=north] (5,-5);
\draw[triple={[line width=1.4pt, white] in [line width=2.2pt, black] in [line width=5.4pt, white]}] 
(-3,-3) to[out=south west, in=north] (-5,-5);
\draw[-<-=.5] (3,3) -- (-3,3);
\draw[-<-=.5] (-3,3) -- (-3,-3);
\draw[-<-=.5] (-3,-3) -- (3,-3);
\draw[-<-=.5] (3,-3) -- (3,3);
\draw[fill=black] (3,3) circle [radius=.8];
\draw[fill=black] (-3,3) circle [radius=.8];
\draw[fill=black] (3,-3) circle [radius=.8];
\draw[fill=black] (-3,-3) circle [radius=.8];
\node at (5,-6)[right]{$\scriptstyle{n}$};
\node at (-5,-6)[left]{$\scriptstyle{n}$};
\node at (5,6)[right]{$\scriptstyle{n}$};
\node at (-5,6)[left]{$\scriptstyle{n}$};
\node at (0,3)[above]{$\scriptstyle{n}$};
\node at (0,-3)[below]{$\scriptstyle{n}$};
\node at (3,0)[right]{$\scriptstyle{n}$};
\node at (-3,0)[left]{$\scriptstyle{n}$};}
\,\Bigg\rangle_{\! 3}\ .
\]
These two identities imply
\[
\Bigg\langle\,
\tikz[baseline=-.6ex, scale=.1]{
\draw[triple={[line width=1.4pt, white] in [line width=2.2pt, black] in [line width=5.4pt, white]}] 
(3,3) to[out=north east, in=south] (5,5);
\draw[triple={[line width=1.4pt, white] in [line width=2.2pt, black] in [line width=5.4pt, white]}] 
(-3,3) to[out=north west, in=south] (-5,5);
\draw[triple={[line width=1.4pt, white] in [line width=2.2pt, black] in [line width=5.4pt, white]}] 
(3,-3) to[out=south east, in=north] (5,-5);
\draw[triple={[line width=1.4pt, white] in [line width=2.2pt, black] in [line width=5.4pt, white]}] 
(-3,-3) to[out=south west, in=north] (-5,-5);
\draw[triple={[line width=1.4pt, white] in [line width=2.2pt, black] in [line width=5.4pt, white]}] 
(5,5) to[out=north, in=south east] (-5,10);
\draw[triple={[line width=1.4pt, white] in [line width=2.2pt, black] in [line width=5.4pt, white]}] 
(-5,5) to[out=north, in=south west] (5,10);
\draw[triple={[line width=1.4pt, white] in [line width=2.2pt, black] in [line width=5.4pt, white]}] 
(5,-5) to[out=south, in=north east] (-5,-10);
\draw[triple={[line width=1.4pt, white] in [line width=2.2pt, black] in [line width=5.4pt, white]}] 
(-5,-5) to[out=south, in=north west] (5,-10);
\draw[-<-=.5] (3,3) -- (-3,3);
\draw[-<-=.5] (-3,3) -- (-3,-3);
\draw[-<-=.5] (-3,-3) -- (3,-3);
\draw[-<-=.5] (3,-3) -- (3,3);
\draw[fill=black] (3,3) circle [radius=.8];
\draw[fill=black] (-3,3) circle [radius=.8];
\draw[fill=black] (3,-3) circle [radius=.8];
\draw[fill=black] (-3,-3) circle [radius=.8];
\node at (5,-6)[right]{$\scriptstyle{n}$};
\node at (-5,-6)[left]{$\scriptstyle{n}$};
\node at (5,6)[right]{$\scriptstyle{n}$};
\node at (-5,6)[left]{$\scriptstyle{n}$};
\node at (0,3)[above]{$\scriptstyle{n}$};
\node at (0,-3)[below]{$\scriptstyle{n}$};
\node at (3,0)[right]{$\scriptstyle{n}$};
\node at (-3,0)[left]{$\scriptstyle{n}$};}
\,\Bigg\rangle_{\! 3}
=\Bigg\langle\,
\tikz[baseline=-.6ex, scale=.1]{
\draw[triple={[line width=1.4pt, white] in [line width=2.2pt, black] in [line width=5.4pt, white]}] 
(3,3) to[out=north east, in=south] (5,5);
\draw[triple={[line width=1.4pt, white] in [line width=2.2pt, black] in [line width=5.4pt, white]}] 
(-3,3) to[out=north west, in=south] (-5,5);
\draw[triple={[line width=1.4pt, white] in [line width=2.2pt, black] in [line width=5.4pt, white]}] 
(3,-3) to[out=south east, in=north] (5,-5);
\draw[triple={[line width=1.4pt, white] in [line width=2.2pt, black] in [line width=5.4pt, white]}] 
(-3,-3) to[out=south west, in=north] (-5,-5);
\draw[->-=.5] (3,3) -- (-3,3);
\draw[->-=.5] (-3,3) -- (-3,-3);
\draw[->-=.5] (-3,-3) -- (3,-3);
\draw[->-=.5] (3,-3) -- (3,3);
\draw[fill=white] (3,3) circle [radius=.8];
\draw[fill=white] (-3,3) circle [radius=.8];
\draw[fill=white] (3,-3) circle [radius=.8];
\draw[fill=white] (-3,-3) circle [radius=.8];
\node at (5,-6)[right]{$\scriptstyle{n}$};
\node at (-5,-6)[left]{$\scriptstyle{n}$};
\node at (5,6)[right]{$\scriptstyle{n}$};
\node at (-5,6)[left]{$\scriptstyle{n}$};
\node at (0,3)[above]{$\scriptstyle{n}$};
\node at (0,-3)[below]{$\scriptstyle{n}$};
\node at (3,0)[right]{$\scriptstyle{n}$};
\node at (-3,0)[left]{$\scriptstyle{n}$};}
\,\Bigg\rangle_{\! 3}
\ \text{and}\ 
\Bigg\langle\,
\tikz[baseline=-.6ex, scale=.1]{
\draw[triple={[line width=1.4pt, white] in [line width=2.2pt, black] in [line width=5.4pt, white]}] 
(3,3) to[out=north east, in=south] (5,5);
\draw[triple={[line width=1.4pt, white] in [line width=2.2pt, black] in [line width=5.4pt, white]}] 
(-3,3) to[out=north west, in=south] (-5,5);
\draw[triple={[line width=1.4pt, white] in [line width=2.2pt, black] in [line width=5.4pt, white]}] 
(3,-3) to[out=south east, in=north] (5,-5);
\draw[triple={[line width=1.4pt, white] in [line width=2.2pt, black] in [line width=5.4pt, white]}] 
(-3,-3) to[out=south west, in=north] (-5,-5);
\draw[triple={[line width=1.4pt, white] in [line width=2.2pt, black] in [line width=5.4pt, white]}] 
(-5,5) to[out=north, in=south west] (5,10);
\draw[triple={[line width=1.4pt, white] in [line width=2.2pt, black] in [line width=5.4pt, white]}] 
(5,5) to[out=north, in=south east] (-5,10);
\draw[triple={[line width=1.4pt, white] in [line width=2.2pt, black] in [line width=5.4pt, white]}] 
(-5,-5) to[out=south, in=north west] (5,-10);
\draw[triple={[line width=1.4pt, white] in [line width=2.2pt, black] in [line width=5.4pt, white]}] 
(5,-5) to[out=south, in=north east] (-5,-10);
\draw[-<-=.5] (3,3) -- (-3,3);
\draw[-<-=.5] (-3,3) -- (-3,-3);
\draw[-<-=.5] (-3,-3) -- (3,-3);
\draw[-<-=.5] (3,-3) -- (3,3);
\draw[fill=black] (3,3) circle [radius=.8];
\draw[fill=black] (-3,3) circle [radius=.8];
\draw[fill=black] (3,-3) circle [radius=.8];
\draw[fill=black] (-3,-3) circle [radius=.8];
\node at (5,-6)[right]{$\scriptstyle{n}$};
\node at (-5,-6)[left]{$\scriptstyle{n}$};
\node at (5,6)[right]{$\scriptstyle{n}$};
\node at (-5,6)[left]{$\scriptstyle{n}$};
\node at (0,3)[above]{$\scriptstyle{n}$};
\node at (0,-3)[below]{$\scriptstyle{n}$};
\node at (3,0)[right]{$\scriptstyle{n}$};
\node at (-3,0)[left]{$\scriptstyle{n}$};}
\,\Bigg\rangle_{\! 3}
=\Bigg\langle\,
\tikz[baseline=-.6ex, scale=.1]{
\draw[triple={[line width=1.4pt, white] in [line width=2.2pt, black] in [line width=5.4pt, white]}] 
(3,3) to[out=north east, in=south] (5,5);
\draw[triple={[line width=1.4pt, white] in [line width=2.2pt, black] in [line width=5.4pt, white]}] 
(-3,3) to[out=north west, in=south] (-5,5);
\draw[triple={[line width=1.4pt, white] in [line width=2.2pt, black] in [line width=5.4pt, white]}] 
(3,-3) to[out=south east, in=north] (5,-5);
\draw[triple={[line width=1.4pt, white] in [line width=2.2pt, black] in [line width=5.4pt, white]}] 
(-3,-3) to[out=south west, in=north] (-5,-5);
\draw[->-=.5] (3,3) -- (-3,3);
\draw[->-=.5] (-3,3) -- (-3,-3);
\draw[->-=.5] (-3,-3) -- (3,-3);
\draw[->-=.5] (3,-3) -- (3,3);
\draw[fill=white] (3,3) circle [radius=.8];
\draw[fill=white] (-3,3) circle [radius=.8];
\draw[fill=white] (3,-3) circle [radius=.8];
\draw[fill=white] (-3,-3) circle [radius=.8];
\node at (5,-6)[right]{$\scriptstyle{n}$};
\node at (-5,-6)[left]{$\scriptstyle{n}$};
\node at (5,6)[right]{$\scriptstyle{n}$};
\node at (-5,6)[left]{$\scriptstyle{n}$};
\node at (0,3)[above]{$\scriptstyle{n}$};
\node at (0,-3)[below]{$\scriptstyle{n}$};
\node at (3,0)[right]{$\scriptstyle{n}$};
\node at (-3,0)[left]{$\scriptstyle{n}$};}
\,\Bigg\rangle_{\! 3}\ .
\]
Consequently, 
$\left[\,
\tikz[baseline=-.6ex, scale=.1]{
\draw (-4,4) -- +(-2,0);
\draw (4,-4) -- +(2,0);
\draw (-4,-4) -- +(-2,0);
\draw (4,4) -- +(2,0);
\draw (0,0) -- (2,3) to[out=north east, in=west] (4,4);
\draw (0,0) -- (-2,3) to[out=north west, in=east] (-4,4);
\draw (0,0) -- (2,-3) to[out=south east, in=west] (4,-4);
\draw (0,0) -- (-2,-3) to[out=south west, in=east] (-4,-4);
\draw[fill=cyan] (0,0) circle [radius=.8];}
\,\right]_{(n,n)}
=
\Bigg\langle\,\tikz[baseline=-.6ex, scale=.1]{
\draw[triple={[line width=1.4pt, white] in [line width=2.2pt, black] in [line width=5.4pt, white]}] 
(3,3) to[out=north east, in=west] (5,4) -- (7,4);
\draw[triple={[line width=1.4pt, white] in [line width=2.2pt, black] in [line width=5.4pt, white]}] 
(-3,3) to[out=north west, in=east] (-5,4) -- (-7,4);
\draw[triple={[line width=1.4pt, white] in [line width=2.2pt, black] in [line width=5.4pt, white]}] 
(3,-3) to[out=south east, in=west] (5,-4) -- (7,-4);
\draw[triple={[line width=1.4pt, white] in [line width=2.2pt, black] in [line width=5.4pt, white]}] 
(-3,-3) to[out=south west, in=east] (-5,-4) -- (-7,-4);
\draw[->-=.5] (3,3) -- (-3,3);
\draw[->-=.5] (-3,3) -- (-3,-3);
\draw[->-=.5] (-3,-3) -- (3,-3);
\draw[->-=.5] (3,-3) -- (3,3);
\draw[fill=white] (3,3) circle [radius=.8];
\draw[fill=white] (-3,3) circle [radius=.8];
\draw[fill=white] (3,-3) circle [radius=.8];
\draw[fill=white] (-3,-3) circle [radius=.8];
\node at (7,-4)[right]{$\scriptstyle{n}$};
\node at (-7,-4)[left]{$\scriptstyle{n}$};
\node at (7,4)[right]{$\scriptstyle{n}$};
\node at (-7,4)[left]{$\scriptstyle{n}$};
\node at (0,3)[above]{$\scriptstyle{n}$};
\node at (0,-3)[below]{$\scriptstyle{n}$};
\node at (3,0)[right]{$\scriptstyle{n}$};
\node at (-3,0)[left]{$\scriptstyle{n}$};
}\,\Bigg\rangle_{\! 3}
+
\Bigg\langle\,\tikz[baseline=-.6ex, scale=.1]{
\draw[triple={[line width=1.4pt, white] in [line width=2.2pt, black] in [line width=5.4pt, white]}] 
(3,3) to[out=north east, in=west] (5,4) -- (7,4);
\draw[triple={[line width=1.4pt, white] in [line width=2.2pt, black] in [line width=5.4pt, white]}] 
(-3,3) to[out=north west, in=east] (-5,4) -- (-7,4);
\draw[triple={[line width=1.4pt, white] in [line width=2.2pt, black] in [line width=5.4pt, white]}] 
(3,-3) to[out=south east, in=west] (5,-4) -- (7,-4);
\draw[triple={[line width=1.4pt, white] in [line width=2.2pt, black] in [line width=5.4pt, white]}] 
(-3,-3) to[out=south west, in=east] (-5,-4) -- (-7,-4);
\draw[-<-=.5] (3,3) -- (-3,3);
\draw[-<-=.5] (-3,3) -- (-3,-3);
\draw[-<-=.5] (-3,-3) -- (3,-3);
\draw[-<-=.5] (3,-3) -- (3,3);
\draw[fill=black] (3,3) circle [radius=.8];
\draw[fill=black] (-3,3) circle [radius=.8];
\draw[fill=black] (3,-3) circle [radius=.8];
\draw[fill=black] (-3,-3) circle [radius=.8];
\node at (7,-4)[right]{$\scriptstyle{n}$};
\node at (-7,-4)[left]{$\scriptstyle{n}$};
\node at (7,4)[right]{$\scriptstyle{n}$};
\node at (-7,4)[left]{$\scriptstyle{n}$};
\node at (0,3)[above]{$\scriptstyle{n}$};
\node at (0,-3)[below]{$\scriptstyle{n}$};
\node at (3,0)[right]{$\scriptstyle{n}$};
\node at (-3,0)[left]{$\scriptstyle{n}$};
}\,\Bigg\rangle_{\! 3}$  is invariant under the Reidemeister move (RV).
\end{proof}

\section{Computing the $A_2$ colored Kauffman-Vogel polynomials}
We define the oriented $4$-valent rigid vertex graph $\operatorname{ST}(k,l)$ and the unoriented $4$-valent rigid vertex graph $\bar{\operatorname{ST}}(k,l)$ as follows:

\[
\operatorname{ST}(k,l)
=
\tikz[baseline=-.6ex, scale=.1]{
\begin{scope}
\draw (0,0) -- (2,3) to[out=north east, in=west] (4,4);
\draw (0,0) -- (-2,3) to[out=north west, in=east] (-4,4);
\draw (0,0) -- (2,-3) to[out=south east, in=west] (4,-4);
\draw (0,0) -- (-2,-3) to[out=south west, in=east] (-4,-4);
\draw (-4,6) -- (4,6);
\draw (-4,8) -- (4,8);
\draw[-<-=.5] (-4,4) to[out=west, in=west] (-4,6);
\draw[-<-=.5] (-4,-4) to[out=west, in=west] (-4,8);
\draw[fill=cyan] (0,0) circle [radius=.8];
\end{scope}
\begin{scope}[xshift=8cm]
\node at (0,0) {$\cdots$};
\node at (0,-4) [below]{$k$ vertices};
\draw (-4,6) -- (4,6);
\draw (-4,8) -- (4,8);
\end{scope}
\begin{scope}[xshift=16cm]
\draw (0,0) -- (2,3) to[out=north east, in=west] (4,4);
\draw (0,0) -- (-2,3) to[out=north west, in=east] (-4,4);
\draw (0,0) -- (2,-3) to[out=south east, in=west] (4,-4);
\draw (0,0) -- (-2,-3) to[out=south west, in=east] (-4,-4);
\draw (-4,6) -- (4,6);
\draw (-4,8) -- (4,8);
\draw[fill=cyan] (0,0) circle [radius=.8];
\end{scope}
\begin{scope}[xshift=24cm]
\draw[->-=.2, white, double=black, double distance=0.4pt, ultra thick] 
(-4,-4) to[out=east, in=west] (4,4);
\draw[-<-=.8, white, double=black, double distance=0.4pt, ultra thick] 
(4,-4) to[out=west, in=east] (-4,4);
\draw (-4,6) -- (4,6);
\draw (-4,8) -- (4,8);
\end{scope}
\begin{scope}[xshift=32cm]
\node at (0,0) {$\cdots$};
\node at (0,-4) [below]{$l$ crossings};
\draw (-4,6) -- (4,6);
\draw (-4,8) -- (4,8);
\end{scope}
\begin{scope}[xshift=40cm]
\draw[white, double=black, double distance=0.4pt, ultra thick] 
(-4,-4) to[out=east, in=west] (4,4);
\draw[white, double=black, double distance=0.4pt, ultra thick] 
(4,-4) to[out=west, in=east] (-4,4);
\draw (-4,6) -- (4,6);
\draw (-4,8) -- (4,8);
\draw (4,4) to[out=east, in=east] (4,6);
\draw (4,-4) to[out=east, in=east] (4,8);
\end{scope}
}
,\]
\[
\bar{\operatorname{ST}}(k,l)
=
\tikz[baseline=-.6ex, scale=.1]{
\begin{scope}
\draw (0,0) -- (2,3) to[out=north east, in=west] (4,4);
\draw (0,0) -- (-2,3) to[out=north west, in=east] (-4,4);
\draw (0,0) -- (2,-3) to[out=south east, in=west] (4,-4);
\draw (0,0) -- (-2,-3) to[out=south west, in=east] (-4,-4);
\draw (-4,6) -- (4,6);
\draw (-4,8) -- (4,8);
\draw (-4,4) to[out=west, in=west] (-4,6);
\draw (-4,-4) to[out=west, in=west] (-4,8);
\draw[fill=cyan] (0,0) circle [radius=.8];
\end{scope}
\begin{scope}[xshift=8cm]
\node at (0,0) {$\cdots$};
\node at (0,-4) [below]{$k$ vertices};
\draw (-4,6) -- (4,6);
\draw (-4,8) -- (4,8);
\end{scope}
\begin{scope}[xshift=16cm]
\draw (0,0) -- (2,3) to[out=north east, in=west] (4,4);
\draw (0,0) -- (-2,3) to[out=north west, in=east] (-4,4);
\draw (0,0) -- (2,-3) to[out=south east, in=west] (4,-4);
\draw (0,0) -- (-2,-3) to[out=south west, in=east] (-4,-4);
\draw (-4,6) -- (4,6);
\draw (-4,8) -- (4,8);
\draw[fill=cyan] (0,0) circle [radius=.8];
\end{scope}
\begin{scope}[xshift=24cm]
\draw[white, double=black, double distance=0.4pt, ultra thick] 
(-4,-4) to[out=east, in=west] (4,4);
\draw[white, double=black, double distance=0.4pt, ultra thick] 
(4,-4) to[out=west, in=east] (-4,4);
\draw (-4,6) -- (4,6);
\draw (-4,8) -- (4,8);
\end{scope}
\begin{scope}[xshift=32cm]
\node at (0,0) {$\cdots$};
\node at (0,-4) [below]{$l$ crossings};
\draw (-4,6) -- (4,6);
\draw (-4,8) -- (4,8);
\end{scope}
\begin{scope}[xshift=40cm]
\draw[white, double=black, double distance=0.4pt, ultra thick] 
(-4,-4) to[out=east, in=west] (4,4);
\draw[white, double=black, double distance=0.4pt, ultra thick] 
(4,-4) to[out=west, in=east] (-4,4);
\draw (-4,6) -- (4,6);
\draw (-4,8) -- (4,8);
\draw (4,4) to[out=east, in=east] (4,6);
\draw (4,-4) to[out=east, in=east] (4,8);
\end{scope}
}
,\]

Elhamdadi and Hajij computed the one-variable Kauffman-Vogel invariant for the Kauffman bracket of $\bar{\operatorname{ST}}(k,l)$ in \cite{ElhamdadiHajij17B}. 
We only compute the one-variable Kauffman-Vogel invariant for the $A_2$ bracket in easy cases.

We use the following formulas to calculate the invariants for some examples. 
Let us denote a $q$-Pochhammer symbol by 
\[
 (q)_k=\prod_{l=1}^{k}(1-q^l)
\]
 and a $q$-binomial coefficient by
\[
 {n\choose k}_q=\frac{(q)_n}{(q)_k(q)_{n-k}}
\]
for $k\leq n$.
We also define a $q$-multinomial coefficient as
\[
 {n\choose n_1,n_2,\dots,n_m}_q=\frac{(q)_n}{(q)_{n_1}(q)_{n_2}\cdots(q)_{n_m}},
\]
where $n_1, n_2,\dots, n_m$ are non-negative integers such that $n_1+n_2+\dots+n_m=n$.

\begin{THM}[{\cite[Theorem~3.11]{Yuasa17}}]\label{coloredA2}
Let $n$ be a positive integer.
\begin{enumerate}
\item 
$\displaystyle
\Bigg\langle\,\tikz[baseline=-.6ex]{
\draw (-.4,.4) -- +(-.2,0);
\draw[->-=.8, white, double=black, double distance=0.4pt, ultra thick] 
(.4,-.4) to[out=west, in=east] (-.4,.4);
\draw (.4,-.4) -- +(.2,0);
\draw (-.4,-.4) -- +(-.2,0);
\draw[->-=.8, white, double=black, double distance=0.4pt, ultra thick] 
(-.4,-.4) to[out=east, in=west] (.4,.4);
\draw (.4,.4) -- +(.2,0);
\draw[fill=white] (.4,-.6) rectangle +(.1,.4);
\draw[fill=white] (-.4,-.6) rectangle +(-.1,.4);
\draw[fill=white] (.4,.6) rectangle +(.1,-.4);
\draw[fill=white] (-.4,.6) rectangle +(-.1,-.4);
\node at (.4,-.6)[right]{$\scriptstyle{n}$};
\node at (-.4,-.6)[left]{$\scriptstyle{n}$};
\node at (.4,.6)[right]{$\scriptstyle{n}$};
\node at (-.4,.6)[left]{$\scriptstyle{n}$};
}\,\Bigg\rangle_{\! 3}
=\sum_{k=0}^{n} (-1)^kq^{\frac{2n^2-6nk+3k^2}{6}}{n\choose k}_q
\Bigg\langle\,\tikz[baseline=-.6ex]{
\draw 
(-.5,.4) -- +(-.2,0)
(.5,-.4) -- +(.2,0)
(-.5,-.4) -- +(-.2,0)
(.5,.4) -- +(.2,0);
\draw[-<-=.5] (-.5,.3) to[out=east, in=east] (-.5,-.3);
\draw[-<-=.5] (.5,.3) to[out=west, in=west] (.5,-.3);
\draw[-<-=.5] (-.5,.5) to[out=east, in=north west] (.0,.0);
\draw[-<-=.5] (.0,.0) to[out=south east, in=west] (.5,-.5);
\draw[-<-=.5] (.5,.5) to[out=west, in=north east] (.0,.0);
\draw[-<-=.5] (.0,.0) to[out=south west, in=east] (-.5,-.5);
\draw[fill=white] (.2,0) -- (0,.2) -- (-.2,0) -- (0,-.2) -- cycle;
\draw (-.2,0) -- (.2,0);
\draw[fill=white] (.5,-.6) rectangle +(.1,.4);
\draw[fill=white] (-.5,-.6) rectangle +(-.1,.4);
\draw[fill=white] (.5,.6) rectangle +(.1,-.4);
\draw[fill=white] (-.5,.6) rectangle +(-.1,-.4);
\node at (.5,-.6)[right]{$\scriptstyle{n}$};
\node at (-.5,-.6)[left]{$\scriptstyle{n}$};
\node at (.5,.6)[right]{$\scriptstyle{n}$};
\node at (-.5,.6)[left]{$\scriptstyle{n}$};
\node at (-.3,0)[left]{$\scriptstyle{n-k}$};
\node at (.3,0)[right]{$\scriptstyle{n-k}$};
\node at (0,.5)[right]{$\scriptstyle{k}$};
\node at (0,.5)[left]{$\scriptstyle{k}$};
\node at (0,-.5)[right]{$\scriptstyle{k}$};
\node at (0,-.5)[left]{$\scriptstyle{k}$};
}\,\Bigg\rangle_{\! 3}
$,
\item 
$\displaystyle
\Bigg\langle\,\tikz[baseline=-.6ex]{
\draw (-.4,-.4) -- +(-.2,0);
\draw[->-=.8, white, double=black, double distance=0.4pt, ultra thick] 
(-.4,-.4) to[out=east, in=west] (.4,.4);
\draw (.4,.4) -- +(.2,0);
\draw (-.4,.4) -- +(-.2,0);
\draw[->-=.8, white, double=black, double distance=0.4pt, ultra thick] 
(.4,-.4) to[out=west, in=east] (-.4,.4);
\draw (.4,-.4) -- +(.2,0);
\draw[fill=white] (.4,-.6) rectangle +(.1,.4);
\draw[fill=white] (-.4,-.6) rectangle +(-.1,.4);
\draw[fill=white] (.4,.6) rectangle +(.1,-.4);
\draw[fill=white] (-.4,.6) rectangle +(-.1,-.4);
\node at (.4,-.6)[right]{$\scriptstyle{n}$};
\node at (-.4,-.6)[left]{$\scriptstyle{n}$};
\node at (.4,.6)[right]{$\scriptstyle{n}$};
\node at (-.4,.6)[left]{$\scriptstyle{n}$};
}\,\Bigg\rangle_{\! 3}
=\sum_{k=0}^{n} (-1)^{k}q^{\frac{-2n^2+3k^2}{6}}{n\choose k}_q
\Bigg\langle\,\tikz[baseline=-.6ex]{
\draw 
(-.5,.4) -- +(-.2,0)
(.5,-.4) -- +(.2,0)
(-.5,-.4) -- +(-.2,0)
(.5,.4) -- +(.2,0);
\draw[-<-=.5] (-.5,.3) to[out=east, in=east] (-.5,-.3);
\draw[-<-=.5] (.5,.3) to[out=west, in=west] (.5,-.3);
\draw[-<-=.5] (-.5,.5) to[out=east, in=north west] (.0,.0);
\draw[-<-=.5] (.0,.0) to[out=south east, in=west] (.5,-.5);
\draw[-<-=.5] (.5,.5) to[out=west, in=north east] (.0,.0);
\draw[-<-=.5] (.0,.0) to[out=south west, in=east] (-.5,-.5);
\draw[fill=white] (.2,0) -- (0,.2) -- (-.2,0) -- (0,-.2) -- cycle;
\draw (-.2,0) -- (.2,0);
\draw[fill=white] (.5,-.6) rectangle +(.1,.4);
\draw[fill=white] (-.5,-.6) rectangle +(-.1,.4);
\draw[fill=white] (.5,.6) rectangle +(.1,-.4);
\draw[fill=white] (-.5,.6) rectangle +(-.1,-.4);
\node at (.5,-.6)[right]{$\scriptstyle{n}$};
\node at (-.5,-.6)[left]{$\scriptstyle{n}$};
\node at (.5,.6)[right]{$\scriptstyle{n}$};
\node at (-.5,.6)[left]{$\scriptstyle{n}$};
\node at (-.3,0)[left]{$\scriptstyle{n-k}$};
\node at (.3,0)[right]{$\scriptstyle{n-k}$};
\node at (0,.5)[right]{$\scriptstyle{k}$};
\node at (0,.5)[left]{$\scriptstyle{k}$};
\node at (0,-.5)[right]{$\scriptstyle{k}$};
\node at (0,-.5)[left]{$\scriptstyle{k}$};
}\,\Bigg\rangle_{\! 3}
$,
\item
$\displaystyle
\Bigg\langle\,\tikz[baseline=-.6ex]{
\draw 
(-.4,.4) -- +(-.2,0)
(-.4,-.4) -- +(-.2,0);
\draw[-<-=.5] (-.4,.4) to[out=east, in=north west] (0,.3);
\draw[->-=.7] (0,.3) -- (.4,.3);
\draw[->-=.4, ->-=.8] (0,.3) -- (0,-.3);
\draw[-<-=.7] (0,-.3) -- (.4,-.3);
\draw[->-=.5] (-.4,-.4) to[out=east, in=south west] (0,-.3);
\node at (.4,-.4)[below right]{$\scriptstyle{n}$};
\node at (-.4,-.6)[left]{$\scriptstyle{n}$};
\node at (.4,.4)[above right]{$\scriptstyle{n}$};
\node at (-.4,.6)[left]{$\scriptstyle{n}$};
\draw[fill=white] (-.2,.2) -- (.1,.2) -- (.1,.5) -- cycle;
\draw[fill=white] (-.2,-.2) -- (.1,-.2) -- (.1,-.5) -- cycle;
\begin{scope}[xshift=.9cm]
\draw 
(.4,-.4) -- +(.2,0)
(.4,.4) -- +(.2,0);
\draw[->-=.5] (.4,.4) to[out=east, in=north west] (0,.3);
\draw[-<-=.7] (0,.3) -- (-.4,.3);
\draw[-<-=.3, -<-=.7] (0,.3) -- (0,-.3);
\draw[->-=.7] (0,-.3) -- (-.4,-.3);
\draw[-<-=.5] (.4,-.4) to[out=east, in=south west] (0,-.3);
\draw[fill=white] (.4,-.6) rectangle +(.1,.4);
\draw[fill=white] (.4,.6) rectangle +(.1,-.4);
\node at (.4,-.6)[right]{$\scriptstyle{n}$};
\node at (.4,.6)[right]{$\scriptstyle{n}$};
\draw[fill=white] (.2,.2) -- (-.1,.2) -- (-.1,.5) -- cycle;
\draw[fill=white] (.2,-.2) -- (-.1,-.2) -- (-.1,-.5) -- cycle;
\draw[fill=white] (.2,-.05) rectangle (-.1,.05);
\end{scope}
\draw[fill=white] (.4,-.5) rectangle +(.1,.3);
\draw[fill=white] (-.4,-.6) rectangle +(-.1,.4);
\draw[fill=white] (.4,.5) rectangle +(.1,-.3);
\draw[fill=white] (-.4,.6) rectangle +(-.1,-.4);
\draw[fill=white] (-.2,-.05) rectangle (.1,.05);
}\,\Bigg\rangle_{\! 3}
=\sum_{k=0}^{n}
\Bigg\langle\,\tikz[baseline=-.6ex]{
\draw 
(-.4,.4) -- +(-.2,0)
(.4,-.4) -- +(.2,0)
(-.4,-.4) -- +(-.2,0)
(.4,.4) -- +(.2,0);
\draw[-<-=.5] (-.4,.5) -- (.4,.5);
\draw[->-=.5] (-.4,-.5) -- (.4,-.5);
\draw[-<-=.5] (-.4,.3) to[out=east, in=east] (-.4,-.3);
\draw[->-=.5] (.4,.3) to[out=west, in=west] (.4,-.3);
\draw[fill=white] (.4,-.6) rectangle +(.1,.4);
\draw[fill=white] (-.4,-.6) rectangle +(-.1,.4);
\draw[fill=white] (.4,.6) rectangle +(.1,-.4);
\draw[fill=white] (-.4,.6) rectangle +(-.1,-.4);
\node at (.4,-.6)[right]{$\scriptstyle{n}$};
\node at (-.4,-.6)[left]{$\scriptstyle{n}$};
\node at (.4,.6)[right]{$\scriptstyle{n}$};
\node at (-.4,.6)[left]{$\scriptstyle{n}$};
\node at (0,.5)[above]{$\scriptstyle{n-k}$};
\node at (0,-.5)[below]{$\scriptstyle{n-k}$};
\node at (-.2,0)[left]{$\scriptstyle{k}$};
\node at (.2,0)[right]{$\scriptstyle{k}$};
}\,\Bigg\rangle_{\! 3}
$,
\item
$\Big\langle\,\tikz[baseline=-.6ex]{
\draw[->-=.7] (-.9,0) -- (-.4,0);
\draw[-<-=.7] (.9,0) -- (.4,0);
\draw[-<-=.7] (-.4,0) to[out=north, in=west] (0,.25);
\draw[->-=.7] (.4,0) to[out=north, in=east] (0,.25);
\draw[-<-=.7] (-.4,0) to[out=south, in=west] (0,-.25);
\draw[->-=.7] (.4,0) to[out=south, in=east] (0,-.25);
\draw[fill=white] (-.05,.1) rectangle (.05,.4);
\draw[fill=white] (-.05,-.1) rectangle (.05,-.4);
\draw[fill=white, xshift=-.4cm] (0:.2) -- (120:.2) -- (240:.2) -- cycle;
\draw[fill=white, xshift=.4cm] (60:.2) -- (180:.2) -- (300:.2) -- cycle;
\draw[fill=white] (-.8,-.2) rectangle (-.7,.2);
\draw[fill=white] (.8,-.2) rectangle (.7,.2);
\node at (0,.2)[above left]{$\scriptstyle{n}$};
\node at (0,-.2)[below left]{$\scriptstyle{n}$};
\node at (-.4,0)[above left]{$\scriptstyle{n}$};
\node at (.4,0)[above right]{$\scriptstyle{n}$};
}\,\Big\rangle_{\! 3}
=
\left[n+1\right]\Big\langle\,\tikz[baseline=-.6ex]{
\draw[->-=.3, ->-=.8] (-.5,0) -- (.5,0);
\draw[fill=white] (-.05,-.2) rectangle (.05,.2);
\node at (0,0)[above right]{$\scriptstyle{n}$};
}\,\Big\rangle_{\! 3}
$,
\item
$\displaystyle\Big\langle\,\tikz[baseline=-.6ex]{
\draw[->-=.5] (0,0) circle [radius=.3];
\draw[fill=white] (.1,-.05) rectangle (.5,.05);
\node at (.3,0)[above right]{$\scriptstyle{n}$};
}\,\Big\rangle_{\! 3}
=\frac{\left[n+1\right]\left[n+2\right]}{\left[2\right]} \emptyset$.
\end{enumerate}
\end{THM}

\begin{THM}[{\cite[Theorem~3.17]{Yuasa17}}]\label{A2mfull}
\begin{align*}
\Bigg\langle\,\tikz[baseline=-.6ex]{
\begin{scope}[xshift=-1cm]
\draw 
(-.5,.4) -- +(-.2,0)
(-.5,-.4) -- +(-.2,0);
\draw[->-=1, white, double=black, double distance=0.4pt, ultra thick] 
(-.5,-.4) to[out=east, in=west] (.0,.4);
\draw[-<-=1, white, double=black, double distance=0.4pt, ultra thick] 
(-.5,.4) to[out=east, in=west] (.0,-.4);
\draw[white, double=black, double distance=0.4pt, ultra thick] 
(0,-.4) to[out=east, in=west] (.5,.4);
\draw[white, double=black, double distance=0.4pt, ultra thick] 
(0,.4) to[out=east, in=west] (.5,-.4);
\draw[fill=white] (-.5,-.6) rectangle +(-.1,.4);
\draw[fill=white] (-.5,.6) rectangle +(-.1,-.4);
\node at (-.5,-.6)[left]{$\scriptstyle{n}$};
\node at (-.5,.6)[left]{$\scriptstyle{n}$};
\end{scope}
\node at (.0,.0){$\cdots$};
\node at (.0,-.4)[below]{$\scriptstyle{l\text{ full twists}}$};
\begin{scope}[xshift=1cm]
\draw
(.5,-.4) -- +(.2,0)
(.5,.4) -- +(.2,0);
\draw[->-=1, white, double=black, double distance=0.4pt, ultra thick] 
(-.5,-.4) to[out=east, in=west] (.0,.4);
\draw[-<-=1, white, double=black, double distance=0.4pt, ultra thick] 
(-.5,.4) to[out=east, in=west] (.0,-.4);
\draw[white, double=black, double distance=0.4pt, ultra thick] 
(0,-.4) to[out=east, in=west] (.5,.4);
\draw[white, double=black, double distance=0.4pt, ultra thick] 
(0,.4) to[out=east, in=west] (.5,-.4);
\draw[fill=white] (.5,-.6) rectangle +(.1,.4);
\draw[fill=white] (.5,.6) rectangle +(.1,-.4);
\node at (.5,-.6)[right]{$\scriptstyle{n}$};
\node at (.5,.6)[right]{$\scriptstyle{n}$};
\end{scope}
}\,\Bigg\rangle_{\! 3}
&=q^{-\frac{2l}{3}(n^2+3n)}
\sum_{0\leq k_l\leq \cdots\leq k_1\leq n}
q^{n-k_l}
q^{\sum_{i=1}^{l}(k_i^2+2k_i)}\\
&\qquad\times\frac{(q)_n}{(q)_{k_l}}
{n \choose k_1',k_2',\dots,k_l',k_l}_{q}
\Bigg\langle\,\tikz[baseline=-.6ex]{
\draw
(-.4,.4) -- +(-.2,0)
(.4,-.4) -- +(.2,0)
(-.4,-.4) -- +(-.2,0)
(.4,.4) -- +(.2,0);
\draw[-<-=.5] (-.4,.5) -- (.4,.5);
\draw[->-=.5] (-.4,-.5) -- (.4,-.5);
\draw[-<-=.5] (-.4,.3) to[out=east, in=east] (-.4,-.3);
\draw[->-=.5] (.4,.3) to[out=west, in=west] (.4,-.3);
\draw[fill=white] (.4,-.6) rectangle +(.1,.4);
\draw[fill=white] (-.4,-.6) rectangle +(-.1,.4);
\draw[fill=white] (.4,.6) rectangle +(.1,-.4);
\draw[fill=white] (-.4,.6) rectangle +(-.1,-.4);
\node at (.4,-.6)[right]{$\scriptstyle{n}$};
\node at (-.4,-.6)[left]{$\scriptstyle{n}$};
\node at (.4,.6)[right]{$\scriptstyle{n}$};
\node at (-.4,.6)[left]{$\scriptstyle{n}$};
\node at (0,.5)[above]{$\scriptstyle{k_l}$};
\node at (0,-.5)[below]{$\scriptstyle{k_l}$};
\node at (-.2,0)[left]{$\scriptstyle{n-k_l}$};
\node at (.2,0)[right]{$\scriptstyle{n-k_l}$};
}\,\Bigg\rangle_{\! 3},
\end{align*}
where $k_i, k_i'$ are integers such that $k_0=n$, $k_{i+1}'=k_i-k_{i+1}$ for $i=0,1,\dots,l-1$.
\end{THM}
\begin{THM}{\cite[Theorem~4.2]{Yuasa17}}\label{A2bubble}
\[
\Bigg\langle\,\tikz[baseline=-.6ex, scale=0.8]{
\draw (-.4,.5) -- +(-.2,0);
\draw (.4,-.5) -- +(.2,0);
\draw (-.4,-.5) -- +(-.2,0);
\draw (.4,.5) -- +(.2,0);
\draw[-<-=.5] (-.4,.5) -- (0,.5);
\draw[-<-=.5] (0,.5) -- (.4,.5);
\draw[->-=.5] (-.4,-.5) -- (0,-.5);
\draw[->-=.5] (0,-.5) -- (.4,-.5);
\draw[-<-=.5] (.05,.3) to[out=east, in=east] (.05,.-.3);
\draw[->-=.5] (-.05,.3) to[out=west, in=west] (-.05,-.3);
\draw[fill=white] (-.4,.3) rectangle +(-.1,.3);
\draw[fill=white] (.4,.3) rectangle +(.1,.3);
\draw[fill=white] (-.4,-.3) rectangle +(-.1,-.3);
\draw[fill=white] (.4,-.3) rectangle +(.1,-.3);
\draw[fill=white] (-.05,.2) rectangle +(.1,.4);
\draw[fill=white] (-.05,-.2) rectangle +(.1,-.4);
\node at (.4,-.5)[below right]{$\scriptstyle{m-l}$};
\node at (-.4,-.5)[below left]{$\scriptstyle{m-k}$};
\node at (.4,.5)[above right]{$\scriptstyle{n-l}$};
\node at (-.4,.5)[above left]{$\scriptstyle{n-k}$};
\node at (-.2,0)[left]{$\scriptstyle{k}$};
\node at (.2,0)[right]{$\scriptstyle{l}$};
\node at (0,-.6)[below]{$\scriptstyle{m}$};
\node at (0,.6)[above]{$\scriptstyle{n}$};
}\,\Bigg\rangle_{\! 3}
=
\sum_{t=\max\{k, l\}}^{\min\{k+l, n, m\}}
\frac{{n\brack t}{m\brack t}{t\brack k}{t\brack l}{n+m-t+2\brack n+m-k-l+2}}{{n\brack k}{m\brack k}{n\brack l}{m\brack l}}
\Bigg\langle\,\tikz[baseline=-.6ex, scale=0.8]{
\draw (-.4,.4) -- +(-.2,0);
\draw (.4,-.4) -- +(.2,0);
\draw (-.4,-.4) -- +(-.2,0);
\draw (.4,.4) -- +(.2,0);
\draw[-<-=.5] (-.4,.5) -- (.4,.5);
\draw[->-=.5] (-.4,-.5)  -- (.4,-.5);
\draw[-<-=.5] (-.4,.3) to[out=east, in=east] (-.4,.-.3);
\draw[->-=.5] (.4,.3) to[out=west, in=west] (.4,-.3);
\draw[fill=white] (-.4,.2) rectangle +(-.1,.4);
\draw[fill=white] (.4,.2) rectangle +(.1,.4);
\draw[fill=white] (-.4,-.2) rectangle +(-.1,-.4);
\draw[fill=white] (.4,-.2) rectangle +(.1,-.4);
\node at (.4,-.6)[right]{$\scriptstyle{m-l}$};
\node at (-.4,-.6)[left]{$\scriptstyle{m-k}$};
\node at (.4,.6)[right]{$\scriptstyle{n-l}$};
\node at (-.4,.6)[left]{$\scriptstyle{n-k}$};
\node at (-.2,0)[left]{$\scriptstyle{t-k}$};
\node at (.2,0)[right]{$\scriptstyle{t-l}$};
\node at (0,-.5)[below]{$\scriptstyle{m-t}$};
\node at (0,.5)[above]{$\scriptstyle{n-t}$};
}\,\Bigg\rangle_{\! 3}
\]
\end{THM}

These formulas work for computations of the one-variable Kauffman-Vogel polynomials for $A_2$. 
As easy examples, 
we compute $\left[\operatorname{ST}(k,l)\right]_{m}^{(m)}$ (see Remark~\ref{general}) and $\left[\bar{\operatorname{ST}}(1,2l)\right]_{(n,n)}$.

\subsection *{A computation of $\left[\operatorname{ST}(k,l)\right]_{m}^{(m)}$}
From Theorem~\ref{coloredA2}~(4) and Lemma~\ref{coloredvertex}~(6), 
\[
\left[\,\tikz[baseline=-.6ex, scale=.1]{
\draw[->] (0,0) -- (2,3) to[out=north east, in=west] (4,4);
\draw[-<-=.8] (0,0) -- (-2,3) to[out=north west, in=east] (-4,4);
\draw[->] (0,0) -- (2,-3) to[out=south east, in=west] (4,-4);
\draw[-<-=.8] (0,0) -- (-2,-3) to[out=south west, in=east] (-4,-4);
\draw[fill=cyan] (0,0) circle [radius=.8];
\begin{scope}[xshift=8cm]
\draw[->-=.8] (0,0) -- (2,3) to[out=north east, in=west] (4,4);
\draw (0,0) -- (-2,3) to[out=north west, in=east] (-4,4);
\draw[->-=.8] (0,0) -- (2,-3) to[out=south east, in=west] (4,-4);
\draw (0,0) -- (-2,-3) to[out=south west, in=east] (-4,-4);
\draw[fill=cyan] (0,0) circle [radius=.8];
\end{scope}
}\,\right]_{m}^{(m)}
=
\Bigg\langle\,
\tikz[baseline=-.6ex, scale=.1]{
\draw[->-=.5] (-5,4) to[out=east, in=north west] (0,0);
\draw[-<-=.5] (0,0) to[out=south east, in=west] (5,-4);
\draw[-<-=.5] (5,4) to[out=west, in=north east] (0,0);
\draw[->-=.5] (0,0) to[out=south west, in=east] (-5,-4);
\draw[fill=white] (2,0) -- (0,2) -- (-2,0) -- (0,-2) -- cycle;
\draw (0,-2) -- (0,2);
\draw[fill=white] (5,-6) rectangle +(1,4);
\draw[fill=white] (-5,-6) rectangle +(-1,4);
\draw[fill=white] (5,6) rectangle +(1,-4);
\draw[fill=white] (-5,6) rectangle +(-1,-4);
\node at (0,5)[right]{$\scriptstyle{m}$};
\node at (0,5)[left]{$\scriptstyle{m}$};
\node at (0,-5)[right]{$\scriptstyle{m}$};
\node at (0,-5)[left]{$\scriptstyle{m}$};
\begin{scope}[xshift=11cm]
\draw[->-=.5] (-5,4) to[out=east, in=north west] (0,0);
\draw[-<-=.5] (0,0) to[out=south east, in=west] (5,-4);
\draw[-<-=.5] (5,4) to[out=west, in=north east] (0,0);
\draw[->-=.5] (0,0) to[out=south west, in=east] (-5,-4);
\draw[fill=white] (2,0) -- (0,2) -- (-2,0) -- (0,-2) -- cycle;
\draw (0,-2) -- (0,2);
\draw[fill=white] (5,-6) rectangle +(1,4);
\draw[fill=white] (5,6) rectangle +(1,-4);
\node at (0,5)[right]{$\scriptstyle{m}$};
\node at (0,5)[left]{$\scriptstyle{m}$};
\node at (0,-5)[right]{$\scriptstyle{m}$};
\node at (0,-5)[left]{$\scriptstyle{m}$};
\end{scope}
}
\,\Bigg\rangle_{\! 3}
=\left[m+1\right]
\Bigg\langle\,
\tikz[baseline=-.6ex, scale=.1]{
\draw[->-=.5] (-5,4) to[out=east, in=north west] (0,0);
\draw[-<-=.5] (0,0) to[out=south east, in=west] (5,-4);
\draw[-<-=.5] (5,4) to[out=west, in=north east] (0,0);
\draw[->-=.5] (0,0) to[out=south west, in=east] (-5,-4);
\draw[fill=white] (2,0) -- (0,2) -- (-2,0) -- (0,-2) -- cycle;
\draw (0,-2) -- (0,2);
\draw[fill=white] (5,-6) rectangle +(1,4);
\draw[fill=white] (-5,-6) rectangle +(-1,4);
\draw[fill=white] (5,6) rectangle +(1,-4);
\draw[fill=white] (-5,6) rectangle +(-1,-4);
\node at (0,5)[right]{$\scriptstyle{m}$};
\node at (0,5)[left]{$\scriptstyle{m}$};
\node at (0,-5)[right]{$\scriptstyle{m}$};
\node at (0,-5)[left]{$\scriptstyle{m}$};
}
\,\Bigg\rangle_{\! 3}
=\left[m+1\right]
\left[\,\tikz[baseline=-.6ex, scale=.1]{
\draw[->-=.8] (0,0) -- (2,3) to[out=north east, in=west] (4,4);
\draw[-<-=.8] (0,0) -- (-2,3) to[out=north west, in=east] (-4,4);
\draw[->-=.8] (0,0) -- (2,-3) to[out=south east, in=west] (4,-4);
\draw[-<-=.8] (0,0) -- (-2,-3) to[out=south west, in=east] (-4,-4);
\draw[fill=cyan] (0,0) circle [radius=.8];
}\,\right]_{m}^{(m)}.
\]
We obtain 
$
\left[\operatorname{ST}(k,l)\right]_{m}^{(m)}=\left[m+1\right]^{k-1}\left[\operatorname{ST}(1,l)\right]_{m}^{(m)}
$ . 
By using Lemma~\ref{twistcoeff}~(2) and Lemma~\ref{coloredvertex}~(6), 
$\left[\operatorname{ST}(1,l)\right]_{m}^{(m)}=(-1)^{lm}q^{-\frac{l(m^2+3m)}{6}}\left[m+1\right]$. 
Therefore, 
\[
 \left[\operatorname{ST}(k,l)\right]_{m}^{(m)}=(-1)^{lm}q^{-\frac{l(m^2+3m)}{6}}\left[m+1\right]^k.
\]

\subsection *{A computation of $\left[\bar{\operatorname{ST}}(1,2l)\right]_{(n,n)}$}

We prepare an easy lemma for colored trivalent graphs.
\begin{LEM}\label{change}
For $0\leq i\leq n$,\\
$\bigg\langle\,
\tikz[baseline=-.6ex, scale=0.3]{
\draw[-|-=.5, triple={[line width=1.4pt, white] in [line width=2.2pt, black] in [line width=5.4pt, white]}]
(-1,0) -- (1,0);
\draw[-<-=.5] (-2,1) to[out=east, in=north west] (-1,0);
\draw[->-=.5] (-2,-1) to[out=east, in=south west] (-1,0);
\draw[-<-=.5] (1,0) to[out=north east, in=west] (2,1);
\draw[->-=.5] (1,0) to[out=south east, in=west] (2,-1);
\draw[fill=black] (-1,0)  circle (.2);
\draw[fill=black] (1,0)  circle (.2);
\node at (-2,1) [above]{$\scriptstyle{n}$};
\node at (-2,-1) [above]{$\scriptstyle{n}$};
\node at (2,1) [above]{$\scriptstyle{n}$};
\node at (2,-1) [above]{$\scriptstyle{n}$};
\node at (0,0) [above]{$\scriptstyle{i}$};}
\,\bigg\rangle_3
=
\bigg\langle\,
\tikz[baseline=-.6ex, scale=0.3]{
\draw[triple={[line width=1.4pt, white] in [line width=2.2pt, black] in [line width=5.4pt, white]}]
(-1,0) -- (1,0);
\draw[-<-=.5] (-2,1) to[out=east, in=north west] (-1,0);
\draw[->-=.5] (-2,-1) to[out=east, in=south west] (-1,0);
\draw[-<-=.5] (1,0) to[out=north east, in=west] (2,1);
\draw[->-=.5] (1,0) to[out=south east, in=west] (2,-1);
\draw[fill=white] (-1,0)  circle (.2);
\draw[fill=white] (1,0)  circle (.2);
\node at (-2,1) [above]{$\scriptstyle{n}$};
\node at (-2,-1) [above]{$\scriptstyle{n}$};
\node at (2,1) [above]{$\scriptstyle{n}$};
\node at (2,-1) [above]{$\scriptstyle{n}$};
\node at (0,0) [above]{$\scriptstyle{i}$};}
\,\bigg\rangle_3$
\ and\ 
$\bigg\langle\,
\tikz[baseline=-.6ex, scale=0.3]{
\draw[triple={[line width=1.4pt, white] in [line width=2.2pt, black] in [line width=5.4pt, white]}]
(-1,0) -- (1,0);
\draw[->-=.5] (-2,1) to[out=east, in=north west] (-1,0);
\draw[-<-=.5] (-2,-1) to[out=east, in=south west] (-1,0);
\draw[->-=.5] (1,0) to[out=north east, in=west] (2,1);
\draw[-<-=.5] (1,0) to[out=south east, in=west] (2,-1);
\draw[fill=black] (-1,0)  circle (.2);
\draw[fill=black] (1,0)  circle (.2);
\node at (-2,1) [above]{$\scriptstyle{n}$};
\node at (-2,-1) [above]{$\scriptstyle{n}$};
\node at (2,1) [above]{$\scriptstyle{n}$};
\node at (2,-1) [above]{$\scriptstyle{n}$};
\node at (0,0) [above]{$\scriptstyle{i}$};}
\,\bigg\rangle_3
=
\bigg\langle\,
\tikz[baseline=-.6ex, scale=0.3]{
\draw[-|-=.5, triple={[line width=1.4pt, white] in [line width=2.2pt, black] in [line width=5.4pt, white]}]
(-1,0) -- (1,0);
\draw[->-=.5] (-2,1) to[out=east, in=north west] (-1,0);
\draw[-<-=.5] (-2,-1) to[out=east, in=south west] (-1,0);
\draw[->-=.5] (1,0) to[out=north east, in=west] (2,1);
\draw[-<-=.5] (1,0) to[out=south east, in=west] (2,-1);
\draw[fill=white] (-1,0)  circle (.2);
\draw[fill=white] (1,0)  circle (.2);
\node at (-2,1) [above]{$\scriptstyle{n}$};
\node at (-2,-1) [above]{$\scriptstyle{n}$};
\node at (2,1) [above]{$\scriptstyle{n}$};
\node at (2,-1) [above]{$\scriptstyle{n}$};
\node at (0,0) [above]{$\scriptstyle{i}$};}
\,\bigg\rangle_3$.
\end{LEM}

\begin{proof}
Thus Lemma follows from 
\[
\Bigg\langle\,
\tikz[baseline=-.6ex, scale=.1]{
\draw[-<-=.5] (-5,4) to[out=east, in=north west] (0,0);
\draw[-<-=.5] (0,0) to[out=south east, in=west] (5,-4);
\draw[-<-=.5] (5,4) to[out=west, in=north east] (0,0);
\draw[-<-=.5] (0,0) to[out=south west, in=east] (-5,-4);
\draw[fill=white] (2,0) -- (0,2) -- (-2,0) -- (0,-2) -- cycle;
\draw (-2,0) -- (2,0);
\draw[fill=white] (-5,-6) rectangle +(-1,4);
\draw[fill=white] (-5,6) rectangle +(-1,-4);
\draw[fill=white] (5,-6) rectangle (6,6);
\draw (5,0) -- (6,0);
\node at (0,5)[right]{$\scriptstyle{i}$};
\node at (0,5)[left]{$\scriptstyle{i}$};
\node at (0,-5)[right]{$\scriptstyle{i}$};
\node at (0,-5)[left]{$\scriptstyle{i}$};
\begin{scope}[xshift=11cm]
\draw[->-=.5] (-5,4) to[out=east, in=north west] (0,0);
\draw[->-=.5] (0,0) to[out=south east, in=west] (5,-4);
\draw[->-=.5] (5,4) to[out=west, in=north east] (0,0);
\draw[->-=.5] (0,0) to[out=south west, in=east] (-5,-4);
\draw[fill=white] (2,0) -- (0,2) -- (-2,0) -- (0,-2) -- cycle;
\draw (-2,0) -- (2,0);
\draw[fill=white] (5,-6) rectangle +(1,4);
\draw[fill=white] (5,6) rectangle +(1,-4);
\node at (0,5)[right]{$\scriptstyle{i}$};
\node at (0,5)[left]{$\scriptstyle{i}$};
\node at (0,-5)[right]{$\scriptstyle{i}$};
\node at (0,-5)[left]{$\scriptstyle{i}$};
\end{scope}
}
\,\Bigg\rangle_{\! 3}
=
\Bigg\langle\,
\tikz[baseline=-.6ex, scale=0.1]{
\draw[->-=.8, white, double=black, double distance=0.4pt, ultra thick] 
(-4,-4) to[out=east, in=west] (4,4);
\draw[->-=.8, white, double=black, double distance=0.4pt, ultra thick] 
(4,-4) to[out=west, in=east] (-4,4);
\draw[fill=white] (-4,-6) rectangle +(-1,4);
\draw[fill=white] (-4,6) rectangle +(-1,-4);
\draw[fill=white] (4,-6) rectangle (5,6);
\draw (4,0) -- (5,0);
\node at (4,-5)[left]{$\scriptstyle{i}$};
\node at (-4,-5)[right]{$\scriptstyle{i}$};
\node at (4,5)[left]{$\scriptstyle{i}$};
\node at (-4,5)[right]{$\scriptstyle{i}$};
\begin{scope}[xshift=9cm]
\draw[-<-=.8, white, double=black, double distance=0.4pt, ultra thick] 
(4,-4) to[out=west, in=east] (-4,4);
\draw[-<-=.8, white, double=black, double distance=0.4pt, ultra thick] 
(-4,-4) to[out=east, in=west] (4,4);
\draw[fill=white] (4,-6) rectangle +(1,4);
\draw[fill=white] (4,6) rectangle +(1,-4);
\node at (4,-5)[left]{$\scriptstyle{i}$};
\node at (-4,-5)[right]{$\scriptstyle{i}$};
\node at (4,5)[left]{$\scriptstyle{i}$};
\node at (-4,5)[right]{$\scriptstyle{i}$};
\end{scope}}
\,\Bigg\rangle_{\! 3}
=
\Bigg\langle\,
\tikz[baseline=-.6ex, scale=.5]{
\draw[-<-=.8] (-.6,.4) -- (.6,.4);
\draw[->-=.8] (-.6,-.4) -- (.6,-.4);
\draw[fill=white] (-.1,-.6) rectangle (.1,.6);
\draw (-.1,.0) -- (.1,.0);
\node at (.4,-.6)[right]{$\scriptstyle{i}$};
\node at (-.4,-.6)[left]{$\scriptstyle{i}$};
\node at (.4,.6)[right]{$\scriptstyle{i}$};
\node at (-.4,.6)[left]{$\scriptstyle{i}$};}
\,\Bigg\rangle_{\! 3}.
\]
The first equation is obtained by applying Theorem~\ref{coloredA2}~(1) and (2) to the center tangle. 
We expand the clasp of type $(i,i)$ in the center tangle by using Definition~\ref{doubleA2clasp} and use Lemme~\ref{A2clasplem} (1) and (3).
Thus, 
we obtain the second equation.
\end{proof}

In the same computation to the proof of Theorem~\ref{unoriKVthm}, 
we can see 
\[
\Bigg\langle\,
\tikz[baseline=-.6ex, scale=.1]{
\draw[triple={[line width=1.4pt, white] in [line width=2.2pt, black] in [line width=5.4pt, white]}] 
(3,3) to[out=north east, in=west] (5,4) -- (7,4);
\draw[triple={[line width=1.4pt, white] in [line width=2.2pt, black] in [line width=5.4pt, white]}] 
(3,-3) to[out=south east, in=west] (5,-4) -- (7,-4);
\draw[triple={[line width=1.4pt, white] in [line width=2.2pt, black] in [line width=5.4pt, white]}] 
(7,-4) to[out=east, in=west] (13,4);
\draw[triple={[line width=1.4pt, white] in [line width=2.2pt, black] in [line width=5.4pt, white]}] 
(7,4) to[out=east, in=west] (13,-4);
\draw[->-=.5] (3,3) -- (-3,3);
\draw[->-=.5] (-3,-3) -- (3,-3);
\draw[->-=.5] (3,-3) -- (3,3);
\draw[fill=white] (3,3) circle [radius=.8];
\draw[fill=white] (3,-3) circle [radius=.8];
\node at (13,-4)[right]{$\scriptstyle{n}$};
\node at (13,4)[right]{$\scriptstyle{n}$};
\node at (0,3)[above]{$\scriptstyle{n}$};
\node at (0,-3)[below]{$\scriptstyle{n}$};
\node at (3,0)[right]{$\scriptstyle{n}$};}
\,\Bigg\rangle_{\! 3}
=q^{-\frac{2n^2+3n}{3}}
\Bigg\langle\,
\tikz[baseline=-.6ex, scale=.1]{
\draw[triple={[line width=1.4pt, white] in [line width=2.2pt, black] in [line width=5.4pt, white]}] 
(3,3) to[out=north east, in=west] (5,4) -- (7,4);
\draw[triple={[line width=1.4pt, white] in [line width=2.2pt, black] in [line width=5.4pt, white]}] 
(3,-3) to[out=south east, in=west] (5,-4) -- (7,-4);
\draw[->-=.3, white, double=black, double distance=0.4pt, ultra thick] 
(-3,-3) to[out=east, in=west] (3,3);
\draw[-<-=.3, white, double=black, double distance=0.4pt, ultra thick] 
(-3,3) to[out=east, in=west] (3,-3);
\draw[-<-=.5] (3,-3) -- (3,3);
\draw[fill=black] (3,3) circle [radius=.8];
\draw[fill=black] (3,-3) circle [radius=.8];
\node at (7,-4)[right]{$\scriptstyle{n}$};
\node at (7,4)[right]{$\scriptstyle{n}$};
\node at (0,3)[above]{$\scriptstyle{n}$};
\node at (0,-3)[below]{$\scriptstyle{n}$};
\node at (3,0)[right]{$\scriptstyle{n}$};}
\,\Bigg\rangle_{\! 3}
\]
and
\[
\Bigg\langle\,
\tikz[baseline=-.6ex, scale=.1]{
\draw[triple={[line width=1.4pt, white] in [line width=2.2pt, black] in [line width=5.4pt, white]}] 
(3,3) to[out=north east, in=west] (5,4) -- (7,4);
\draw[triple={[line width=1.4pt, white] in [line width=2.2pt, black] in [line width=5.4pt, white]}] 
(3,-3) to[out=south east, in=west] (5,-4) -- (7,-4);
\draw[triple={[line width=1.4pt, white] in [line width=2.2pt, black] in [line width=5.4pt, white]}] 
(7,-4) to[out=east, in=west] (13,4);
\draw[triple={[line width=1.4pt, white] in [line width=2.2pt, black] in [line width=5.4pt, white]}] 
(7,4) to[out=east, in=west] (13,-4);
\draw[-<-=.5] (3,3) -- (-3,3);
\draw[-<-=.5] (-3,-3) -- (3,-3);
\draw[-<-=.5] (3,-3) -- (3,3);
\draw[fill=black] (3,3) circle [radius=.8];
\draw[fill=black] (3,-3) circle [radius=.8];
\node at (13,-4)[right]{$\scriptstyle{n}$};
\node at (13,4)[right]{$\scriptstyle{n}$};
\node at (0,3)[above]{$\scriptstyle{n}$};
\node at (0,-3)[below]{$\scriptstyle{n}$};
\node at (3,0)[right]{$\scriptstyle{n}$};}
\,\Bigg\rangle_{\! 3}
=q^{-\frac{2n^2+3n}{3}}
\Bigg\langle\,
\tikz[baseline=-.6ex, scale=.1]{
\draw[triple={[line width=1.4pt, white] in [line width=2.2pt, black] in [line width=5.4pt, white]}] 
(3,3) to[out=north east, in=west] (5,4) -- (7,4);
\draw[triple={[line width=1.4pt, white] in [line width=2.2pt, black] in [line width=5.4pt, white]}] 
(3,-3) to[out=south east, in=west] (5,-4) -- (7,-4);
\draw[-<-=.3, white, double=black, double distance=0.4pt, ultra thick] 
(-3,-3) to[out=east, in=west] (3,3);
\draw[->-=.3, white, double=black, double distance=0.4pt, ultra thick] 
(-3,3) to[out=east, in=west] (3,-3);
\draw[->-=.5] (3,-3) -- (3,3);
\draw[fill=white] (3,3) circle [radius=.8];
\draw[fill=white] (3,-3) circle [radius=.8];
\node at (7,-4)[right]{$\scriptstyle{n}$};
\node at (7,4)[right]{$\scriptstyle{n}$};
\node at (0,3)[above]{$\scriptstyle{n}$};
\node at (0,-3)[below]{$\scriptstyle{n}$};
\node at (3,0)[right]{$\scriptstyle{n}$};}
\,\Bigg\rangle_{\! 3}
\]
The above equations and Theorem~\ref{A2mfull} derives:
\begin{align*}
\Bigg\langle\,
\tikz[baseline=-.6ex, scale=.1]{
\draw[triple={[line width=1.4pt, white] in [line width=2.2pt, black] in [line width=5.4pt, white]}] 
(3,3) to[out=north east, in=west] (5,4) -- (7,4);
\draw[triple={[line width=1.4pt, white] in [line width=2.2pt, black] in [line width=5.4pt, white]}] 
(-3,3) to[out=north west, in=east] (-5,4) -- (-7,4);
\draw[triple={[line width=1.4pt, white] in [line width=2.2pt, black] in [line width=5.4pt, white]}] 
(3,-3) to[out=south east, in=west] (5,-4) -- (7,-4);
\draw[triple={[line width=1.4pt, white] in [line width=2.2pt, black] in [line width=5.4pt, white]}] 
(-3,-3) to[out=south west, in=east] (-5,-4) -- (-7,-4);
\draw[triple={[line width=1.4pt, white] in [line width=2.2pt, black] in [line width=5.4pt, white]}] 
(7,-4) to[out=east, in=west] (13,4);
\draw[triple={[line width=1.4pt, white] in [line width=2.2pt, black] in [line width=5.4pt, white]}] 
(7,4) to[out=east, in=west] (13,-4);
\begin{scope}[xshift=6cm]
\node at (10,0) {$\dots$};
\node at (10,-5) [below]{$2l$ crossings};
\end{scope}
\begin{scope}[xshift=12cm]
\draw[triple={[line width=1.4pt, white] in [line width=2.2pt, black] in [line width=5.4pt, white]}] 
(7,-4) to[out=east, in=west] (13,4);
\draw[triple={[line width=1.4pt, white] in [line width=2.2pt, black] in [line width=5.4pt, white]}] 
(7,4) to[out=east, in=west] (13,-4);
\node at (13,-4)[right]{$\scriptstyle{n}$};
\node at (13,4)[right]{$\scriptstyle{n}$};
\end{scope}
\draw[->-=.5] (3,3) -- (-3,3);
\draw[->-=.5] (-3,3) -- (-3,-3);
\draw[->-=.5] (-3,-3) -- (3,-3);
\draw[->-=.5] (3,-3) -- (3,3);
\draw[fill=white] (3,3) circle [radius=.8];
\draw[fill=white] (-3,3) circle [radius=.8];
\draw[fill=white] (3,-3) circle [radius=.8];
\draw[fill=white] (-3,-3) circle [radius=.8];
\node at (-7,-4)[left]{$\scriptstyle{n}$};
\node at (-7,4)[left]{$\scriptstyle{n}$};
\node at (0,3)[above]{$\scriptstyle{n}$};
\node at (0,-3)[below]{$\scriptstyle{n}$};
\node at (3,0)[right]{$\scriptstyle{n}$};
\node at (-3,0)[left]{$\scriptstyle{n}$};}
\,\Bigg\rangle_{\! 3}
&=q^{-\frac{2l}{3}(2n^2+3n)}
\Bigg\langle\,
\tikz[baseline=-.6ex, scale=.1]{
\begin{scope}[xshift=12cm]
\draw[triple={[line width=1.4pt, white] in [line width=2.2pt, black] in [line width=5.4pt, white]}] 
(3,3) to[out=north east, in=west] (5,4) -- (7,4);
\draw[triple={[line width=1.4pt, white] in [line width=2.2pt, black] in [line width=5.4pt, white]}] 
(3,-3) to[out=south east, in=west] (5,-4) -- (7,-4);
\end{scope}
\draw[triple={[line width=1.4pt, white] in [line width=2.2pt, black] in [line width=5.4pt, white]}] 
(-3,3) to[out=north west, in=east] (-5,4) -- (-7,4);
\draw[triple={[line width=1.4pt, white] in [line width=2.2pt, black] in [line width=5.4pt, white]}] 
(-3,-3) to[out=south west, in=east] (-5,-4) -- (-7,-4);
\draw[->-=.3, white, double=black, double distance=0.4pt, ultra thick] 
(-3,-3) to[out=east, in=west] (3,3);
\draw[-<-=.3, white, double=black, double distance=0.4pt, ultra thick] 
(-3,3) to[out=east, in=west] (3,-3);
\begin{scope}[xshift=6cm]
\node at (0,0) {$\dots$};
\node at (0,-5) [below]{$2l$ crossings};
\end{scope}
\begin{scope}[xshift=12cm]
\draw[-<-=.3, white, double=black, double distance=0.4pt, ultra thick] 
(-3,-3) to[out=east, in=west] (3,3);
\draw[->-=.3, white, double=black, double distance=0.4pt, ultra thick] 
(-3,3) to[out=east, in=west] (3,-3);
\end{scope}
\begin{scope}[xshift=12cm]
\draw[->-=.5] (3,-3) -- (3,3);
\draw[fill=white] (3,3) circle [radius=.8];
\draw[fill=white] (3,-3) circle [radius=.8];
\end{scope}
\draw[->-=.5] (-3,3) -- (-3,-3);
\draw[fill=white] (-3,3) circle [radius=.8];
\draw[fill=white] (-3,-3) circle [radius=.8];
\begin{scope}[xshift=12cm]
\node at (3,0)[right]{$\scriptstyle{n}$};
\node at (7,4)[right]{$\scriptstyle{n}$};
\node at (7,-4)[right]{$\scriptstyle{n}$};
\end{scope}
\node at (-7,-4)[left]{$\scriptstyle{n}$};
\node at (-7,4)[left]{$\scriptstyle{n}$};
\node at (-3,0)[left]{$\scriptstyle{n}$};}
\,\Bigg\rangle_{\! 3}\\
&=q^{-2l(n^2+2n)}
\sum_{0\leq k_l\leq \cdots\leq k_1\leq n}
q^{n-k_l}
q^{\sum_{i=1}^{l}(k_i^2+2k_i)}\\
&\qquad\times\frac{(q)_n}{(q)_{k_l}}
{n \choose k_1',k_2',\dots,k_l',k_l}_{q}
\Bigg\langle\,
\tikz[baseline=-.6ex, scale=.1]{
\begin{scope}[xshift=6cm]
\draw[triple={[line width=1.4pt, white] in [line width=2.2pt, black] in [line width=5.4pt, white]}] 
(3,3) to[out=north east, in=west] (5,4) -- (7,4);
\draw[triple={[line width=1.4pt, white] in [line width=2.2pt, black] in [line width=5.4pt, white]}] 
(3,-3) to[out=south east, in=west] (5,-4) -- (7,-4);
\end{scope}
\draw[triple={[line width=1.4pt, white] in [line width=2.2pt, black] in [line width=5.4pt, white]}] 
(-3,3) to[out=north west, in=east] (-5,4) -- (-7,4);
\draw[triple={[line width=1.4pt, white] in [line width=2.2pt, black] in [line width=5.4pt, white]}] 
(-3,-3) to[out=south west, in=east] (-5,-4) -- (-7,-4);
\draw (-3,3) -- (-1,3);
\draw (-3,-3) -- (-1,-3);
\draw (9,3) -- (7,3);
\draw (9,-3) -- (7,-3);
\begin{scope}[xshift=6cm]
\draw[->-=.5] (3,-3) -- (3,3);
\draw[fill=white] (3,3) circle [radius=.8];
\draw[fill=white] (3,-3) circle [radius=.8];
\end{scope}
\draw[->-=.5] (-3,3) -- (-3,-3);
\draw[fill=white] (-3,3) circle [radius=.8];
\draw[fill=white] (-3,-3) circle [radius=.8];
\draw[-<-=.5] (0,4) -- (6,4);
\draw[->-=.5] (0,-4) -- (6,-4);
\draw[-<-=.5] (0,2) to[out=east, in=east] (0,-2);
\draw[->-=.5] (6,2) to[out=west, in=west] (6,-2);
\draw[fill=white] (-1,1) rectangle (0,5);
\draw[fill=white] (-1,-1) rectangle (0,-5);
\draw[fill=white] (7,1) rectangle (6,5);
\draw[fill=white] (7,-1) rectangle (6,-5);
\begin{scope}[xshift=6cm]
\node at (3,0)[right]{$\scriptstyle{n}$};
\node at (7,4)[right]{$\scriptstyle{n}$};
\node at (7,-4)[right]{$\scriptstyle{n}$};
\end{scope}
\node at (-7,-4)[left]{$\scriptstyle{n}$};
\node at (-7,4)[left]{$\scriptstyle{n}$};
\node at (-3,0)[left]{$\scriptstyle{n}$};
\node at (3,4)[above]{$\scriptstyle{k_l}$};
\node at (3,-4)[below]{$\scriptstyle{k_l}$};
}
\,\Bigg\rangle_{\! 3}.
\end{align*}
In the same way, 
we obtain
\begin{align*}
\Bigg\langle\,
\tikz[baseline=-.6ex, scale=.1]{
\draw[triple={[line width=1.4pt, white] in [line width=2.2pt, black] in [line width=5.4pt, white]}] 
(3,3) to[out=north east, in=west] (5,4) -- (7,4);
\draw[triple={[line width=1.4pt, white] in [line width=2.2pt, black] in [line width=5.4pt, white]}] 
(-3,3) to[out=north west, in=east] (-5,4) -- (-7,4);
\draw[triple={[line width=1.4pt, white] in [line width=2.2pt, black] in [line width=5.4pt, white]}] 
(3,-3) to[out=south east, in=west] (5,-4) -- (7,-4);
\draw[triple={[line width=1.4pt, white] in [line width=2.2pt, black] in [line width=5.4pt, white]}] 
(-3,-3) to[out=south west, in=east] (-5,-4) -- (-7,-4);
\draw[triple={[line width=1.4pt, white] in [line width=2.2pt, black] in [line width=5.4pt, white]}] 
(7,-4) to[out=east, in=west] (13,4);
\draw[triple={[line width=1.4pt, white] in [line width=2.2pt, black] in [line width=5.4pt, white]}] 
(7,4) to[out=east, in=west] (13,-4);
\begin{scope}[xshift=6cm]
\node at (10,0) {$\dots$};
\node at (10,-5) [below]{$2l$ crossings};
\end{scope}
\begin{scope}[xshift=12cm]
\draw[triple={[line width=1.4pt, white] in [line width=2.2pt, black] in [line width=5.4pt, white]}] 
(7,-4) to[out=east, in=west] (13,4);
\draw[triple={[line width=1.4pt, white] in [line width=2.2pt, black] in [line width=5.4pt, white]}] 
(7,4) to[out=east, in=west] (13,-4);
\node at (13,-4)[right]{$\scriptstyle{n}$};
\node at (13,4)[right]{$\scriptstyle{n}$};
\end{scope}
\draw[-<-=.5] (3,3) -- (-3,3);
\draw[-<-=.5] (-3,3) -- (-3,-3);
\draw[-<-=.5] (-3,-3) -- (3,-3);
\draw[-<-=.5] (3,-3) -- (3,3);
\draw[fill=black] (3,3) circle [radius=.8];
\draw[fill=black] (-3,3) circle [radius=.8];
\draw[fill=black] (3,-3) circle [radius=.8];
\draw[fill=black] (-3,-3) circle [radius=.8];
\node at (-7,-4)[left]{$\scriptstyle{n}$};
\node at (-7,4)[left]{$\scriptstyle{n}$};
\node at (0,3)[above]{$\scriptstyle{n}$};
\node at (0,-3)[below]{$\scriptstyle{n}$};
\node at (3,0)[right]{$\scriptstyle{n}$};
\node at (-3,0)[left]{$\scriptstyle{n}$};}
\,\Bigg\rangle_{\! 3}
&=q^{-2l(n^2+2n)}
\sum_{0\leq k_l\leq \cdots\leq k_1\leq n}
q^{n-k_l}
q^{\sum_{i=1}^{l}(k_i^2+2k_i)}\\
&\qquad\times\frac{(q)_n}{(q)_{k_l}}
{n \choose k_1',k_2',\dots,k_l',k_l}_{q}
\Bigg\langle\,
\tikz[baseline=-.6ex, scale=.1]{
\begin{scope}[xshift=6cm]
\draw[triple={[line width=1.4pt, white] in [line width=2.2pt, black] in [line width=5.4pt, white]}] 
(3,3) to[out=north east, in=west] (5,4) -- (7,4);
\draw[triple={[line width=1.4pt, white] in [line width=2.2pt, black] in [line width=5.4pt, white]}] 
(3,-3) to[out=south east, in=west] (5,-4) -- (7,-4);
\end{scope}
\draw[triple={[line width=1.4pt, white] in [line width=2.2pt, black] in [line width=5.4pt, white]}] 
(-3,3) to[out=north west, in=east] (-5,4) -- (-7,4);
\draw[triple={[line width=1.4pt, white] in [line width=2.2pt, black] in [line width=5.4pt, white]}] 
(-3,-3) to[out=south west, in=east] (-5,-4) -- (-7,-4);
\draw (-3,3) -- (-1,3);
\draw (-3,-3) -- (-1,-3);
\draw (9,3) -- (7,3);
\draw (9,-3) -- (7,-3);
\begin{scope}[xshift=6cm]
\draw[-<-=.5] (3,-3) -- (3,3);
\draw[fill=black] (3,3) circle [radius=.8];
\draw[fill=black] (3,-3) circle [radius=.8];
\end{scope}
\draw[-<-=.5] (-3,3) -- (-3,-3);
\draw[fill=black] (-3,3) circle [radius=.8];
\draw[fill=black] (-3,-3) circle [radius=.8];
\draw[->-=.5] (0,4) -- (6,4);
\draw[-<-=.5] (0,-4) -- (6,-4);
\draw[->-=.5] (0,2) to[out=east, in=east] (0,-2);
\draw[-<-=.5] (6,2) to[out=west, in=west] (6,-2);
\draw[fill=white] (-1,1) rectangle (0,5);
\draw[fill=white] (-1,-1) rectangle (0,-5);
\draw[fill=white] (7,1) rectangle (6,5);
\draw[fill=white] (7,-1) rectangle (6,-5);
\begin{scope}[xshift=6cm]
\node at (3,0)[right]{$\scriptstyle{n}$};
\node at (7,4)[right]{$\scriptstyle{n}$};
\node at (7,-4)[right]{$\scriptstyle{n}$};
\end{scope}
\node at (-7,-4)[left]{$\scriptstyle{n}$};
\node at (-7,4)[left]{$\scriptstyle{n}$};
\node at (-3,0)[left]{$\scriptstyle{n}$};
\node at (3,4)[above]{$\scriptstyle{k_l}$};
\node at (3,-4)[below]{$\scriptstyle{k_l}$};
}
\,\Bigg\rangle_{\! 3}.
\end{align*}
The closure of $A_2$ webs appearing on the right-hand side of the above two equations are the same planar web with the opposite orientation each other by Lemma~\ref{change}. 
Therefore, 
these webs have the same value.
We compute the value of the closure. 
For $0\leq k_l\leq n$,
\begin{align*}
\Bigg\langle
\tikz[baseline=-.6ex, scale=.1]{
\draw[triple={[line width=1.4pt, white] in [line width=2.2pt, black] in [line width=5.4pt, white]}] 
(-7,4) to[out=west, in=west](-7,8);
\draw[triple={[line width=1.4pt, white] in [line width=2.2pt, black] in [line width=5.4pt, white]}] 
(-7,-4) to[out=west, in=west](-7,-8);
\draw[triple={[line width=1.4pt, white] in [line width=2.2pt, black] in [line width=5.4pt, white]}] 
(13,4) to[out=east, in=east](13,8);
\draw[triple={[line width=1.4pt, white] in [line width=2.2pt, black] in [line width=5.4pt, white]}] 
(13,-4) to[out=east, in=east](13,-8);
\draw[triple={[line width=1.4pt, white] in [line width=2.2pt, black] in [line width=5.4pt, white]}] 
(-7,8) -- (13,8);
\draw[triple={[line width=1.4pt, white] in [line width=2.2pt, black] in [line width=5.4pt, white]}] 
(-7,-8) -- (13,-8);
\begin{scope}[xshift=6cm]
\draw[triple={[line width=1.4pt, white] in [line width=2.2pt, black] in [line width=5.4pt, white]}] 
(3,3) to[out=north east, in=west] (5,4) -- (7,4);
\draw[triple={[line width=1.4pt, white] in [line width=2.2pt, black] in [line width=5.4pt, white]}] 
(3,-3) to[out=south east, in=west] (5,-4) -- (7,-4);
\end{scope}
\draw[triple={[line width=1.4pt, white] in [line width=2.2pt, black] in [line width=5.4pt, white]}] 
(-3,3) to[out=north west, in=east] (-5,4) -- (-7,4);
\draw[triple={[line width=1.4pt, white] in [line width=2.2pt, black] in [line width=5.4pt, white]}] 
(-3,-3) to[out=south west, in=east] (-5,-4) -- (-7,-4);
\draw (-3,3) -- (-1,3);
\draw (-3,-3) -- (-1,-3);
\draw (9,3) -- (7,3);
\draw (9,-3) -- (7,-3);
\begin{scope}[xshift=6cm]
\draw[->-=.5] (3,-3) -- (3,3);
\draw[fill=white] (3,3) circle [radius=.8];
\draw[fill=white] (3,-3) circle [radius=.8];
\end{scope}
\draw[->-=.5] (-3,3) -- (-3,-3);
\draw[fill=white] (-3,3) circle [radius=.8];
\draw[fill=white] (-3,-3) circle [radius=.8];
\draw[-<-=.5] (0,4) -- (6,4);
\draw[->-=.5] (0,-4) -- (6,-4);
\draw[-<-=.5] (0,2) to[out=east, in=east] (0,-2);
\draw[->-=.5] (6,2) to[out=west, in=west] (6,-2);
\draw[fill=white] (-1,1) rectangle (0,5);
\draw[fill=white] (-1,-1) rectangle (0,-5);
\draw[fill=white] (7,1) rectangle (6,5);
\draw[fill=white] (7,-1) rectangle (6,-5);
\begin{scope}[xshift=6cm]
\node at (3,0)[right]{$\scriptstyle{n}$};
\node at (7,4)[right]{$\scriptstyle{n}$};
\node at (7,-4)[right]{$\scriptstyle{n}$};
\end{scope}
\node at (-7,-4)[left]{$\scriptstyle{n}$};
\node at (-7,4)[left]{$\scriptstyle{n}$};
\node at (-3,0)[left]{$\scriptstyle{n}$};
\node at (3,4)[above]{$\scriptstyle{k_l}$};
\node at (3,-4)[below]{$\scriptstyle{k_l}$};
}
\,\Bigg\rangle_{\! 3}
&=
\Bigg\langle
\tikz[baseline=-.6ex, scale=.1]{
\draw[triple={[line width=1.4pt, white] in [line width=2.2pt, black] in [line width=5.4pt, white]}] 
(-7,4) to[out=west, in=west](-7,8);
\draw[triple={[line width=1.4pt, white] in [line width=2.2pt, black] in [line width=5.4pt, white]}] 
(-7,-4) to[out=west, in=west](-7,-8);
\draw[triple={[line width=1.4pt, white] in [line width=2.2pt, black] in [line width=5.4pt, white]}] 
(13,4) to[out=east, in=east](13,8);
\draw[triple={[line width=1.4pt, white] in [line width=2.2pt, black] in [line width=5.4pt, white]}] 
(13,-4) to[out=east, in=east](13,-8);
\draw[triple={[line width=1.4pt, white] in [line width=2.2pt, black] in [line width=5.4pt, white]}] 
(-7,8) -- (13,8);
\draw[triple={[line width=1.4pt, white] in [line width=2.2pt, black] in [line width=5.4pt, white]}] 
(-7,-8) -- (13,-8);
\draw[-<-=.5] (-6,5.5) -- (12,5.5);
\draw[->-=.5] (-6,-5.5) -- (12,-5.5);
\draw[-<-=.5] (-6,4.5) to[out=east, in=north] (-2,0)
to[out=south, in=east] (-6,-4.5);
\draw[->-=.5] (12,4.5) to[out=west, in=north] (8,0)
to[out=south, in=west] (12,-4.5);
\draw[->-=.5] (-6,3) to[out=east, in=east] (-6,-3);
\draw[-<-=.5] (12,3) to[out=west, in=west] (12,-3);
\draw[fill=white] (-7,2) rectangle (-6,6);
\draw[fill=white] (-7,-2) rectangle (-6,-6);
\draw[fill=white] (13,2) rectangle (12,6);
\draw[fill=white] (13,-2) rectangle (12,-6);
\draw (-7,4) -- (-6,4);
\draw (-7,-4) -- (-6,-4);
\draw (13,4) -- (12,4);
\draw (13,-4) -- (12,-4);
\begin{scope}[xshift=6cm]
\node at (7,4)[right]{$\scriptstyle{n}$};
\node at (7,-4)[right]{$\scriptstyle{n}$};
\end{scope}
\node at (-6,0) {$\scriptstyle{n}$};
\node at (12,0) {$\scriptstyle{n}$};
\node at (-7,-4)[left]{$\scriptstyle{n}$};
\node at (-7,4)[left]{$\scriptstyle{n}$};
\node at (-2.5,1)[right]{$\scriptstyle{n-k_l}$};
\node at (8.5,-1)[left]{$\scriptstyle{n-k_l}$};
\node at (3,6)[below]{$\scriptstyle{k_l}$};
\node at (3,-6)[above]{$\scriptstyle{k_l}$};
}
\,\Bigg\rangle_{\! 3}
=
\Bigg\langle
\tikz[baseline=-.6ex, scale=.1]{
\draw[->-=.5, rounded corners] (0,7) -- (-6,7) -- (-6,-7) -- (0,-7);
\draw[-<-=.5, rounded corners] (0,7) -- (6,7) -- (6,-7) -- (0,-7);
\draw[-<-=.5, rounded corners] (0,4) -- (-4,4) -- (-4,-4) -- (0,-4);
\draw[->-=.5, rounded corners] (0,4) -- (4,4) -- (4,-4) -- (0,-4);
\draw[-<-=.5, rounded corners] (0,3) -- (-2,3) -- (-2,1) -- (2,1) -- (2,3) -- (0,3);
\draw[->-=.5, rounded corners] (0,-3) -- (-2,-3) -- (-2,-1) -- (2,-1) -- (2,-3) -- (0,-3);
\draw[fill=white] (-.5,2) rectangle (.5,8);
\draw[fill=white] (-.5,-2) rectangle (.5,-8);
\draw (-.5,5) -- (.5,5);
\draw (-.5,-5) -- (.5,-5);
\node at (-6,0) [left]{$\scriptstyle{n}$};
\node at (6,0) [right]{$\scriptstyle{n}$};
\node at (-2,.5){$\scriptstyle{k_l}$};
\node at (2,-.5){$\scriptstyle{k_l}$};
}
\,\Bigg\rangle_{\! 3}\\
&=\sum_{s=0}^n(-1)^s\frac{{n \brack s}^2}{{2n+1 \brack s}}
\Bigg\langle
\tikz[baseline=-.6ex, scale=.1]{
\draw[->-=.5, rounded corners] (-6,9) -- (-8,9) -- (-8,-7) -- (0,-7);
\draw[-<-=.5, rounded corners] (6,9) -- (8,9) -- (8,-7) -- (0,-7);
\draw[-<-=.5, rounded corners] (-4,4) -- (-4,-4) -- (0,-4);
\draw[->-=.5, rounded corners] (4,4) -- (4,-4) -- (0,-4);
\draw[->-=.5, rounded corners] (-2,4) -- (-2,5) -- (2,5) -- (2,4);
\draw[-<-=.5, rounded corners] (-2,3) -- (-2,2) -- (2,2) -- (2,3);
\draw[-<-=.5] (-5,10) -- (5,10);
\draw[rounded corners] (-5,8) -- (-4,8) -- (-4,4);
\draw[rounded corners] (5,8) -- (4,8) -- (4,4);
\draw[rounded corners] (0,-3) -- (-2,-3) -- (-2,-1) -- (2,-1) -- (2,-3) -- (0,-3);
\draw[fill=white] (-6,7) rectangle (-5,11);
\draw[fill=white] (6,7) rectangle (5,11);
\draw[fill=white] (-5,3) rectangle (-1,4);
\draw[fill=white] (5,3) rectangle (1,4);
\draw[fill=white] (-.5,-2) rectangle (.5,-8);
\draw (-.5,-5) -- (.5,-5);
\node at (0,9) {$\scriptstyle{n-s}$};
\node at (0,4) [above]{$\scriptstyle{n-s}$};
\node at (-8,0) [left]{$\scriptstyle{n}$};
\node at (8,0) [right]{$\scriptstyle{n}$};
\node at (-2,1){$\scriptstyle{k_l}$};
\node at (2,-.5){$\scriptstyle{k_l}$};
}
\,\Bigg\rangle_{\! 3}\\
&=\sum_{s=0}^n(-1)^s\frac{{n \brack s}^2}{{2n+1 \brack s}}
\left(\sum_{t=\max\{n-s,k_l\}}^{\min\{n-s+k_l,n\}}\frac{{n\brack t}^2{t\brack k_l}{t\brack n-s}{2n-t+2\brack k_l+n-s-t}}{{n\brack k_l}^2{n\brack n-s}^2}
\Bigg\langle
\tikz[baseline=-.6ex, scale=.1]{
\draw[->-=.5, rounded corners] (-6,6) -- (-8,6) -- (-8,-6) -- (0,-6);
\draw[-<-=.5, rounded corners] (6,6) -- (8,6) -- (8,-6) -- (0,-6);
\draw (-5,7) -- (5,7);
\draw[-<-=.5, rounded corners] (-5,5) -- (-4,5) -- (-4,1) -- (-4,-3) -- (0,-3);
\draw[->-=.5, rounded corners] (5,5) -- (4,5) -- (4,1) -- (4,-3) -- (0,-3);
\draw[->-=.5, rounded corners] (0,-2) -- (-2,-2) -- (-2,0) -- (2,0) -- (2,-2) -- (0,-2);
\draw[fill=white] (-6,4) rectangle (-5,8);
\draw[fill=white] (6,4) rectangle (5,8);
\draw[fill=white] (-.5,-1) rectangle (.5,-7);
\draw (-.5,-4) -- (.5,-4);
\node at (0,6) {$\scriptstyle{t}$};
\node at (-8,0) [left]{$\scriptstyle{n}$};
\node at (8,0) [right]{$\scriptstyle{n}$};
\node at (0,2) {$\scriptstyle{t}$};
}
\,\Bigg\rangle_{\! 3}\right)\\
&=\sum_{n-k_l\leq n-s\leq n}(-1)^s\frac{{n \brack s}{n+2\brack k_l-s}}{{2n+1 \brack s}{n\brack k_l}}
\Big\langle\,
\tikz[baseline=-.6ex, scale=0.1]{
\draw[-<-=.5] (0,0) circle [radius=2];
\draw[->-=.5] (0,0) circle [radius=4];
\draw[fill=white] (1,-.5) rectangle (5,.5);
\draw (3,-.5) -- (3,.5);
\node at (0,0){$\scriptstyle{n}$};
\node at (0,5){$\scriptstyle{n}$};
}
\,\Big\rangle_{\! 3}.
\end{align*}
We used Definition~\ref{doubleA2clasp} in the third equation, 
Theorem~\ref{A2bubble} in the fourth equation, and Lemma~\ref{doubleA2claspprop} in the last equation. 
We remark that 
$\Big\langle\,
\tikz[baseline=-.6ex, scale=0.1]{
\draw[-<-=.5] (0,0) circle [radius=2];
\draw[->-=.5] (0,0) circle [radius=4];
\draw[fill=white] (1,-.5) rectangle (5,.5);
\draw (3,-.5) -- (3,.5);
\node at (0,0){$\scriptstyle{n}$};
\node at (0,5){$\scriptstyle{n}$};
}
\,\Big\rangle_{\! 3}
=\frac{\left[n+1\right]^2\left[2n+2\right]}{\left[2\right]}$ 
(see, for example, Lemma~5.6 in \cite{OhtsukiYamada97}).
Consequently, 
\begin{align*}
\left[\bar{\operatorname{ST}}(1,2l)\right]_{(n,n)}
&=2q^{-2l(n^2+2n)}
\sum_{0\leq k_l\leq \cdots\leq k_1\leq n}
q^{n-k_l}q^{\sum_{i=1}^{l}(k_i^2+2k_i)}\frac{(q)_n}{(q)_{k_l}}
{n \choose k_1',k_2',\dots,k_l',k_l}_{q}\\
&\qquad\times
\left(
\sum_{0\leq s\leq k_l}(-1)^s\frac{{n \brack s}{n+2\brack k_l-s}}{{2n+1 \brack s}{n\brack k_l}}
\Big\langle\,
\tikz[baseline=-.6ex, scale=0.1]{
\draw[-<-=.5] (0,0) circle [radius=2];
\draw[->-=.5] (0,0) circle [radius=4];
\draw[fill=white] (1,-.5) rectangle (5,.5);
\draw (3,-.5) -- (3,.5);
\node at (0,0){$\scriptstyle{n}$};
\node at (0,5){$\scriptstyle{n}$};
}
\,\Big\rangle_{\! 3}
\right)\\
&=2q^{-2l(n^2+2n)}
\sum_{0\leq s\leq k_l\leq \cdots\leq k_1\leq n}
(-1)^sq^{n-k_l}q^{\sum_{i=1}^{l}(k_i^2+2k_i)}q^{s(n-k_l)-k_l}q^{\frac{s^2+3s}{2}}\\
&\qquad\times
\frac{(q)_n}{(q)_{k_l}}{n \choose k_1',k_2',\dots,k_l',k_l}_{q}
\frac{{n \choose s}_q{n+2\choose k_l-s}_q}{{2n+1 \choose s}_q{n\choose k_l}_q}
\Big\langle\,
\tikz[baseline=-.6ex, scale=0.1]{
\draw[-<-=.5] (0,0) circle [radius=2];
\draw[->-=.5] (0,0) circle [radius=4];
\draw[fill=white] (1,-.5) rectangle (5,.5);
\draw (3,-.5) -- (3,.5);
\node at (0,0){$\scriptstyle{n}$};
\node at (0,5){$\scriptstyle{n}$};
}
\,\Big\rangle_{\! 3}.
\end{align*}

\subsection*{Acknowledgment}
The author would like to express his gratitude to his adviser, Hisaaki Endo, for his encouragement.

\bibliographystyle{amsalpha}
\bibliography{KV}
\end{document}